\documentclass[12pt]{amsart}
\usepackage[english]{babel}
\usepackage[T1]{fontenc}
\usepackage[utf8]{inputenc}

\usepackage{amsfonts,amsmath,amssymb,amscd,amsthm,url}
\usepackage[inline]{enumitem}
\usepackage{graphicx,epsf,afterpage}
\usepackage{multicol}
\usepackage{array}
\usepackage{braket}
\usepackage{epigraph}
\usepackage{rotating}

\usepackage{xcolor}

\usepackage{ulem}

\usepackage{tikz}
\usetikzlibrary{positioning, calc, intersections, through, arrows}
\usetikzlibrary{graphs}
\usetikzlibrary{graphs.standard}
\usetikzlibrary{patterns}

\theoremstyle{plain}
    \newtheorem{theorem}{Theorem}[section]
    \newtheorem{lemma}[theorem]{Lemma}
    \newtheorem{proposition}[theorem]{Proposition}
    
    \newtheorem{conjecture}[theorem]{Conjecture}

\theoremstyle{definition}
    \newtheorem{remark}[theorem]{Remark}
    \newtheorem{example}[theorem]{Example}

\usepackage[
    draft = false,
    unicode = true,
    colorlinks = true,
    allcolors = blue,
    hyperfootnotes = true,
    citecolor = red
]{hyperref}

\newcommand{\aronly}[1]{#1}
\newcommand{\jonly}[1]{}

\newcommand\abs[1]{\ensuremath{\left\lvert#1\right\rvert}}
\newcommand\babs[1]{\ensuremath{\bigl\lvert#1\bigr\rvert}}
\DeclareMathOperator{\cl}{cl}
\DeclareMathOperator{\rk}{rk}
\newcommand{\R}{\ensuremath{\mathbb{R}}}
\newcommand{\Z}{\ensuremath{\mathbb{Z}}}

\newcommand{\Q}{\ensuremath{\mathbb{Q}}}

\renewcommand{\ge}{\geqslant}
\renewcommand{\geq}{\geqslant}
\renewcommand{\le}{\leqslant}

\renewcommand{\emptyset}{\varnothing}
\def\t{\widetilde}
\def\Int{\mathop{\fam0 Int}}

\newcommand{\setenumi}{
    \topsep=0mm
    \partopsep=0mm
    \parsep=0mm
    \itemsep=0mm
    \leftmargin=0pt
    \listparindent=7mm
    \itemindent=3mm
    \labelsep=2mm
    \labelwidth=-5mm
    \usecounter{enumi}
    }

\newenvironment{remarkenumi}
    {\begin{list}{\labelenumi}{\setenumi}}
    {\end{list}}

\renewcommand{\theenumi}{\alph{enumi}}
\renewcommand{\labelenumi}{(\alph{enumi})}

\newcounter{mcnt}
\renewcommand{\themcnt}{(\alph{mcnt})}
\newcommand{\cnt}{\addtocounter{mcnt}{1}\themcnt\ \nopagebreak}
\newcommand{\juststep}{\addtocounter{mcnt}{1}\themcnt}

\begin{document}

\title{A quadratic estimation \\ for the K\"uhnel conjecture on embeddings}

\author{Slava Dzhenzher and Arkadiy Skopenkov}

\date{}

\begin{abstract}
    The classical Heawood inequality states that if the complete graph $K_n$ on $n$ vertices is embeddable into the sphere with $g$ handles, then
    $g\geqslant\dfrac{(n-3)(n-4)}{12}$. 
    A higher-dimensional analogue of the Heawood inequality is the K\"uhnel conjecture.
    In a simplified form, it states that \emph{for every integer $k>0$ there is $c_k>0$ such that if the union of $k$-faces of $n$-simplex embeds into the connected sum of $g$ copies of the Cartesian product $S^k\times S^k$ of two $k$-dimensional spheres, then $g\geqslant c_k n^{k+1}$}.
    For $k>1$ only linear estimates were known. 
    We present a quadratic estimate $g\geqslant c_kn^2$. 
    The proof is based on the beautiful and fruitful interplay between geometric topology, combinatorics, and linear algebra. 
\end{abstract}

\thanks{Both authors: Moscow Institute of Physics and Technology. 
\emph{A. Skopenkov:} Independent University of Moscow; corresponding author; \texttt{askopenkov@gmail.com}
\newline
We are grateful to R.~Fulek, T. Garaev, R.~Karasev, E.~Kogan, S.~Melikhov, R.~Nikkuni, S.~Zhilina, and the anonymous referees for useful discussions.}

\maketitle



\tableofcontents

\section{Introduction and main result}\label{s:intr}

\subsection*{Main result}

The classical Heawood inequality states that if the complete graph $K_n$ on~$n$ vertices is embeddable into the sphere with $g$ handles, then
\[
    g \geqslant \frac{(n-3)(n-4)}{12}.
\]
Denote by

\begin{itemize}[nosep]
    \item $\Delta^k_n$ the union of the $k$-dimensional faces of an $n$-dimensional simplex; 

    \item $S_g^{2k}$ the connected sum of $g$ copies of the Cartesian product $S^k\times S^k$ of two $k$-dimensional spheres.
\end{itemize}

A higher-dimensional analogue of the Heawood inequality is the K\"uhnel conjecture \ref{c:kuh} on embeddings.
In a simplified form, it states that \emph{for every integer $k>0$ there is $c_k>0$ such that if $\Delta^k_n$ embeds into $S_g^{2k}$, then}
\[
    g \geqslant c_k n^{k+1}.
\]
We present a quadratic in $n$ estimate (Theorems~\ref{t:quadr:complex} and~\ref{t:quadr:join}). 
 
\smallskip 
\emph{Notation and conventions.}
From now on we shorten `$s$-dimensional' to just `$s$-'.

In this text, a manifold may have a non-empty boundary.
 
For a simple definition of a homology group $H_k(\cdot;\Z_2)$ accessible to non-specialists in topology see \cite{HG}, \cite[\S6, \S10]{Sk20}.
For a $2k$-manifold $M$ let 
\[
    \beta_k(M):=\dim H_k(M;\Z_2).
\]
(This is called the \emph{$k$-th mod~$2$ Betti number} of $M$. 
Observe that $\beta_1(S_g^2)=2g$.) 

In this text remarks and references to them could be ignored by some readers, but could be important for others; it is for a reader to choose whether to read a remark.

We consider only compact piecewise linear (PL) $2k$-manifolds \cite{RS72}. 
Unless otherwise specified, we consider only PL maps. 
Thus we omit `PL' from statements, definitions and proofs (except when topological embeddings are around).
The analogues of our results are correct for topological embeddings (Remark~\ref{r:hist:complex}.\ref{en:r-histc-top}), for almost embeddings and for $\Z_2$-embeddings (defined and discussed in Remarks~\ref{r:alem} and~\ref{r:emb-mod-2}). 

\begin{theorem}[Skeleton]
\label{t:quadr:complex}
    If $\Delta^k_n$ embeds into a $2k$-manifold~$M$, then 
    \[
        \beta_k(M)\gtrsim\frac{n^2}{2^k(k+1)^2}\quad\text{as}\quad n\to \infty\quad\text{for fixed}\quad k \ge 1
    \]
    (more precisely, $\geqslant \frac{(n-4k-2)^2}{2^k(k+1)^2}$ for $n\geqslant5k+3$).
\end{theorem}

Denote $[n] := \{1,\ldots,n\}$.

Let $[n]^{*k+1}$ be the $k$-complex with vertex set $[k+1]\times[n]$, in which every $k+1$ vertices from different rows span a $k$-face.   
For $k=1$ this is the complete bipartite graph $K_{n,n}$.
For a geometric interpretation see \cite[Proposition~4.2.4]{Ma03}.

\begin{theorem}[Joinpower]
\label{t:quadr:join}
    If $[n]^{*k+1}$ embeds into a $2k$-manifold~$M$, then 
    \[
        \beta_k(M)\gtrsim\frac{n^2}{2^k} \quad\text{as}\quad n\to\infty
        \quad\text{for fixed}\quad k \ge 1
    \]
    (more precisely, $\geqslant \frac{(n-3)^2}{2^k}$ for $n\geqslant4$).
\end{theorem}
 
Theorem~\ref{t:quadr:complex} (Skeleton) follows from Theorem~\ref{t:quadr:join} (Joinpower) because\footnote{Observe that vice versa $K := \bigl[n\bigr]*\bigl[\binom{n}{2}\bigr]*\ldots*\bigl[\binom{n}{k+1}\bigr]$ contains a subdivision of $\Delta^k_n$ (here $K$ is a complex with set $1 \times \bigl[n\bigr] \sqcup 2 \times \bigl[\binom{n}{2}\bigr] \sqcup \ldots \sqcup (k+1) \times \bigl[\binom{n}{k+1}\bigr]$ of vertices; its $k$-faces are $\{(i,a_i)\}_{i\in[k+1]}$ for $a_i\in \bigl[\binom{n}{i}\bigr]$).
In order to prove this take a baricentric subdivision of $\Delta_n^k$.
Then every vertex $(i,a_i)$ of $K$ corresponds to the barycenter of some $i$-face of $\Delta_n^k$.}
$\Delta^k_n\supset [s]^{*k+1}$ for some $s\geqslant\dfrac{n-k+1}{k+1}$.
Theorem~\ref{t:quadr:join} in turn follows from Theorems~\ref{t:embed}, \ref{t:lowrank} below.

For a $2k$-manifold $M$ into which $\Delta_n^k$ can be embedded, the linear in $n$ estimate $\beta_k(M)\geqslant\dfrac{n-2k-1}{k+1}$ is proved in \cite{PT19}\footnote{See \cite{KS21} for a simpler \cite[Remark~1.2.e]{KS21} exposition of this estimation. 
Observe that the paper \cite{PT19} contains another result mentioned in Remark \ref{r:hist:complex}.\ref{general}.} 
after a weaker linear estimate of \cite{GMP+}.  
So even \emph{linear} estimates appeared more than 20 years after the K\"uhnel conjecture \ref{c:kuh} on embeddings was put forward, and are published in good journals.
For 6 years after the linear estimate in \cite{GMP+}, and before the current paper, no stronger estimates appeared \cite[\S4]{Ku23}. 
 
There is another interesting approach to the K\"uhnel conjecture \ref{c:kuh} on embeddings: the papers \cite{Ad18, AP24} jointly claim it (without explicitly stating this in \cite{AP24}). 
In \aronly{\S\ref{s:adipa}}\jonly{\cite[\S6]{DS22}} we justify that the realization of this approach in \cite{Ad18, AP24} is not reliable up to the standards of refereed journals, by  specific critical \aronly{Remarks \ref{r:ggks}, \ref{r:helpadi} and \ref{r:histc-adpa}}\jonly{remarks \cite[Remarks~6.2, 6.5, 6.7]{DS22}}\footnote{We welcome public statements of different reliability standards \cite{Sk21d}. 
So far none have appeared, see Remark~\ref{r:ap-let}. 
\newline
Because of a referee report we received, we have to recall that arXiv version 2 of the current paper (where the proof is complete, or at least not criticized) appeared on September~1, 2022, earlier than \cite{AP24}.}. 

For $k=1$ Theorem~\ref{t:quadr:complex} (Skeleton) (and the linear estimates above) follows from the Heawood inequality at the beginning of \S\ref{s:intr}, and Theorem~\ref{t:quadr:join} (Joinpower) is also due to Heawood.
The usual proof of the Heawood inequality via Euler inequality does not work for $k>1$ because a $k$-sphere does not split $\R^{2k}$.  

\begin{conjecture}[The K\"uhnel conjecture on embeddings]\label{c:kuh} 
\cite[Conjecture~B]{Ku94} \cite[Conjecture~22]{Ku23}.
If $\Delta_n^k$ embeds into a $(k-1)$-connected closed $2k$-manifold $M$, then 
$$\displaystyle \binom{2k+1}{k+1} \babs{\chi(M)-2} \geqslant \binom{n-k-1}{k+1}.$$ 
\end{conjecture}

\begin{remark}\label{r:hist:complex}
    \begin{remarkenumi}
        \item\label{en:r-hist-beta} Different authors have considered stronger conjectures, in which $M$ is not $(k-1)$-connected, and $\babs{\chi(M)-2}$ is replaced either by $\beta_k(M)$ \cite{GMP+, PT19}, or by the \emph{$k$-th rational/integer Betti number}  
    $$b_k(M):=\dim H_k(M;\Q)=\rk H_k(M;\Z)$$ 
    (see \aronly{Conjecture~\ref{c:kuh-simp} and Remark~\ref{r:helpadi}}\jonly{\cite[Conjecture~6.3 and Remark~6.5]{DS22}}).
    We have 

    \begin{itemize}[nosep, leftmargin=10mm]
        \item $\beta_1(M)=b_1(M)=2-\chi(M)$ for a closed connected $2$-manifold $M$,

        \item 
        $\beta_k(M)=b_k(M)=\babs{\chi(M)-2}$ for a closed $(k-1)$-connected $2k$-manifold $M$, and  

        \item $\beta_k(M)\ge b_k(M)$ by the Universal Coefficients Formula, see e.g. \cite[\S15.5]{FF89} \cite[Theorem~11.8.1]{Sk20}.
    \end{itemize}
     
    \item\label{en:r-histc-top} 
    In the K\"uhnel conjecture on embeddings the PL and the topological embeddability are equivalent for $k\ge3$ by Remark~\ref{r:alem} (or by the PL approximation theorem \cite[Theorem~1]{Br72}; in fact, weaker `metastable' versions of this result, which are cited in \cite{Br72} are sufficient; recall that we consider topological embeddings into PL manifolds).
    See also the K\"uhnel conjecture \aronly{\ref{c:kuh-simp}}\jonly{\cite[Conjecture~6.3]{DS22}} for simplicial embeddings and \aronly{Remark~\ref{r:helpadi}}\jonly{\cite[Remark~6.5]{DS22}}. 

       \item\label{general} 
        For a short description of references on embeddability of general $k$-complexes into $2k$-manifolds see \cite[Remark 1.1.7.bc]{Sk24}. 
        There are algebraic criteria for such embeddability,  
        due to Harris-Krushkal-Johnson-Pat\'ak-Tancer-Skopenkov, see \cite[\S1.1, \S1.3]{Sk24} and the references therein. 
        Theorems~\ref{t:quadr:complex} and~\ref{t:quadr:join} are non-trivial in spite of the existence of these criteria and the last paragraph of Remark \ref{r:idea}.\ref{en:r-idea-nontriv}. 
        
        The criteria from \cite{FK19}, \cite[\S1.1]{Sk24} show that such embeddability is closely related to the low rank matrix completion problem (and thus to the Netflix problem from machine learning), see the references in \cite[beginning of \S1.1]{Sk24}. 
        This is the problem of minimizing the rank of a matrix, of which entries some are fixed, and the others can be changed (see~\cite{DGN+} for references and an introduction accessible to students).
        Our proof of Theorem~\ref{t:quadr:join} (Joinpower) is also related to this problem. 
        We study a more general problem, in which instead of knowing specific matrix elements, we know linear relations on such elements.
        We estimate the minimal rank of matrices with such relations (Theorem~\ref{t:lowrank}).

        \item\label{en:r-hist-asym}
        (The asymptotic version of the K\"uhnel conjecture on embeddings.)
        \emph{Under the conditions of Conjecture~\ref{c:kuh} for any $k\ge1$ there is $c_k>0$ such that}
        \[
            b_k(M) \gtrsim c_k n^{k+1} \quad\textit{as}\quad n\to\infty. 
        \]
        Note that \(\beta_k(M) \geq b_k(M)\), see~(\ref{en:r-hist-beta}).

       \item\label{en:r-histc-wrong} 
We believe that even the asymptotic version~(\ref{en:r-hist-asym})
is wrong or hard to prove. 
Indeed, the integer version of the obstruction to embeddability constructed in this paper (Theorem~\ref{t:basis-emb}) is presumably complete for $k\ge3$ \cite[Conjecture~1.7.b]{SS23} (the obstruction is complete for $\Z_2$-embeddability, see Remark \ref{r:emb-mod-2}.\ref{en:r-emb-mod-2-def} and \cite[Theorem~1.5]{SS23}).

\aronly{\item For another K\"uhnel conjectures of 1994 see \cite[Conjecture~C]{Ku94}, \cite[Conjecture~B]{Ku95}, \cite{Ku23}. 

\item Shortened versions of this paper were rejected from Intern. Math. Res. Notices upon incompetent reports, see \cite[Examples 6.7 and 6.8]{Sk21d}. 
This is some evidence that the results of this paper are interesting and reliable.} 
\end{remarkenumi}
\end{remark}


\subsection*{Topological and linear algebraic parts}
 
Our theoretical achievement allowing us to prove Theorems \ref{t:quadr:complex} and \ref{t:quadr:join} is to fit what we can prove in topology to what is sufficient for algebra. 
Thus our main idea is the notion of an $(n,k)$-matrix, whose definition is postponed until after Remark~\ref{rem:inters:2man}. 
Before we introduce the definition, we show how it works.
Theorem~\ref{t:quadr:join} (Joinpower) is implied by the following Theorems~\ref{t:embed} (Embeddability) and~\ref{t:lowrank} (Low Rank).
Thus the proof is split into two independent parts. 

\begin{theorem}[Embeddability; proved in \S\ref{s:intr} below]
\label{t:embed}
    If $[n]^{*k+1}$ embeds into a $2k$-manifold $M$, then there is an $(n,k)$-matrix of rank at most $\beta_k(M)$.
\end{theorem}

\begin{theorem}[Low Rank; proved in \S\ref{s:proof-induct}]\label{t:lowrank}
    For $n\geqslant4$ the rank of any $(n,k)$-matrix is at least $(n-3)^2/2^k$.
\end{theorem}


Denote by $\cap_M$ the mod~$2$ algebraic intersection of $k$-cycles 
on a $2k$-manifold $M$; for a simple definition accessible to non-specialists in topology, see \cite[\S2]{IF}, \cite[\S10]{Sk20}.

Denote by $\oplus$ the mod~$2$ sum of sets.

\begin{remark}\label{rem:inters:2man}
    Here we motivate by low-dimensional examples the definition of an $(n,1)$-matrix (to be introduced later), and Theorem~\ref{t:embed} (Embeddability).

Denote by $K_{n,n}$ the complete bipartite graph with parts $[n]$ and $[n]':=\{k'\ :\ k\in[n]\}$.  
    Let $M$ be a $2$-manifold, and $f\colon [n]^{*2}=K_{n,n}\to M$ a map.
    For $2$-element subsets $P_1 = \{a,b\}$ and $P_2 = \{u,v\}$ of $[n]$ denote by $P = P_1 * P_2 := au'bv'$ the (set of edges of the) cycle of length $4$ in $K_{n,n}$.
    For such cycles $P,Q$ denote
    \[
        A_{P,Q} = A(f)_{P,Q} := fP \cap_M fQ\in\Z_2.
    \]
    The obtained square matrix $A$ is symmetric.
    The matrix $A$ is the Gram matrix (with respect to $\cap_M$) of some homology classes in $H_1(M;\Z_2)$. 
    Hence $\dim H_1(M;\Z_2)\geqslant\rk A$.
    
    If $f$ is an embedding then the following properties hold for any cycles $P,Q \subset K_{n,n}$ of length~$4$ (for additivity it is not even required that $f$ is an embedding):
     
    (\emph{independence})
    $A_{P,Q} = 0$ if $P$ and $Q$ are vertex-disjoint;
        
    (\emph{additivity})
    $A_{P,Q}=A_{X,Q} + A_{Y,Q}$ if\aronly{\footnote{The condition $P = X \oplus Y$ in additivity means that there is $i \in [2]$ such that $P_i = X_i \oplus Y_i$ and $P_{3-i} = X_{3-i} = Y_{3-i}$. Additivity holds, for example, for the cycles $X = [2]*\{u,v\}$, $Y = [2]*\{u,w\}$ and $P = [2] * \{v,w\}$.}} $X,Y \subset K_{n,n}$ are cycles of length $4$ and $P=X\oplus Y$;
        
    (\emph{non-triviality})
    if $\{P,Q\},\{P',Q'\}$ are the two different unordered pairs of cycles of length $4$ in $K_{3,3} \subset K_{n,n}$ such that $P\cap Q=P'\cap Q'$ is the edge $11'$, then $SA := A_{P,Q} + A_{P',Q'} = 1$;
    in other words,
    \[
        SA = A_{\{1,2\} * \{1,2\}, \, \{1,3\} * \{1,3\}} + A_{\{1,2\} * \{1,3\}, \, \{1,3\} * \{1,2\}}
        = A_{11'22', \, 11'33'} + A_{11'23', \, 11'32'}
        = 1.
    \]
    Independence and additivity clearly follow from properties of the mod~$2$ algebraic intersection of $1$-cycles. 
    Non-triviality is a reformulation of \cite[Lemma~17]{FK19}, and is a version of the following result\jonly{ \cite[Satz 5]{vK32}, \cite[Lemma~7.1]{DS22}}: 
    
    \emph{for any general position map of $K_{3,3}$ in the plane there is an odd number of intersection points of images of vertex-disjoint edges} (cf. \cite[Remark~1.3]{KS21}).
    
    \aronly{This result is proved in \cite[Satz 5]{vK32} (for more general case; see an alternative proof as proof of Lemma~\ref{lem:vanKampen}), and is rediscovered in the Kleitman 1976 paper cited in \cite[\S5]{FK19}.}

    A symmetric, independent, additive, non-trivial matrix, whose rows correspond to cycles of length~$4$ in $K_{n,n}$, is called an \emph{$(n,1)$-matrix}.

\end{remark}

Now we move on to the definition of an $(n,k)$-matrix.
        
A \textbf{\emph{$k$-octahedron}} is the set of $k$-faces of a subcomplex (of $[n]^{*k+1}$) isomorphic\aronly{\footnote{\label{f:paral}A $k$-octahedron is uniquely defined by a \emph{parallelepiped} $P_1 \times \ldots \times P_{k+1} \subset [n]^{k+1}$. 
So below one may work with parallelepipeds instead of $k$-octahedra.
This is more convenient for formal statements (because parallelepipeds are simpler than $k$-octahedra), but less convenient for topological motivations.
Everything that is said on the language of $k$-octahedra can be said in the dual language of parallelepipeds.}}
to $[2]^{*k+1} \cong S^k$.
For $2$-element subsets $P_1,\ldots,P_{k+1}\subset[n]$ such a subcomplex
\[
    P = P(P_1, \ldots, P_{k+1}) = P_1 * \ldots * P_{k+1}
\]
is defined by the set $1 \times P_1 \sqcup \ldots \sqcup (k+1) \times P_{k+1}$ of its vertices.  
Its $k$-faces $a_1*\ldots *a_{k+1}$, $a_i \in P_i$, are spanned by vertices $(i,a_i)$.

We consider only matrices with entries in $\Z_2$.
The matrices are square matrices, symmetric, whose rows and whose columns correspond to all $k$-octahedra\aronly{\footnote{Such matrix is a block matrix of size $\binom{n}{2}$, where each block is a block matrix of size $\binom{n}{2}$, etc.}}, unless otherwise specified.

For a $2k$-manifold $M$, a map $f\colon[n]^{*k+1}\to M$, and $k$-octahedra $P,Q$ denote
\[
    A(f)_{P,Q} := fP \cap_M fQ\in\Z_2.
\]
The obtained matrix $A(f)$ is symmetric.
The definition of an $(n,k)$-matrix spells out the properties of $A(f)$ sufficient for the proofs of Theorems \ref{t:embed} and \ref{t:lowrank}.  

A matrix $A$ is said to be an \textbf{\emph{$(n,k)$-matrix}} if it has the following 
properties of independence, additivity, and non-triviality.

A matrix $A$ is said to be \textbf{\emph{independent}} if for any $k$-octahedra $P,Q$
\[
    \text{$A_{P,Q} = 0$ if $P$ and $Q$ are vertex-disjoint}.
\]
Clearly, $A(f)$ is independent if $f$ is an embedding.

A matrix $A$ is said to be \textbf{\emph{additive}} if for any $k$-octahedra $P,Q$
\[
    \text{$A_{P,Q} = A_{X,Q} + A_{Y,Q}$ if $P = X \oplus Y$ for some $k$-octahedra $X,Y$}.
\]
The additivity\aronly{\footnote{For additivity it is not even required that $f$ is an embedding. The condition $P = X \oplus Y$ in additivity holds, for example, for $k$-octahedra $X, Y$ such that $X_j = Y_j$ for $j \neq i$, and $\abs{X_i \cap Y_i} = 1$, for some $i \in [k+1]$.
Then $P = X_1 * \ldots * X_{i-1} * (X_i \oplus Y_i) * X_{i+1} * \ldots * X_{k+1}$.
Presumably there are no other octahedra such that $P = X \oplus Y$.}} of $A(f)$ holds since the mod~$2$ intersection $\cap_M$ distributes over the mod~$2$ summation of $k$-cycles on $M$.

We shorten $\{1\}^{*k+1}$ to $1^{*k+1}$.


A matrix $A$ is said to be \textbf{\emph{non-trivial}} if $SA=1$, where $SA$ is the sum of $A_{P,Q}$ over all unordered pairs $\{P,Q\}$ of $k$-octahedra from $[3]^{*k+1}$ such that $P\cap Q=1^{*k+1}$.
(The sum is meaningful since $A_{P,Q}=A_{Q,P}$.) 

As an example we give explicit formulas for $SA$ (which are not used later).
Denote $\overline x := \{1,x\}$ for $x \in \{2,3\}$. 
Then 
\begin{align*}
    SA &= A_{\{1,2\},\,\{1,3\}} = A_{\overline 2, \overline 3}, \qquad&\text{$k=0$}; \\
    SA &= A_{\overline 2 * \overline 2, \, \overline 3 * \overline 3} + A_{\overline 2 * \overline 3, \, \overline 3 * \overline 2}, \qquad&\text{$k=1$}; \\
    SA &=
    A_{\overline 2 * \overline 2 * \overline 2, \, \overline 3 * \overline 3 * \overline 3} +
    A_{\overline 2 * \overline 3 * \overline 2, \, \overline 3 * \overline 2 * \overline 3} + 
    A_{\overline 3 * \overline 2 * \overline 2, \, \overline 2 * \overline 3 * \overline 3} +
    A_{\overline 2 * \overline 2 * \overline 3, \, \overline 3 * \overline 3 * \overline 2},\;&\text{$k=2$}.
\end{align*}

Now the reader can read the proof of Theorem~\ref{t:lowrank} (Low Rank) in \S\ref{s:proof-induct}.

\begin{lemma}[Embedding]\label{lem:non-triv}
     For any embedding $f\colon [n]^{*k+1}\to M$ to a $2k$-manifold $M$ the matrix $A(f)$ is an $(n,k)$-matrix.
\end{lemma}

In this lemma the additivity and the independence are obvious and are already proved after their definitions. 
The non-triviality is harder and is proved in \S\ref{s:proof-nontr}.

\begin{proof}[Proof of Theorem~\ref{t:embed}]
    Take any embedding $f\colon [n]^{*k+1}\to M$.
    The matrix $A(f)$ is an $(n,k)$-matrix by Lemma~\ref{lem:non-triv}.
    Also, the matrix $A(f)$ is the Gram matrix (with respect to $\cap_M$) of some homology classes in $H_k(M;\Z_2)$. 
    Hence $\rk A(f)\le\beta_k(M)$ by the following well-known result.  
    
    \emph{Let $v_1, v_2, \ldots, v_r$ be vectors in a linear space $V$ over $\Z_2$ with a bilinear symmetric product.
    Then the rank of the Gram matrix of $v_1,v_2,\ldots, v_r$ does not exceed $\dim V$.}
\end{proof}

The following theorem is a stronger version of Theorem~\ref{t:embed} (Embeddability).
    
\begin{theorem}\label{t:basis-emb}
Let $M$ be a closed $2k$-manifold. 
Let $\Omega_M$ be

\begin{itemize}[nosep, leftmargin=5mm]
    \item the identity matrix of size $\beta_k(M)$ if there is $x\in H_k(M;\Z_2)$ such that $x\cap_Mx=1$, and 

    \item the direct sum of $\beta_k(M)/2$ hyperbolic matrices $\begin{pmatrix} 0 & 1 \\ 1 & 0 \end{pmatrix}$, otherwise (it is known that $\beta_k(M)$ is even in the `otherwise' case for closed manifolds).  
\end{itemize}

If $[n]^{*k+1}$ embeds into $M$, then there is a $\beta_k(M)\times\binom{n}{2}^{k+1}$-matrix $Y$ such that $Y^T\Omega_MY$ is an $(n,k)$-matrix.
\end{theorem}
 
\begin{proof}
    Let $f\colon[n]^{*k+1}\to M$ be an embedding.
    By \cite[Theorems 3, 4 and 6]{AA38} 
    (see also \cite[Theorem 1]{MW69} and \cite[Theorem~6.1]{IF}) there is a basis in $H_k(M;\Z_2)$ in which the matrix of $\cap_M$ is $\Omega_M$.
    Let $Y$ be the $\beta_k(M) \times \binom{n}{2}^{k+1}$-matrix whose columns are coordinates of $k$-octahedra in this basis.
    Then $Y^T\Omega_MY = A(f)$ is an $(n,k)$-matrix by Lemma~\ref{lem:non-triv} (Embedding).
\end{proof}

\section{Improvements, idea of proof and corollaries}\label{s:imprem}

\begin{remark}[On almost embeddings]\label{r:alem}
\begin{remarkenumi}
\item\label{en:r-alem-def} For a complex $K$ and any space $M$ a map $f\colon K\to M$ is called an \emph{almost embedding} if $f\sigma \cap f\tau=\emptyset$ for any vertex-disjoint faces $\sigma,\tau$. 
See some motivations in \cite[Remark~6.7.5]{Sk}. 
Clearly, the property of being an almost embedding is preserved under sufficiently small perturbations of the map (as opposed to the property of being an embedding). 
Thus by approximation of continuous maps with PL maps we observe that for complex in manifolds 

$\bullet$ topological embeddability implies PL almost embeddability;   

$\bullet$ PL almost embeddability is equivalent to topological almost embeddability.    

\item \emph{For $k\ge3$ almost embeddability of a $k$-complex to a simply-connected $2k$-manifold implies PL embeddability.} 

For $\R^{2k}$ this is due to van Kampen-Shapiro-Wu, 
and the case of general $2k$-manifolds is analogous as explained in \cite[Proposition~7 for $M=M'$, and \S5, step~3 of proof of Theorem~6]{PT19}, \cite[comments after Theorem~1.2.1]{Sk24}.

The analogue of this result for $k=2$ is false \cite{SSS}, and for $k=1$ is unknown, cf.~\cite{FPS} and the references therein. 

Hence for $k\ge3$ the K\"uhnel conjecture~\ref{c:kuh} on embeddings is equivalent to the analogous conjecture on almost embeddings.  

\item\label{en:r-alem-res} Our topological results (Theorems~\ref{t:quadr:complex}, \ref{t:quadr:join}, \ref{t:embed}, \ref{t:basis-emb}, and Lemma~\ref{lem:non-triv}) hold under the weaker assumption of almost embeddability.  
The arguments do not change.
\end{remarkenumi}
\end{remark}

Denote by $h|_X$ the restriction of a map $h$ to a set $X$.

Denote by $\abs{X}_2 \in \Z_2$ the parity of the number of elements in a finite set $X$.

A map $f\colon K\to M$ to a $2k$-manifold is called a \textbf{general position map} if

\begin{itemize}[nosep]
    \item any vertex-disjoint faces the sum of which dimensions is less than $2k$ have disjoint images,

    \item the restriction of $f$ to any $k$-face has a finite set of self-intersection points,

    \item for any vertex-disjoint $k$-faces $\sigma,\tau$
    \begin{itemize}
        \item[$-$] the set $f\sigma \cap f\tau$ is finite and is disjoint with self-intersections of $f|_\sigma$ and $f|_\tau$,

        \item[$-$] for any point from $f\sigma \cap f\tau$, and some small $(2k-1)$-sphere $S$ centered at this point, the intersections $S\cap f\sigma$ and $S\cap f\tau$ are $(k-1)$-spheres having linking number $\pm1$ in $S$.
    \end{itemize}
\end{itemize}

\begin{remark}[On $\Z_2$-embeddings]\label{r:emb-mod-2}
\begin{remarkenumi}

\item\label{en:r-emb-mod-2-def} Let $M$ be a $2k$-manifold, and $K$ be a $k$-complex. 
    A general position map $f\colon K\to M$ is called a \emph{$\Z_2$-embedding} if $\abs{f\sigma \cap f\tau}$ is even for any vertex-disjoint faces $\sigma,\tau$.
    
    Clearly, any almost embedding 
    (defined in Remark~\ref{r:alem}.\ref{en:r-alem-def}) 
    is a $\Z_2$-embedding.
    Observe that $\Z_2$-embeddability of $k$-complexes in $\R^{2k}$ does not imply almost embeddability, even for $k\ge3$ \cite[Example~3.6]{Me06}.

\item\label{en:r-emb-mod-2-res} Our topological results (Theorems~\ref{t:quadr:complex}, \ref{t:quadr:join}, \ref{t:embed}, \ref{t:basis-emb}, and Lemma~\ref{lem:non-triv}) hold under the weaker assumption of $\Z_2$-embeddability.
The arguments change only in the proof of the independence of $A(f)$, which holds since for any vertex-disjoint $k$-octahedra $P,Q$
\[
A(f)_{P,Q} = fP\cap_M fQ = \sum_{(\sigma,\tau) \in P\times Q} \abs{f\sigma\cap f\tau}_2 = 0.  \]
A converse to the version of Theorem~\ref{t:basis-emb} for $\Z_2$-embeddings is proved in \cite{SS23}. 
It allows to reduce the K\"uhnel conjecture for $\Z_2$-embeddings \cite[Conjecture~1.6.a]{SS23} to a purely algebraic problem.

 
\item For a proof of non-$\Z_2$-embeddability of graphs to $2$-manifolds the Euler inequality does not work, as opposed to non-embeddability; methods of~\cite{FK19} do work.
\end{remarkenumi}
\end{remark}

\begin{remark}[Idea of proof and its relation to known proofs]\label{r:idea}
    \begin{remarkenumi}
        \item\label{en:r-idea-proof} The cases $k=1$ of all our results are proved in~\cite{FK19} (implicitly except 
        for Theorem~\ref{t:quadr:join} (Joinpower); under the weaker assumption of $\Z_2$-embeddability defined in Remark~\ref{r:emb-mod-2}.\ref{en:r-emb-mod-2-def}).
 
However, our proofs are not higher-dimensional generalizations of the proofs from \cite{FK19}.
We did not succeed in generalizing to higher dimensions an important property
proved and used in \cite[\S\S4--5]{FK19}
(the conclusion of Lemma~\ref{l:ind}).
So we discover that the additivity and the independence for $k=1$ imply the property (Lemma~\ref{l:ind}).  
We observe that the matrix $A(f)$ constructed from an embedding $f$ is additive and independent, even in higher dimensions. 
We prove Lemma~\ref{lem:non-triv} (Embedding), cf.~Theorems~\ref{t:embed} and~\ref{t:basis-emb}, and~(\ref{en:r-idea-nontriv}). 
We invent the inductive step for Theorem~\ref{t:lowrank} (Low Rank): an $(n,k-1)$-matrix constructed from the given $(n,k)$-matrix is not a submatrix of the latter, but is the sum of two submatrices (see Lemma \ref{l:hered}). 
Since the base $k=1$ is proved in~\cite[\S4]{FK19} only implicitly, we present in~\S\ref{s:proof-base} a detailed and well-structured proof of Theorem~\ref{t:lowrank} for $k=1$.

         
        \item\label{en:r-idea-nontriv} 
Lemma~\ref{lem:non-triv} for $k=1$ is known \cite[Lemma~17]{FK19} (cf.~the last paragraphs of Remark~\ref{rem:inters:2man}).
This case $k=1$ is easily reduced to a result on linking of points on the circle. 
For higher dimensions the corresponding $(2k-1)$-dimensional linking results are cumbersome, see \aronly{Remark~\ref{r:ramsey}}\jonly{\cite[Remark~7.5]{DS22}}, so we use a different approach.

The non-triviality (and the independence) in Lemma~\ref{lem:non-triv} is the analogue of
 \cite[Proposition~16.C2]{PT19}, \cite[Lemma~1.5]{KS21} (and of \cite[Proposition~16.C1]{PT19})
for $[n]^{*k+1}$ instead of $\Delta_{2k+2}^k$. 
The difference with the proof from \cite{PT19} is that we give a direct proof instead of reference to a cumbersome cohomological criterion.  
The difference with the proof from \cite{KS21} is that we do not use a reduction to a $(2k-1)$-dimensional linking result.
        
Lemma~\ref{lem:non-triv}, and Theorems~\ref{t:embed},~\ref{t:basis-emb} (and their versions of Remarks~\ref{r:alem}.\ref{en:r-alem-res} and~\ref{r:emb-mod-2}.\ref{en:r-emb-mod-2-res}) could be deduced from \cite[Lemma 2.3.1.a]{Sk24} using Lemma~\ref{l:comb} (Combinatorial), see the deductions in
\aronly{Remark~\ref{r:deduc-emb-ks21e}.\ref{en:r-deduc-emb-ks21e:ded-t}\ref{en:r-deduc-emb-ks21e:ded-lem}}\jonly{\cite[Remark~7.2.bc]{DS22}}.
The independence and the additivity in Lemma~\ref{lem:non-triv} are trivial (and so are easier than the deductions). 
We present in~\S\ref{s:proof-nontr} a direct proof of the non-triviality from Lemma~\ref{lem:non-triv}, because such a proof is not very much longer than \aronly{Remark~\ref{r:deduc-emb-ks21e}.\ref{en:r-deduc-emb-ks21e:ded-lem}}\jonly{\cite[Remark~7.2.c]{DS22}}, and because the paper \cite{Sk24} is unpublished.
\end{remarkenumi}
\end{remark}

\begin{remark}[Corollaries]
    \begin{remarkenumi}
        \item
In \cite{PT19} the improved Radon type and Helly type results \cite[Theorem~2 and Corollary~3]{PT19} are deduced from the linear estimate \cite[Theorem~1]{PT19}.
The results below can be deduced analogously from the version (stated in Remark~\ref{r:alem}.\ref{en:r-alem-res}) of Theorem~\ref{t:quadr:complex} (Skeleton) for almost embeddings.

\emph{Let $M$ be a $2k$-manifold and $r \geqslant (k+1)2^{k-1}\sqrt{\beta_k(M)} + 4k + 4$.}

Radon type result.
\emph{Let \( \cl\colon2^M \to 2^M \) be a closure operator. Let \( P \subset M\) be an $r$-element subset such that $\cl S$ is (topologically) $k$-connected for every subset $S \subset P$ of size at most $k+1$.
Then there are two disjoint subsets $P_1, P_2 \subset P$ such that \( \cl P_1 \cap \cl P_2 \neq \emptyset \)}.

Helly type result. 
\emph{Take a finite family of subsets of $M$ such that}
        \begin{itemize}
            \item \emph{the intersection of any proper subfamily is either empty or $k$-connected;}
            
            \item \emph{the intersection of any $r$-element subfamily is nonempty.}
        \end{itemize}
\emph{Then the intersection of all members of the family is nonempty.}

        
        \item Theorem~\ref{t:quadr:complex} (Skeleton) in a standard way gives lower estimation of \emph{crossing number} of $\Delta^k_n$.
        Given a general position map $\Delta^k_n \to \R^{2k}$ with minimal number of crossings, one eliminates any crossing by adding handle $S^k \times S^k$.
        So the crossing number of $\Delta^k_n$ is equal to the number of added handles, which is at least $\frac{(n-4k-2)^2}{2^k(k+1)^2}$ for $n\geqslant5k+3$ by~Theorem~\ref{t:quadr:complex} (Skeleton). 
    \end{remarkenumi}
\end{remark}

\section{The inductive step of Theorem~\ref{t:lowrank} (Low Rank)}\label{s:proof-induct}

In order to grasp the main idea the reader may first check the following proofs for $k=2$.

Proof of Theorem~\ref{t:lowrank} is by induction.
The base $k=1$ is proved implicitly in \cite[\S4]{FK19} and explicitly in~\S\ref{s:proof-base}
(for the idea of the proof see Remark~\ref{r:idea}.\ref{en:r-idea-proof}).

Denote by
\[
    \Z_2^{\binom{n}{2}^{k+1} \times \binom{n}{2}^{k+1}}
\]
the set of matrices whose rows and columns are numerated by $k$-octahedra in $[n]^{*k+1}$.

Take any $A \in \Z_2^{\binom{n}{2}^{k+1} \times \binom{n}{2}^{k+1}}$.
For each $2$-element subsets $U,V \subset [n]$ define \emph{the $k$-coordinate block}
\[
    A_{U,V} \in \Z_2^{\binom{n}{2}^k \times \binom{n}{2}^k}\quad\text{by}\quad
    (A_{U,V})_{P,Q} := A_{U * P,\,V * Q}\quad\text{for $(k-1)$-octahedra}\quad P,Q.
\]

Denote $\overline x := \{1, x\}$ for $x > 1$.

\begin{lemma}[Heredity; proved below]\label{l:hered}
    Suppose that $n \geqslant 4$, $k \geqslant 1$ and $A$ is an $(n,k)$-matrix.
    Then $A_{\overline 2,\overline 3} + A_{\overline 3,\overline 2}$ is an $(n,k-1)$-matrix.
\end{lemma}

\begin{proof}[Inductive step $k-1 \to k$ in the proof of Theorem~\ref{t:lowrank}]
    Take an $(n,k)$-matrix~$A$,
    and set $Z := A_{\overline 2,\overline 3} + A_{\overline 3,\overline 2}$. Then
    \[
        \rk A \geqslant
        \frac{1}{2} \left(\rk A_{\overline 2,\overline 3} + \rk A_{\overline 3,\overline 2}\right) \geqslant
        \frac{1}{2} \rk Z \geqslant
        \frac{(n-3)^2}{2^k},
        \qquad\text{where}
    \]
    \begin{itemize}[leftmargin=8mm, nosep]
        \item the first inequality holds since $\rk A \geqslant \rk A_{U,V}$ for any $U,V$,
        
        \item the second inequality holds by subadditivity of rank,
        
        \item the third inequality holds by the induction hypothesis applied to $Z$, which can be applied by Lemma~\ref{l:hered}.
    \end{itemize}
\end{proof}

\begin{proposition}[One-coordinate swap]
\label{pr:oneswap}
    Suppose that $A \in \Z_2^{\binom{n}{2}^{k+1} \times \binom{n}{2}^{k+1}}$ is independent and additive.
    Suppose that two $k$-octahedra $P = P_1 * \ldots *P_{k+1}$ and $Q = Q_1 * \ldots * Q_{k+1}$ `have only one common vertex', i.e. for some $i \in [k+1]$ we have $\abs{P_i \cap Q_i} = 1$, and for any $j \neq i$ we have $P_j \cap Q_j = \varnothing$.
    Then $A_{P,Q} = A_{P',Q}$ for any $k$-octahedron\aronly{\footnote{One may say, $P'$ shares 
    with $Q$ the same common vertex $P_i \cap Q_i$, and intersects $P$ by the cone over common $(k-1)$-octahedron.}} 
    $$P' = P_1 * \ldots * P_{i-1} * P_i' * P_{i+1} * \ldots * P_{k+1}\quad\text{such that}\quad P_i' \cap Q_i = P_i \cap Q_i.$$
\end{proposition}

\begin{proof}
    For $P = P'$ this is a tautology. Otherwise the proposition follows since
    \[
        A_{P,Q} = A_{P',Q} + A_{P\oplus P',Q} = A_{P',Q},\qquad\text{where}
    \]
    \begin{itemize}[nosep]
        \item the first equality holds by the additivity of $A$,

        \item the second equality holds by the independence of $A$, since the $k$-octahedra $Q$ and $P \oplus P' = P_1 * \ldots  * P_{i-1} * (P_i \oplus P_i') * P_{i+1} * \ldots * P_{k+1}$ are vertex-disjoint.
    \end{itemize}
\end{proof}

\begin{proof}[Proof of Lemma~\ref{l:hered} (Heredity)]
    The additivity holds for $Z := A_{\overline 2,\overline 3} + A_{\overline 3,\overline 2}$ since it holds for $A_{\overline 2,\overline 3}$ and $A_{\overline 3,\overline 2}$.

    The independence\aronly{\footnote{The independence does not hold for blocks $A_{\overline 2,\overline 3}$ and $A_{\overline 3,\overline 2}$ alone, only for their sum.}} 
holds for $Z$ and $n \geqslant 4$, since for vertex-disjoint $(k-1)$-octahedra $P,Q$
    \[
        A_{\overline 2*P,\, \overline 3*Q} = A_{\overline 4*P,\, \overline 3*Q} = A_{\overline 4 * P,\, \overline 2*Q} = A_{\overline 3*P,\, \overline 2*Q},
    \]
    where each equality holds by Proposition~\ref{pr:oneswap}. 
    
    Since $A$ is symmetric, we have $A_{\overline 2 * P,\,\overline 3 * Q} = A_{\overline 3 * Q,\,\overline 2 * P}$, i.e. $A_{\overline 3,\overline 2} = A_{\overline 2,\overline 3}^T$. 
    Hence $Z = A_{\overline 2,\overline 3} + A_{\overline 2,\overline 3}^T$ is symmetric.

    For any $l \in \{k-1,k\}$ denote by $G_l$ the set of unordered pairs of $l$-octahedra from $[3]^{*l+1}$ whose intersection is $1^{*l+1}$.
    The non-triviality holds for $Z$ since
    \[
        S Z = S A_{\overline 2,\overline 3} + S A_{\overline 3,\overline 2} =
        \sum\limits_{\{P,Q\} \in G_{k-1}}
            (A_{\overline 2 * P,\, \overline 3 * Q} + A_{\overline 3 * P,\, \overline 2 * Q}) =
        \sum\limits_{\{P',Q'\} \in G_k}
            A_{P',Q'} = 1,
    \]
    where the last equality is the non-triviality of $A$.
\end{proof}


\section{Proof of the non-triviality in Lemma~\ref{lem:non-triv} (Embedding)}\label{s:proof-nontr}

In order to grasp the main idea the reader may first check the following proofs for $k=2$.

\setcounter{mcnt}{0}
\renewcommand{\themcnt}{\text{(\arabic{mcnt})}}

\begin{proof}[Proof of the non-triviality in Lemma~\ref{lem:non-triv} (Embedding)]
    Recall that $\Int M$ is the interior of~$M$.
    Denote by $\partial x$ the boundary of $x$, where $x$ is either a $k$-face or a $2k$-ball.
    We may assume that $n=3$. 
    
    The join \cite[\S4.2]{Ma03} of $s$ non-empty complexes is $(s-2)$-connected \cite[Proposition~4.4.3]{Ma03}. 
    So $[3]^{*k+1}$ is $(k-1)$-connected. 
    Let $L:=([3]^{*k+1})^{(k-1)}$ be the union of all those faces of $[3]^{*k+1}$ whose dimension is less than $k$. 
    Then $f|_L$ is null-homotopic. 
We shall use the following Borsuk Homotopy Extension Theorem \cite[\S5.5]{FF89}: 

\emph{if $(K,L)$ is a polyhedral pair, $Z\subset\R^m$, \ $F\colon L\times [\,0,1\,]\to Z$ is a homotopy, and $g\colon K\to Z$ is a map such that $g|_L=F|_{L\times0}$, then $F$ extends to a homotopy $G\colon K\times [\,0,1\,]\to Z$ such that $g=G|_{K\times0}$.}
 
Hence $f$ is homotopic to a map $f''\colon[3]^{*k+1}\to M$ such that $f''L$ is a point. 
    Take a $2k$-ball $B\subset\Int M$ such that $B\cap f''[3]^{*k+1} = f''L\in\partial B$. 
    Then $f''$ (and so $f$) is homotopic to a general position map $f'\colon[3]^{*k+1}\to M-\Int B$ such that $f'L\subset\partial B$.
    By general position we may assume that $f'|_L$ is an embedding.\aronly{\footnote{Observe that $f'$ is not necessarily an embedding, almost embedding or a $\Z_2$-embedding. 
    This paragraph is analogous to \cite[Lemma~12]{PT19}, \cite[\S2.3, beginnig of proof of Lemma~2.3.1.a]{Sk24}.}}
    
    Define the map $g\colon [3]^{*k+1}\to B$ to be $f'$ on $L$, and to be the cone map over $f'|_{\partial\sigma}$ with a vertex in $\Int B$ on every $k$-face $\sigma$ of $[3]^{*k+1}$.
    By proper choosing these vertices we may assume that $g$ is a general position map. 
    Then  
    \begin{multline*}
        SA(f) \stackrel{\juststep}{=} SA(f')                                               \stackrel{\juststep}{=}
        \sum_{\{P,Q\}\in G_k}
            f'P \cap_M f'Q                               \stackrel{\juststep}{=}  
        \sum_{\{P,Q\}\in G_k}
            f'_gP \cap_M f'_gQ
            \stackrel{\juststep}{=} \\
        \sum_{\{P,Q\}\in G_k}
            \sum_{\{\alpha,\beta\} \in T\{P,Q\}}
                f'_g\alpha \cap_M f'_g\beta
                \stackrel{\juststep}{=}  
        \sum_{\{\alpha, \beta\}\in H}
            f'_g\alpha \cap_M f'_g\beta
            \stackrel{\juststep}{=}     
        \sum_{\{\alpha, \beta\}\in H}
            \abs{f'_g\alpha \cap f'_g\beta}_2
            \stackrel{\juststep}{=}   \\ 
        \sum_{\{\alpha, \beta\}\in H}
            (|f'\alpha\cap f'\beta|_2+|g\alpha\cap g\beta|_2) \stackrel{\juststep}{=}  
        \sum_{\{\alpha, \beta\}\in H}
            |g\alpha \cap g\beta|_2              \stackrel{\juststep}{=} 1.
    \end{multline*} 
\setcounter{mcnt}{0}
    Here
    \begin{itemize}[leftmargin=6mm, nosep]
        \item equality \cnt holds since $SA(f)$ is independent of homotopy of $f$;

        \item $G_k$ is the set of unordered pairs of $k$-octahedra from $[3]^{*k+1}$ whose intersection is~$1^{*k+1}$;
        
        \item equality \cnt is the definition of $SA(f')$;

        \item $f'_g\xi := f'\xi \cup g\xi$, where $\xi$ is either a $k$-octahedron or a $k$-face;

        \item equality \cnt holds since $gP$ and $gQ$ are null-homologous;

        \item $T\{P,Q\}$ is the set\aronly{\footnote{Note that $T\{P,Q\}$ is the image of the torus $P \times Q$ under the projection to the quotient of $[3]^{*k+1}\times[3]^{*k+1}$ under the symmetry exchanging factors.}} of pairs $\{\alpha,\beta\}$ formed by $k$-faces $\alpha,\beta$ of $[3]^{*k+1}$ such that either $\alpha \in P$ and $\beta \in Q$, or vice versa (note that $\alpha = \beta = 1^{*k+1}$ is possible);

        \item
        equality \cnt holds since for any $\{P,Q\} \in G_k$
        \[
            f'_gP \cap_M f'_gQ =   
            \sum_{(\alpha,\beta) \in P \times Q} 
                f'_g\alpha \cap_M f'_g\beta = 
            \sum_{\{\alpha,\beta\} \in T\{P,Q\}}
                f'_g\alpha \cap_M f'_g\beta,
            \quad\text{where}
        \]
        \begin{itemize}[leftmargin=5mm]
            \item[$\circ$] the second (and the first) term is meaningful
            for any \emph{unordered} pair $\{P,Q\}$ since the term
            is symmetric in $P,Q$,
            
            \item[$\circ$] the first equality holds since $f'_gP = \bigoplus\limits_{\alpha\in P} f'_g\alpha$ for $k$-cycles $f'_g\alpha$, and since we have the analogous equality for $f'_gQ$ and $f'_g\beta$,

            \item[$\circ$] we prove the second equality as follows: since $P\cap Q=1^{*k+1}$, for $\{\alpha, \beta\} \in T\{P,Q\}$ exactly one of the pairs $(\alpha, \beta)$ and $(\beta, \alpha)$ lies in $P \times Q$; hence the formula $(\alpha,\beta) \mapsto \{\alpha, \beta\}$ gives a bijection $P \times Q \to T\{P,Q\}$; this implies the second equality;
        \end{itemize}
        
        \item $H$ is the set of all unordered pairs of vertex-disjoint $k$-faces of $[3]^{*k+1}$;

        \item equality \cnt holds since $H=\bigoplus\limits_{\{P,Q\}\in G_k} T\{P,Q\}$, which is a reformulation of Lemma~\ref{l:comb} (Combinatorial) below;

        \item equality \cnt holds by definition of $\cap_M$ since $f',g$ are general position maps;

        \item equality \cnt holds since $f'\alpha \cap g\beta = \varnothing$ for vertex-disjoint $\alpha,\beta$; this holds since 
        $$f'[3]^{*k+1} \subset M - \Int B,\quad g[3]^{*k+1} \subset B,\quad f'|_L = g|_L,$$ 
        and $g$ is a general position map;

\addtocounter{mcnt}{1}
        \item equality \cnt is the result of van Kampen \cite[Satz~5]{vK32} (see~\aronly{Lemma~\ref{lem:vanKampen}}\jonly{\cite[Lemma~7.1]{DS22}}).
    \end{itemize}
\addtocounter{mcnt}{-2}
    It remains to prove equality~\juststep.
    For a general position map $h\colon [3]^{*k+1}\to M$ the \emph{van Kampen number}
    \[
        v(h) := \sum_{\{\alpha, \beta\}\in H} |h\alpha \cap h\beta|_2
    \]
    is the parity of the number of all pairs $\{\alpha,\beta\}$ of vertex-disjoint $k$-faces of $[3]^{*k+1}$ such that $\abs{h\alpha \cap h\beta}_2 = 1$.
    Then \themcnt\ holds since
    \[
        v(f') = v(f) = 0.
    \]
    Here the second equality holds since $f$ is an embedding.
    Let us present a fairly standard argument for the first equality. 
    
    For a general position map $h\colon [3]^{*k+1}\to M$ the \emph{intersection cocycle} $\nu(h) \subset H$ is the set of all pairs $\{\sigma,\tau\}$ such that $\abs{h\sigma \cap h\tau}_2 = 1$;
    so $v(h) = \babs{\nu(h)}_2$.
    For vertex-disjoint $(k-1)$-face $e$ and $k$-face $\alpha$ the \emph{elementary coboundary} of $(\alpha,e)$ is the set of all unordered pairs $\{\alpha,\beta\}$ of vertex-disjoint $k$-faces such that $e\subset\beta$.
    Since $f'$ is homotopic to $f$, by a lemma of van Kampen-Shapiro-Wu-Johnson \cite[Lemma~3.5]{Sh57} \cite[Lemma~2.3.2 and Remark 2.3.3]{Sk24}, $\nu(f)$ and $\nu(f')$ are \emph{cohomologous}, i.e. $\nu(f) \oplus \nu(f')$ is the mod~$2$ sum of some elementary coboundaries. 
    For any $(k-1)$-face $e = e_1 * \ldots * e_{k+1}$ of $[3]^{*k+1}$ there is the unique $t(e)\in[k+1]$ such that $e_{t(e)}=\varnothing$. 
    For a $k$-face $\alpha$ and a $(k-1)$-face $e$ the elementary coboundary of $(\alpha, e)$ consists of pairs $\{\alpha,\beta = \beta_1 * \ldots *\beta_{k+1}\}$ such that $\beta_{t(e)} \neq \alpha_{t(e)}$ and $\beta_s=e_s$ for every $s\ne t(e)$. 
    Hence any elementary coboundary consists of two elements. 
    Since the size of any elementary coboundary is even, $\babs{\nu(f)\oplus\nu(f')}_2 = 0$. 
    Hence $\babs{\nu(f)}_2 = \babs{\nu(f')}_2$. 
\end{proof}

\begin{lemma}[Combinatorial]\label{l:comb} 
    The following two sets\aronly{\footnote{
    In the dual language of parallelelepipeds (see footnote~\ref{f:paral}) Lemma~\ref{l:comb} states that the following two sets are equal:
            
    \begin{itemize}[nosep]
        \item the set of all ordered pairs of vectors in $[3]^{k+1}$ such that the vectors have no equal components;

        \item the mod~$2$ sum of products $P\times Q$ over all ordered pairs $(P,Q)$ of parallelepipeds whose intersection is $1^{\times k+1}$.
    \end{itemize}}}
    are equal:

    \begin{itemize}[nosep]
        \item the set of all ordered pairs $(\sigma,\tau)$ of vertex-disjoint $k$-faces of $[3]^{*k+1}$;

        \item the mod $2$ sum of products $P\times Q$ over all ordered pairs $(P,Q)$ of $k$-octahedra from $[3]^{*k+1}$ whose intersection is $1^{*k+1}$.
    \end{itemize}
\end{lemma}

Lemma~\ref{l:comb} can be deduced from \cite[Proposition~2.5.4.d]{KS21e}, but we present a simpler
direct proof below.
The analogue of Lemma~\ref{l:comb} for $\Delta_{2k+2}^k$ instead of $[3]^{*k+1}$
is the main idea of \cite[Lemma~20]{PT19}, and is implicit in \cite[proof of Lemma~20]{PT19}; cf.~\cite[Proposition~2.2]{KS21}.

\begin{proof}[Proof of Lemma~\ref{l:comb}]
    It suffices to prove that

    {
    \renewcommand{\theenumi}{\Alph{enumi}}
    \renewcommand{\labelenumi}{(\Alph{enumi})}
    \begin{enumerate}
        \item\label{enum:sq:non-adj}
        \emph{for any vertex-disjoint $k$-faces $\alpha,\beta \in [3]^{*k+1}$ there is exactly one ordered pair $(P,Q)$ of $k$-octahedra from $[3]^{*k+1}$ such that $P\cap Q = 1^{*k+1}$ and $(\alpha,\beta) \in P \times Q$}; and
        
        \item\label{enum:sq:adj}
        \emph{for any $k$-faces $\alpha, \beta \in [3]^{*k+1}$ sharing a common vertex there is an even number of ordered pairs $(P,Q)$ of $k$-octahedra from $[3]^{*k+1}$ such that $P\cap Q = 1^{*k+1}$ and $(\alpha,\beta) \in P \times Q$}.
    \end{enumerate}
    }

    For a $k$-octahedron $P$ from $[3]^{*k+1}$ containing the $k$-face $1^{*k+1}$ denote by $\sigma^P$ the $k$-face of $P$ that is opposite to $1^{*k+1}$.
    
    \smallskip    
    \emph{Proof of \eqref{enum:sq:non-adj}.}
    Take any $k$-octahedra $P,Q$ from $[3]^{*k+1}$ such that $\alpha \in P$, $\beta \in Q$ and $P\cap Q = 1^{*k+1}$.
    Take any $i \in [k+1]$.
    
    \emph{Suppose that $\alpha_i \neq 1$.}
    Since $\alpha \in P$, it follows that $\sigma^P_i = \alpha_i$.
    Since $P\cap Q = 1^{*k+1}$, we have $\sigma^Q_i = 5 - \sigma^P_i = 5 - \alpha_i$.
    
    \emph{Suppose that $\alpha_i = 1$.}
    Then $\beta_i \neq 1$.
    Hence $\sigma^Q_i = \beta_i$ and $\sigma^P_i = 5 - \beta_i$ analogously to the previous paragraph.
    
    Hence the $i$-th coordinates $\sigma^P_i$ and $\sigma^Q_i$ are uniquely defined for each $i \in [k+1]$.
    Thus there is exactly one pair $(P,Q)$ from the statement of~\eqref{enum:sq:non-adj}.

    \smallskip    
    \emph{Proof of \eqref{enum:sq:adj}.}
    Since $\alpha$ and $\beta$ share a common vertex, we may assume that $\alpha_1 = \beta_1$ (the other cases are analogous).
    
    \emph{Suppose that $\alpha_1\ne1$.}
    Take any $k$-octahedra $P,Q$ from $[3]^{*k+1}$ such that $P\cap Q = 1^{*k+1}$.
    Since $\sigma^P,\sigma^Q$ are vertex-disjoint, we have $\sigma^P_1\ne\sigma^Q_1$. 
    Hence either $\sigma^P_1 \ne \alpha_1$ or $\sigma^Q_1 \ne \beta_1$. 
    Then either $\alpha \not\in P$ or $\beta \not\in Q$. 
    Thus $(\alpha, \beta) \not\in P \times Q$.
        
    \emph{Suppose that $\alpha_1=1$.}
    For every $k$-octahedron $R = \overline{r_1} * \overline{r_2} * \ldots * \overline{r_{k+1}} \subset [3]^{*k+1}$ denote $R':= \overline{5-r_1} * \overline{r_2} * \ldots * \overline{r_{k+1}}$.
    Clearly, $(R')'=R$, and if $P\cap Q = 1^{*k+1}$, so $P'\cap Q' = 1^{*k+1}$.
    Thus the pairs $(P,Q)$ from~\eqref{enum:sq:adj} split into couples corresponding to `opposite' pairs $(P,Q)$ and $(P',Q')$.
    Since $\alpha_1 = 1$, the $k$-face $\alpha$ is contained either in both $P$ and $P'$ or in none.
    Analogously for $\beta$.
    Then for every couple $\{(P,Q),(P',Q')\}$ the pair $(\alpha,\beta)$ is contained either in both $P \times Q$ and $P' \times Q'$ or in none.
    This implies~\eqref{enum:sq:adj}.
\end{proof}


\section{Proof of Theorem~\ref{t:lowrank} (Low Rank) for $k=1$}\label{s:proof-base}

Here and below rows and columns of matrices are not necessarily numerated by octahedra (as opposed to \S\ref{s:intr},\ref{s:proof-induct}).

For a (block) matrix $X$ whose rows are numerated by $[\ell] \times [m]$, and any $i,j\in[\ell]$ define the \emph{$m\times m$-block} $X_{i,j}$ by $(X_{i,j})_{a,b} := X_{(i,a)(j,b)}$. 

Denote by $0_m$ the zero $m \times m$-matrix, and by $J_m$ the $m\times m$-matrix consisting of units.

\begin{lemma}
\label{lem:block-units}
    Suppose $N$ is a matrix whose rows are numerated by $[\ell] \times [m]$, such that for every $i,j\in[\ell]$
    \[
        \begin{cases}
            N_{i,j} = 0_m, & i \leqslant j, \\
            N_{i,j}\in\{0_m,J_m\}, & i > j.
        \end{cases}
    \]
    Then $\rk N \leqslant \ell-1$.
\end{lemma}

\begin{proof}
    Define the matrix $F$ of size $\ell$ so that $F_{i,j} = 0$ if and only if $N_{i,j} = 0_m$.
    Clearly, $F_{i,j}=0$ if $i \leqslant j$.
    Then the first row of $F$ consists of zeros.
    Thus $\rk N = \rk F \leqslant \ell-1$.
\end{proof}

A matrix $Y$ with $\Z_2$-entries is said to be \textbf{\emph{tournament}} if $Y_{a,b}+Y_{b,a}=1$ for all $a\ne b$.
In other words, $Y$ is a tournament matrix if $Y + Y^T$ is the inversed identity matrix, i.e. the sum of the identity matrix and $J_m$, where $m$ is the number of rows of $Y$.

\begin{lemma}
[Tournament; {\cite[Theorem~1]{Ca91}}]\label{lem:caen}
    The rank of a tournament $m \times m$-matrix is at least $\dfrac{m-1}{2}$.
\end{lemma}

\begin{proof}
    For a tournament $m\times m$-matrix $Y$
    \[
        \rk Y = \frac{\rk Y + \rk Y^T}{2} \geqslant \frac{\rk (Y + Y^T)}{2} = \frac{\rk (I_m + J_m)}{2}  \geqslant \frac{\rk I_m - \rk J_m}{2} = \frac{m-1}{2},
    \]
    where
    \begin{itemize}[nosep]
        \item $I_m$ is the identity matrix,
        \item the inequalities hold by subadditivity of rank,
        \item the middle equality holds since $Y$ is a tournament matrix, so $Y+Y^T = I_m + J_m$.
    \end{itemize}
\end{proof}

Recall that 
\[
    [m_1] \sqcup [m_2] \sqcup \ldots \sqcup [m_{\ell}] =
    1 \times [m_1] \sqcup 2 \times [m_2] \sqcup \ldots \sqcup \ell \times [m_{\ell}],
\]
so that the disjoint union of $\ell$ copies of $[m]$ is $[\ell] \times [m]$.

Take a (block) matrix $X$ whose rows are numerated by $[m_1] \sqcup \ldots \sqcup [m_{\ell}]$.

For $i,j\in[\ell]$ define the $m_i\times m_j$-\emph{block} $X_{i,j}$ by $(X_{i,j})_{a,b} := X_{(i,a)(j,b)}$.

The matrix $X$ is said to be \textbf{\emph{tournament-like}} if for every $i \in [\ell]$ the diagonal block $X_{i,i}$ is a tournament matrix.

The matrix $X$ is said to be \textbf{\emph{diagonal-like}} if it is obtained by removing some rows and columns symmetric to the rows, from a block matrix $\t{X}$ with the following properties:
\begin{itemize}[nosep]
    \item its rows are numerated by $[\ell] \times [m]$,
    
    \item $m \geqslant m_i$ for every $i \in [\ell]$,
    
    \item the under-diagonal block $\t{X}_{i,j}$ is a diagonal matrix for every $i,j\in[\ell]$, $i > j$.
\end{itemize}

\begin{lemma}
[Diagonal-Tournament]\label{lem:tour-diag}
    The rank of any tournament-like diagonal-like matrix, whose rows are numerated by $[m_1] \sqcup \ldots \sqcup [m_{\ell}]$, is at least
    $\sum\limits_{i=1}^{\ell} \dfrac{m_i - 1}{2}$.
\end{lemma}

\begin{proof}
    The simple case when there are no units in the under-diagonal blocks follows by Lemma~\ref{lem:caen} because all under-diagonal blocks are zero matrices.

    The proof is by induction on $m_1+\ldots+m_{\ell}$. 
    The base follows by the above simple case.
    Let us prove the inductive step in the case when there is a unit in the union of under-diagonal blocks.
    Denote by $D$ the given matrix.
    Arrange the rows and the columns of $D$ lexicographically\aronly{\footnote{
    The pair $(x,a)$ is said to be \emph{lexicographically smaller} than $(y,b)$ if either $x = y$ and $a < b$, or $x < y$.}}.
    Let $(i,a)$ be the lexicographically maximal (i.e. the lowest) row whose intersection with the union of under-diagonal blocks of $D$ is non-zero.
    Let $(j,b)$ be the lexicographically minimal (i.e. the leftmost) column whose intersection with the row $(i,a)$ is non-zero.
    Formally, $(i,a)$ is the lexicographically maximal row such that there is 
    \[
        (j,b)\quad\text{for which}\quad i > j,\quad D_{(i,a)(j,b)}=1,\quad\text{and}\quad D_{(i,a)(j',b')} = 0\quad\text{for all}\quad
        (j', b') < (j, b).
    \]
    (Note that by the choice of the row, $D_{(i',a')(j,b)} = 0$ for all
    $(i',a') > (i,a)$.)
    
    Let $D'$ be the matrix obtained from $D$ by adding the row $(i,a)$ to all other rows whose intersection with the column $(j,b)$ is non-zero.
    Let $D''$ be the matrix obtained from $D'$ by adding the column $(j,b)$ to all other columns whose intersection with the row $(i,a)$ is non-zero.
    In $D''$ the union of the row $(i,a)$ and the column $(j,b)$ contains only one unit, located at the intersection of these row and column.
    
    Let $D'''$ be the matrix obtained from $D''$ by removing the rows $(i,a)$, $(j,b)$, and the columns $(i,a)$, $(j,b)$.
    Then $D''' \in \Z_2^{([n_1] \sqcup \ldots \sqcup [n_{\ell}])^2}$ for $n_i = m_i - 1$, $n_j = m_j - 1$, and $n_s = m_s$ for $s \not\in \{i,j\}$, and $D'''$ is a tournament-like diagonal-like matrix.
    Then by induction hypothesis
    \[
        \rk D = \rk D' = \rk D'' \geqslant \rk D''' + 1 \geqslant 
        \sum_{s=1}^{\ell} \frac{n_s-1}{2} + 1 = \sum_{s=1}^{\ell} \frac{m_s - 1}{2}.
    \]
\end{proof}

Recall that
\[
    \Z_2^{\binom{n}{2}^2 \times \binom{n}{2}^2}
\]
is the set of matrices whose rows and columns are numerated by $2$-octahedra (i.e. by cycles of length~$4$) in \([n]^{*2}=K_{n,n}\).
Recall that 
\[
    \overline a = \{1,a\}\quad\text{for}\quad a>1.
\]
For $A \in \Z_2^{\binom{n}{2}^2 \times \binom{n}{2}^2}$ define the matrix 
\[
    B = B(A) \in \Z_2^{[n-1]^2 \times [n-1]^2}\quad\text{by}\quad
    B_{(i,a)(j,b)} := A_{\overline{i+1} * \overline{a+1},\,\overline{j+1} * \overline{b+1}}
\]
(this notation helps to make a transition from all cycles of length~$4$ in $K_{n,n}$ to cycles containing the edge $(1,1')$, and from matrices whose rows and columns are numerated by such cycles to block matrices).

Recall that an \emph{inversed diagonal matrix} is the sum of some diagonal matrix and $J_m$.

\begin{lemma}
\label{l:ind}
    Suppose that $n \geqslant 4$ and $A \in \Z_2^{\binom{n}{2}^2 \times \binom{n}{2}^2}$ is independent and additive.
    Take $B = B(A)$.
    Then for any pairwise distinct $i,j,s \in [n-1]$ the residue $B_{(i,a)(j,b)} + B_{(s,a)(j,b)}$ does not depend on distinct $a,b \in [n-1]$.
    In other words, the sum $B_{i,j} + B_{s,j}$ of blocks is either a diagonal matrix or an inversed diagonal matrix.
\end{lemma}

\begin{proof}
    For fixed $i,j,s$ denote
    \[
        P(a) := \overline{i+1} * \overline{a+1} \oplus \overline{s+1} * \overline{a+1} = \{i+1,s+1\} * \overline{a+1} \quad\text{and}\quad
        Q(b) := \overline{j+1} * \overline{b+1}.
    \]
    By the additivity,
    \[
        B_{(i,a)(j,b)} + B_{(s,a)(j,b)} = A_{P(a),Q(b)}.
    \]
    
    The residue $B_{(i,a)(j,b)} + B_{(s,a)(j,b)}$ does not depend on $a$ since for any $a' \in [n-1]$ distinct from $a$ and $b$
    \[
        A_{P(a),Q(b)} = A_{P(a'),Q(b)}
    \]
    by Proposition~\ref{pr:oneswap} (One-coordinate swap).
    Analogously the residue does not depend on $b$.
    Now, the residue does not depend on both $a$ and $b$ since $n-1 \geqslant 3$.
\end{proof}

\begin{lemma}
\label{l:pre-tour-like}
    Suppose that $n \geqslant 4$ and $A$ is an $(n,1)$-matrix.
    Then for any $j > 1$ the off-diagonal block $B(A)_{1,j}$ is a tournament matrix.
\end{lemma}

\begin{proof}
    Denote $B := B(A)$.
    First, let us show that $B_{1,2}$ is a tournament matrix.
    Take the matrix $Z := A_{\overline 2,\overline 3} + A_{\overline 3, \overline 2}$, which is an $(n,0)$-matrix by Lemma~\ref{l:hered} (Heredity).
    By the symmetry of $A$ and $Z$ it suffices to check that $Z_{\overline x, \overline y} = 1$ for any numbers $1 < x < y$.
    This follows since
    \[
        Z_{\overline x, \overline y} = Z_{\overline{2\mathstrut}, \overline{y\mathstrut}} = Z_{\overline 2, \overline 3} = 1,
        \qquad\text{where}
    \]
    \begin{itemize}[nosep]
        \item each of the first and the second equalities is either a tautology or holds by Proposition~\ref{pr:oneswap} (One-coordinate swap),

        \item the last equality is the non-triviality of $Z$.
    \end{itemize}
    
    Now for any $j > 2$ the matrix $B_{1,2} + B_{1,j}$ is either a diagonal matrix or an inversed diagonal matrix, by the symmetry of~$B$ and Lemma~\ref{l:ind}. Then $B_{1,j}$ is a tournament matrix.
\end{proof}

\begin{proof}[Proof of Theorem~\ref{t:lowrank} for $k=1$]
    Take $B = B(A)$.
    Define the matrix $C$ whose rows are numerated by $[n-2] \times [n-1]$, by $C_{i,j} := B_{i+1,j+1} + B_{1,j+1}$, i.e.~$C$ is obtained from $B$ by row addition and taking submatrix.
    By Lemma~\ref{l:ind}, for every $i \neq j$ the block $C_{i,j}$ is either a diagonal matrix or an inversed diagonal matrix.
    Thus the following formula defines the matrix $D$ of the same block structure as $C$:
    for $i,j \in [n-2]$
    \[
        D_{i,j} :=
        \begin{cases}
            C_{i,j}, & \text{if either $i \leqslant j$ or $C_{i,j}$ is a diagonal matrix;} \\
            C_{i,j} + J_{n-1}, & \text{if both $i > j$ and $C_{i,j}$ is an inversed diagonal matrix.}
        \end{cases}
    \]
    Now the theorem follows since
    \[
        \rk A \geqslant
        \rk B \geqslant
        \rk C \geqslant
        \rk D - \rk (C+D) \geqslant
        \frac{(n-2)^2}{2} - (n-3) \geqslant
        \frac{(n-3)^2}{2}.
    \]
    Here
    \begin{itemize}[nosep]
        \item the first and second inequalities follow by definition of $B$ and $C$, respectively;
        
        \item the third inequality holds by subadditivity of rank;
        
        \item the last inequality is obvious.
    \end{itemize}
    The fourth inequality holds by Lemma~\ref{lem:tour-diag} (Diagonal-Tournament) applied to $D$, $\ell = n - 2$ and $m_1=\ldots=m_\ell = n - 1$, and Lemma~\ref{lem:block-units} applied to $D+C$, $\ell = n - 2$ and $m = n - 1$.
    It is obvious that the hypotheses of Lemma~\ref{lem:block-units} are fulfilled for $D+C$.
    It remains to prove that the hypotheses of Lemma~\ref{lem:tour-diag} (Diagonal-Tournament) are fulfilled for~$D$.
    
    Clearly, $D$ is diagonal-like.
    We prove that $D$ is tournament-like as follows.
    Take any $s \in [n-2]$.
    By Lemma~\ref{l:pre-tour-like} for $j=s+1$, the block $B_{1, s+1}$ is a tournament matrix.
    Since $A$ is symmetric, $B_{s+1,s+1}$ is symmetric.
    Then $D_{s,s} = C_{s,s} = B_{s+1,s+1}+B_{1,s+1}$ is a tournament matrix.
\end{proof}

\begin{remark}[On generalization of linear algebraic properties]
    During the work with similar linear algebraic properties of $(n,k)$-matrices
    (e.g.~Proposition~\ref{pr:oneswap}, Lemma~\ref{l:ind}) we got the impression that probably there may be another simple property, which can help with the estimate $c_k n^{k+1}$.
    However, we did not succeed in obtaining it.
\end{remark}

\aronly{

\section{An approach via generalized Gr\"unbaum-Kalai-Sarkaria conjecture}\label{s:adipa}

 
Here we present an interesting alternative approach \cite{Ad18, AP24} to the asymptotic version~\ref{r:hist:complex}.\ref{en:r-hist-asym} of the K\"uhnel conjecture on embeddings. 
(This approach does not work for almost- or $\Z_2$-embeddability defined in Remarks~\ref{r:alem} and~\ref{r:emb-mod-2}.)
This version follows from Conjectures~\ref{c:kuh-simp}.\ref{en:c-kuh-simp-asym} and~\ref{c:adi} below.
Conjecture~\ref{c:kuh-simp}.\ref{en:c-kuh-simp-asym} in its turn follows from Conjecture~\ref{c:ggks} as explained in Remark~\ref{r:simp-asym-deduc}.

We write `conjecture' not `theorem' in~\ref{c:ggks}, \ref{c:kuh-simp} and~\ref{c:adi} because there is a publicly available but so far unanswered criticism of their proofs (see Remarks~\ref{r:ggks}, \ref{r:helpadi} and~\ref{r:histc-adpa}).  
These remarks show that the argument in \cite{Ad18, AP24} for Conjectures~\ref{c:ggks}, \ref{c:kuh-simp} and~\ref{c:adi} is not reliable up to the standards of refereed journals, hence these papers do not present legitimate claims for the conjectures. 
So we do not 
present other critical remarks; we do not certify that the other parts of these papers are reliable.
Besides, Remark~\ref{r:hist:complex}.\ref{en:r-histc-wrong} suggests that at least one of Conjectures~\ref{c:ggks} and~\ref{c:adi} is wrong or hard to prove, so that a (dis)proof of the K\"uhnel conjecture~\ref{c:kuh} on embeddings would require technique different from the one used for Conjectures~\ref{c:ggks} and~\ref{c:adi}\footnote{The paper \cite{Ad18} is unpublished. 

\smallskip
\emph{Historical note.}
The paper \cite{AP24} appeared on arXiv in April 2024, after arXiv version 2 of the current paper (where the proof is complete, or at least not criticized) appeared in September, 2022. 
So the paper \cite{AP24} was apparently submitted to a journal after the submission (of the shortened arXiv version 3) of the current paper to a journal in February, 2024. 
See also footnotes \ref{f:refe} and \ref{f:letadi}.}. 

\smallskip
\emph{Suppression of criticism.}
We sent to K.~Adiprasito the criticism of \cite{Ad18} from Remarks~\ref{r:ggks} and~\ref{r:helpadi} in April 2024, before putting to arXiv version~4 of the current paper (containing the remarks). 
So far neither we receive any reply nor an update of \cite{Ad18} corrected according to the criticism, or referring to the criticism, appeared on arXiv. 
We sent to K.~Adiprasito and Z.~Pat\' akov\' a the criticism of (arXiv version~2 of) \cite{AP24} from Remarks~\ref{r:histc-adpa}.(b--h) in February 2025 before putting to arXiv version~6 of the current paper (containing the remark). 
After the appearance of this version~6 on arXiv, we received letters 
expressing disagreement with the criticism of \cite{AP24}, but not justifying (and even not mentioning) disagreement with any of the specific Remarks~\ref{r:histc-adpa}.(b--h); see our replies in Remark~\ref{r:ap-let}.
In (arXiv version~3 and journal version of) \cite{AP24} all the specific remarks (except a part of minor Remark~\ref{r:histc-adpa}.f from arXiv version 6 of the current paper) are ignored, and no reference to the criticism is presented.
We sent to the Editors of Bull. LMS the criticism of (published version of) \cite{AP24} from Remarks~\ref{r:histc-adpa}.(b--i) in March 2025. 
The Editors' 
reply (Remark \ref{r:aplms}), its criticism (Remark \ref{e:ap-blms}), lack of the Editors' reply to the criticism, and no information on the criticism on the journal's site, could be interesting to math community, because they show that the reliability standards of Bull. LMS are surprisingly low. 

\smallskip
Recall that for a $2k$-manifold $M$ the \emph{rational Betti number} \(b_k(M) = \dim H_k(M; \Q)\) is defined in Remark~\ref{r:hist:complex}.\ref{en:r-hist-beta}.
For a complex $K$ denote by $f_j(K)$ the number of its $j$-faces, and denote 
$\gamma(K) := f_k(K)-(k+2)f_{k-1}(K)$.

\begin{conjecture}[generalized Gr\"unbaum-Kalai-Sarkaria conjecture]\label{c:ggks}
For any subcomplex~$K$ of any triangulation of any $2k$-manifold~$M$ we have
\[\binom{2k + 1}{k+1} b_k(M) \ge \gamma(K).\]
\end{conjecture}

\begin{remark}\label{r:ggks} 
Conjecture \ref{c:ggks} is stated as a result in \cite[Remark~4.9]{Ad18}. 
In \cite{Ad18} the only comments on the proof of Conjecture \ref{c:ggks} are given in \cite[Remark~4.9]{Ad18}; it is only stated that Conjecture \ref{c:ggks} is proved analogously to \cite[Corollary~4.8]{Ad18}, i.e. to the particular case of Conjecture \ref{c:ggks} for $M$ a $2k$-dimensional rational sphere in $\R^{2k+1}$ satisfying the hard Lefschetz property.  
It is not explained there how to get rid of these conditions, which presumably are used in the proof of \cite[Corollary~4.8]{Ad18} (at least it is not stated in \cite[Remark~4.9]{Ad18} that they are not used).   
\end{remark}

\begin{conjecture}[The K\"uhnel conjecture for simplicial embeddings]\label{c:kuh-simp}
\begin{remarkenumi}
\item\label{en:c-kuh-simp-est}    
If some triangulation of a $2k$-manifold $M$ has a subcomplex isomorphic to $\Delta_n^k$, then
\[
    \binom{2k + 1}{k + 1} b_k(M) \ge \binom{n-k-1}{k+1}. 
\]

\item\label{en:c-kuh-simp-asym}
\emph{(The asymptotic version.)}
Under the conditions of~(\ref{en:c-kuh-simp-est}) for any $k\ge1$ there is $c_k>0$ such that
\[
    b_k(M) \gtrsim c_k n^{k+1} \quad\textit{as}\quad n\to\infty. 
\]
\end{remarkenumi}
\end{conjecture}

\begin{remark}[Deduction of Conjecture~\ref{c:kuh-simp}.\ref{en:c-kuh-simp-asym} from Conjecture~\ref{c:ggks}]
\label{r:simp-asym-deduc}
    By Conjecture~\ref{c:ggks}
    \[
        \binom{2k+1}{k+1}b_k(M)  \ge    \gamma(\Delta_n^k) = \binom {n+1}{k+1}-(k+2)\binom {n+1}k \sim \frac{n^{k+1}}{(k+1)!}.
    \]
\end{remark}

\begin{remark}\label{r:helpadi}
\begin{remarkenumi}
\item
Conjecture~\ref{c:kuh-simp}.\ref{en:c-kuh-simp-est} is stated as a result in \cite[Remark~4.9]{Ad18}: \emph{`...
This [Conjecture~\ref{c:ggks}] implies at once another conjecture of Kühnel [Küh94]: 
if a complete $k$-dimensional complex on $n$ vertices embeds into $M$ sufficiently tamely (so that it extends to a triangulation of $M$), then 
$\binom{n-k-1}{k+1} \le \binom{2k + 1}{k + 1} b_k(M)$'}. 

Here `$n$ vertices' should be changed to `$n+1$ vertices', cf.~Conjecture \ref{c:kuh}.

The \emph{`implies at once'} in the above quotation from \cite[Remark~4.9]{Ad18} is unclear and is presumably wrong. 
Conjecture~\ref{c:ggks} implies not Conjecture~\ref{c:kuh-simp}.\ref{en:c-kuh-simp-est} but only its asymptotic version~\ref{c:kuh-simp}.\ref{en:c-kuh-simp-asym}. 
(More precisely, the natural `at once' deduction presented in Remark~\ref{r:simp-asym-deduc} gives only the asymptotic version as explained in Remark~\ref{r:helpadi}.\ref{en:r-helpadi-wr}, and the paper \cite{Ad18} does not contain any `at once' deduction.)
The asymptotic version~\ref{c:kuh-simp}.\ref{en:c-kuh-simp-asym} is strong enough a conjecture. 
Still, this situation illustrates negligence of \cite{Ad18} in claiming famous conjectures. 
We do not question K.~Adiprasito's priority for the deduction presented in Remark~\ref{r:simp-asym-deduc}.

\item\label{en:r-helpadi-wr}
Conjecture~\ref{c:kuh-simp}.\ref{en:c-kuh-simp-est} (even with a typo $n \to n+1$ from \cite[Remark~4.9]{Ad18}) does not follow from Conjecture~\ref{c:ggks} by the argument of~Remark~\ref{r:simp-asym-deduc}.
Indeed, 
\[
    \binom n{k+1}-(k+2)\binom nk < \binom {n+1}{k+1}-(k+2)\binom {n+1}k < \binom{n-k-1}{k+1}
\]
for $n$ large compared to $k$.
Here the first inequality is obvious; let us prove the second one.  
Let $N:=n+1$. 
Both parts of the equivalent inequality 
\[
    N(N-1)\ldots(N-k+1)(N-k-(k+1)(k+2)) < (N-k-2)(N-k-2)\ldots(N-2k-2)
\]
are unitary polynomials in $N$ of degree $k+1$. 
For the coefficients of $N^k$ we have
\begin{multline*}
    1+\ldots+(k-1)+k+(k+1)(k+2) = \frac{k(k+1)}{2} +(k+1)(k+2) =
    \\ = (k+1)\frac{(k+2)+(2k+2)}2 = (k+2)+(k+3)+\ldots+(2k+2). 
\end{multline*}
For the coefficients of $N^{k-1}$ we have
\begin{multline*}
    \sum_{i,j=1}^k ij + \sum_{i=1}^{k-1} i(k+1)(k+2) =
    \frac{k^2(k+1)^2}{4} + \frac{k(k-1)(k+1)(k+2)}{2} = \\ =
    \frac{k+1}{4}(3k^3- k^2 - 2k) <
    \frac{k+1}{4}9k^3 <
   \frac{(k+1)^2(k+2+2k+2)^2}{4} = \sum_{i,j=k+2}^{2k+2} ij.
   \qed
\end{multline*}

\item\label{en:r-helpadi-comp}
The asymptotic version~\ref{r:hist:complex}.\ref{en:r-hist-asym} of the K\"uhnel conjecture on embeddings follows (analogously to the argument of~Remark~\ref{r:simp-asym-deduc}) from Conjecture~\ref{c:ggks} and the inequality $\gamma(K) \geq \gamma(\Delta_n^k)$ for any subdivision $K$ of $\Delta_n^k$.
However, this inequality is not clear.\footnote{The $k$- and $(k-1)$-skeleta of $K$ are larger than those of $\Delta_n^k$, so we only have $f_k(K)\ge f_k(\Delta_n^k)$ and $f_{k-1}(K)\ge f_{k-1}(\Delta_n^k)$, which does not imply the inequality from~\eqref{en:r-helpadi-comp}. 
The inequality from~\eqref{en:r-helpadi-comp} is clear when $K$ is obtained from $\Delta_n^k$ by subdivision of an edge, but is not clear for subsequent subdivisions.}
\end{remarkenumi}
\end{remark}
 
\begin{conjecture}\label{c:adi}
If a simplicial complex embeds into a manifold $M$, then
some triangulation of $M$ has a subcomplex isomorphic to the complex.
\end{conjecture}

Let us illustrate Conjecture~\ref{c:adi} by showing that its analogue for a graph $M$  (which is not a manifold) is wrong. 
Define the graph $M$ to be the union of the cycle on $4$~vertices, and leaf edges added to each vertex of the cycle. 
Then $K_3$ embeds into $M$ but no triangulation of $M$ has a subgraph isomorphic to $K_3$. 

\begin{remark}\label{r:histc-adpa} 
(a) Conjecture \ref{c:adi} is stated as a result in \cite[Theorem 2]{AP24}. 
For criticism of \cite{AP24} see important specific remarks (b,c,i) and minor specific remarks (d--h) below. 
In all these remarks (b--h) we refer to abstract, introduction, lemmas, etc. from arXiv versions~2 and~3 of \cite{AP24} (all the remarks are relevant to the published version).

We are grateful to K.~Adiprasito, S.~Melikhov, P.~Patak, B.~Sanderson, and M.~Tancer, e-mail exchange with whom helped us to conclude that Conjecture~\ref{c:adi} was open in April~2024 (and so was open in~2022 when the first versions of the current paper appeared on arXiv).\footnote{\label{f:refe} This contradicts a referee report of March 2024 (on the shortened arXiv version~3 of the current paper). 
The report wrongly treated Conjecture \ref{c:adi} as trivial, and so wrongly claimed that the K\"uhnel conjecture~\ref{c:kuh} for PL embeddings is solved by \cite{Ad18}.  
Based only upon this wrong claim, the report recommended rejection.} 
So Conjecture~\ref{c:adi} remains open until the indicated gaps of \cite{AP24} are filled (or until a proof following another approach is written). 

The fact that the paper \cite{AP24} is published does not change the situation. 
The paper was not properly refereed, because the referee did not give even those specific critical remarks which are obvious (even to a non-specialist in the area), see e.g. (d).  
Now that the responsibility for the claim is shared by the journal, we sent our critical remarks to the Editors (see Remark \ref{r:ed-let}). 
Having their answer (or having none) would expose the reliability standards of the journal. 

(b) In p. 3, proof of Lemma 6, it is not proved that, and it is not clear why\footnote{\label{f:clever} (added in this arXiv version 8 comparatively to version 6) We are grateful to a clever mathematician (who expressed his wish to remain anonymous) for his efforts to prove the statements highlighted in (b), and to the Editors of Bull. LMS for their public reply to (b) (see Remarks \ref{r:aplms}.b and \ref{e:ap-blms}.b). 
These efforts and reply helped us to conclude that leaving the statements highlighted in (b) without proof is not up to the publicly asserted reliability standards of refereed journals, cf. \cite{Sk21d}.} 

$\bullet$ $\text{st}_\sigma \Delta'$ intersects $\Gamma$ exactly by $\sigma_1\cap\Gamma$. 

$\bullet$ $\sigma_1\cap\Gamma$ is a face of $\Delta$, which is required in the phrase `Since $\Gamma$ is induced in $\Delta$, the latter [$\sigma_1\cap\Gamma$] is a face [of $\Gamma$]'.


(c) The introduction hides the claim for (the asymptotic version~\ref{r:hist:complex}.\ref{en:r-hist-asym} of) the K\"uhnel conjecture on embeddings.
(Recall that this version follows from Conjectures~\ref{c:kuh-simp}.\ref{en:c-kuh-simp-asym} and~\ref{c:adi}.)
This claim was put forward at the beginning of April~2024 by K.~Adiprasito\footnote{\label{f:letadi} In order to avoid confusion, here we present our April 14, 2024 letter to K. Adiprasito. 
Conjectures~1.3.f and~4.1.a of this letter are Conjectures~\ref{c:adi} and~\ref{c:ggks} of the current paper. 
\newline
\emph{Dear Karim, 
\newline
We wish you all the best for proving Conjecture~4.1.a (for $2k$-manifolds $M$ non-embeddable into $\R^{2k+1}$) and Conjecture~1.3.f from our paper attached. 
This would be an outstanding result of yours, because this would imply the K\"uhnel conjecture on embeddings (except topological embeddings for $k=2$). 
We would be glad to refer in our paper to arXiv update of [Ad18], or to a new arXiv paper.
We encourage you to put your paper on arXiv whenever you feel your text is ready for praise and for criticism. 
\newline
In our opinion, making a claim for Conjecture~1.3.f upon the text you sent us on April 6, 2024 will jeopardize your reputation. 
You will presumably realize this by critical reading of your text, so there is no need to send you our specific critical remarks (also, your letter does not ask for them). 
However, we would be glad to present critical remarks to (or praise) any text publicly available on arXiv, and relevant to our paper. 
ArXiv publication (which could never be completely removed) allows one to bear responsibility for a claim, which is necessary for development of mathematics. 
So in order to avoid confusion, unfortunately we would have to delete without reading your letters making a claim for Conjecture~1.3.f. 
But we would feel obliged to publicly react to an arXiv update of [Ad18], or to a new arXiv paper making such a claim. 
\newline
Best wishes, Arkadiy, Slava.}}, cf.~footnote~\ref{f:refe}. 
Instead of presenting this claim, the introduction 
reads: 
\emph{`Some people may consider this theorem [Theorem~2] as a folklore'}. 
This is misleading because this creates the impression that the proof of Conjecture~\ref{c:adi} in \cite{AP24} is a minor folklore result, not a claim for a famous conjecture, so the proof need not be carefully checked, and potential critical remarks on the proof only concern minor technicalities. 
The proof is quite technical, so the critical remarks on the proof are necessarily technical.  
(For an intuitive remark see Remark~\ref{r:hist:complex}.\ref{en:r-histc-wrong}.)

 
(d) In the abstract `a generalization of Istv\'an F\'ary's celebrated theorem' is misleading because F\'ary's theorem does not follow from the main result (Theorem 2). 
In the title `A higher-dimensional version' is misleading because the main result (Theorem 2) 
is different from F\'ary's theorem even in low dimensions. 
Moreover, the main result is a `new triangulation version' of F\'ary's theorem (even in low dimensions) because the affine-on-every-simplex embedding of F\'ary's theorem is replaced by an embedding \emph{simplicial in some new triangulation of the range}.   

(e) In p. 3, proof of Lemma 7, `the obvious fact' is not proved and needs to be proved.

(f) (updated in this arXiv version 7 comparatively to version 6)
In p. 1, Example 3, `A flat neighborhood with respect to' is undefined and is unclear; the explanation `the embedding is homeomorphic to the standard embedding' is meaningless because a homeomorphism a relation between spaces not maps.  


(g) In p.~1, the last line, `the problem' is unclear because no problem was posed. 


(h) In p.~3, statement of Lemma~6, $\Delta$ is not defined, and no agreement `in this paper $\Delta$ is...' is given before. 
Presumably `subcomplex of $\Delta$' should be replaced by `subcomplex of a simplicial complex $\Delta$'.  

\begin{figure}[ht]
    \centering
    \begin{multicols}{2}
        \resizebox{40mm}{!}
        {
        \begin{tikzpicture}
        
            \def \dx {1}
            \def \dy {1}
            
            \coordinate (zero) at (\dx,\dy);
            \coordinate (one) at (0,0);
            \coordinate (two) at (\dx,2*\dy);
            \coordinate (three) at (2*\dx,0);
            \coordinate (four) at (\dx,3*\dy);
            \coordinate (five) at (0,3*\dy);
            \coordinate (six) at (2*\dx,3*\dy);
        
            \draw (zero) node [below] {$0$};
            \draw (one) node [left] {$1$};
            \draw (two) node [left] {$2$};
            \draw (three) node [right] {$3$};
            \draw (four) node [above] {$4$};
            \draw (five) node [left] {$5$};
            \draw (six) node [right] {$6$};
                    
            \fill[black] (zero) circle [radius=1pt];
            \fill[black] (one) circle [radius=1pt];
            \fill[black] (two) circle [radius=1pt];
            \fill[black] (three) circle [radius=1pt];
            \fill[black] (four) circle [radius=1pt];
            \fill[black] (five) circle [radius=1pt];
            \fill[black] (six) circle [radius=1pt];

            \draw[black,-,thin] (zero) -- (one);
            \draw[red,-,thin] (zero) -- (two);
            \draw[black,-,thin] (zero) -- (three);
            \draw[black,-,thin] (two) -- (one);
            \draw[black,-,thin] (three) -- (two);
            \draw[black,-,thin] (one) -- (three);
            \draw[black,-,thin] (one) -- (five);
            \draw[black,-,thin] (six) -- (three);
            \draw[black,-,thin] (four) -- (five);
            \draw[black,-,thin] (four) -- (six);
            \draw[red,-,thin] (four) -- (two);
            \draw[black,-,thin] (two) -- (five);
            \draw[black,-,thin] (two) -- (six);
        \end{tikzpicture}
        }

        \resizebox{40mm}{!}
        {
        \begin{tikzpicture}
        
            \def \dx {1}
            \def \dy {1}
            
            \coordinate (zero-two) at (\dx,1.5*\dy);
            \coordinate (one) at (0,0);
            \coordinate (three) at (2*\dx,0);
            \coordinate (four) at (\dx,3*\dy);
            \coordinate (five) at (0,3*\dy);
            \coordinate (six) at (2*\dx,3*\dy);
        
            \draw (zero-two) node [left] {$02$};
            \draw (one) node [left] {$1$};
            \draw (three) node [right] {$3$};
            \draw (four) node [above] {$4$};
            \draw (five) node [left] {$5$};
            \draw (six) node [right] {$6$};
                    
            \fill[black] (zero-two) circle [radius=1pt];
            \fill[black] (one) circle [radius=1pt];
            \fill[black] (three) circle [radius=1pt];
            \fill[black] (four) circle [radius=1pt];
            \fill[black] (five) circle [radius=1pt];
            \fill[black] (six) circle [radius=1pt];

            \draw[black,-,thin] (zero-two) -- (one);
            \draw[black,-,thin] (zero-two) -- (three);
            \draw[black,-,thin] (one) -- (three);
            \draw[black,-,thin] (one) -- (five);
            \draw[black,-,thin] (six) -- (three);
            \draw[black,-,thin] (four) -- (five);
            \draw[black,-,thin] (four) -- (six);
            \draw[red,-,thin] (four) -- (zero-two);
            \draw[black,-,thin] (zero-two) -- (five);
            \draw[black,-,thin] (zero-two) -- (six);
        \end{tikzpicture}
        }
    \end{multicols}
    \caption{The (red) complex $K = (0-2-4)$ in the triangulation $M$ (left). 
The (red) complex $K'=(02-4)$ in the triangulation $M'$ (right). 
Both $K',M'$ are obtained from $K,M$ by valid contraction of the edge $(0-2)$. 
The complex $K$ is obtained from $K'$ by edge subdivision, but $M$ cannot be obtained from $M'$ by edge subdivision.} 
    \label{fig:valid-edge-cont}
\end{figure}

(i) (added in this arXiv version 7 comparatively to version 6) 
P.~4, proof of Theorem~2. 
The proof does not explicitly use the assumption that $M$ is a PL manifold, although the analogue of Theorem~2 for complexes that are not manifolds is incorrect.
Presumably, the usage of this assumption is hidden 
in the phrase `It follows that the resulting triangulation $M'$ of $M$ is the desired one'. 
In this phrase it is not proved that the obtained complex $M'$ is PL homeomorphic to $M$, i.e. is indeed a triangulation of $M$.
This could be proved using the following conjecture (whose analogue for complexes $K$ is incorrect): \emph{the complex obtained from a triangulation $K$ of a PL manifold by valid edge contraction is PL homeomorphic to $K$}. 
The paper gives no proof or reference to this conjecture\footnote{Even if this is only lack of an explicit statement and a reference, not a non-trivial conjecture used without an explicit statement and a proof, this contributes to gradual loss of knowledge of PL topology, of which the authors complain before Example~3. 
\newline
For a reader's convenience recall the used result of Alexander [2, Corollary 10:2d and page~299] and Newman [6, pages~187--188] (we use the citation as in \cite{AP24}): \emph{any two of PL homeomorphic complexes can be obtained one from the other by a sequence of edge subdivisions and inverse operations}, cf.~\cite[page~2]{Go13}.} 
(and no other proof that $M'$ is PL homeomorphic to $M$). 

The conjecture is non-trivial because there is a valid edge contraction $\Psi$ in a triangulation of a PL manifold such that $\Psi$ is not an inverse to edge subdivision, although the restriction of $\Psi$ to some subcomplex is an inverse to edge subdivision in the subcomplex (Figure~\ref{fig:valid-edge-cont}). 
  
\end{remark}

There are complexes PL embeddable into $\R^d$ but for which there are no embeddings into $\R^d$ affine on any simplex \cite{vK41, PW}. 
See stronger results in \cite[Theorems 1.4, 2.1, 2.4]{FK13}. 
Although these results do not refute Conjecture \ref{c:adi} (cf. Remark \ref{r:histc-adpa}.d), it is interesting if the complexes can be used to construct counterexamples to Conjecture~\ref{c:adi}. 

\begin{remark}\label{r:aplms} 
Below we present the Editors' reply to the criticism of Remark \ref{r:histc-adpa} (technically speaking, to the letter of Remark \ref{r:ed-let}.a). 
The Editors kindly allowed to publish their reply (cf. Remark \ref{r:ed-let}.b). 
This cannot but deserve respect.    
However, we are sorry to justify in Remark \ref{e:ap-blms} that the reply is not up to the reliability and clarity standards expected from a refereed journal (see Example \ref{e:our-blms}). 
References in the reply are to \cite{AP24}, except the place where we added 
`[DS: of arXiv:2208.04188v7]'. 

\textbf{Apr 12, 2025} \quad \emph{Dear Professors Skopenkov and Dzhenzher,}

Many thanks for your patience. Following receipt of your message dated 31 March concerning the paper 'A higher dimensional version of Fáry's theorem’ [Bull. London Math. Soc., February 2025], we consulted the Section Editor for Geometry and Topology, the Handling Editor of the paper, and the expert referee. 
They do not concur with your assessment that there are important gaps in the proof of Theorem 2 of the aforementioned article.

Specifically, regarding the mathematical portions of Remark 6.7 
in arXiv:2208.04188v7 (note that it is this version we have examined), we have agreed upon the following responses:

(b) The first bullet point follows from the description of biased barycentric subdivision (sometimes called partial barycentric subdivision) in terms of chains given in the proof of the lemma, and illustrated in Figure 1.2, while the second bullet point follows from the definition of induced subcomplex (usually called full subcomplex). 

(c) This paper itself is making no direct claims about the Kuehnel conjecture. 
Separately, the question that is addressed in the main result of the paper does not appear to be a `famous conjecture'. 

(d) It is indeed a simplicial rather than an affine version, but nevertheless very much a Fáry-type theorem. 
As such, the title (and the other phrases referred to) are acceptable (as a certain level of critical engagement can be expected of the reader).

(e) The ‘obvious fact' follows from the definition of 'strongly induced subcomplex' (though some further details may indeed have helped the reader).

(i) "collapsing a valid edge in a PL manifold is a PL homeomorphism'' is an elementary fact in PL topology, not a conjecture. 
(The fact that "not every edge contraction is the inverse of an edge subdivision'' does not change this.) 

Regarding the paragraph preceding Remark 6.8 [DS: of arXiv:2208.04188v7], there is no reason to believe that the examples of complexes that embed PL but not affinely in $\R^d$ would lead to counterexamples to Theorem 2. 

Overall, we have found the exposition of the paper  to be within the bounds of clarity and correctness that we normally expect of publications in the journal. 

In this instance, it is our collective view that the authors of [AP25] do not have a case to answer. 
We regard the matter as closed.

Sincerely,

Minhyong Kim and Julia Wolf
\end{remark}


\begin{example}[on standards of reply to a public criticism, cf. \cite{Sk21d}]\label{e:our-blms} 
The reliability standards expected from refereed journals 
require the following. 

A reply to a public requirement (from a scientist having some decent papers in the area, e.g. in PL topology) of a reference to a statement should give either a reference to (a specific result in) a textbook, or a several-line deduction from (specific results in) a textbook. 

A reply to a public criticism of a proof by a scientist (having some decent papers in the area) should assume that the scientist have read the proof. 
So the scientist's critial remark `\emph{it is not proved that, and it is not clear why [a statement]}' means `\emph{it is not proved that, and it is not clear why [a statement], even taking into account the definitions of the objects involved in the statement, and the previous sentences of the proof}'.\footnote{The burden of proof lies with those who assert the statement.  
\newline
By public requirement of details of the proof (in particular, of a reference) a scientist takes the reputation risk of missing a trivial proof (in particular, a reference in a textbook). 
So it is only natural that those who assert the statement take the risk of serving scientific community, either by providing necessary details of the proof, or by seeing that the missing argument they have in mind is fallible (in particular, that the necessary result is not present in textbooks). 
\newline 
See e.g. post-publication arxiv updates of \cite{Sk05, Sk08, GS06} (prepared after private not public criticism). 
}  

In particular, nothing of the following can contribute to a reliable justification of a statement.   

(a) The information that the statement follows from 

$\bullet$ the definitions of the objects involved in the statement, or

$\bullet$ the previous argument, even if the relevant two of the four sentences of the argument are specified.   

See e.g. Remark \ref{r:aplms}.b.  

(b) A figure, especially a low-dimensional figure in the proof of a higher-dimensional statement. 
See e.g. Remark \ref{r:aplms}.b.  

(c) An information on different terminology. See e.g. Remark \ref{r:aplms}.b.   

(d) Naming the statement elementary, without giving either a reference or a deduction. 
(If this comes from a person or from an organization of authority, then this a logical fallacy `argument from authority' \cite{AA}.) See e.g. Remark \ref{r:aplms}.i.   

(e) A speculation that the non-triviality of the statement is not sufficiently illustrated by some given fact. 
See e.g. Remark \ref{r:aplms}.i, the sentence in parenthesis.  

(f) Lack of a counterexample to the statement (or lack of reasons to believe that some ideas to construct a counterexample work). 
See e.g. Remark \ref{r:aplms}, after (i).    

(g) A `straw man' logical fallacy  \cite{SM}. 
See e.g. Remark \ref{e:ap-blms}.cj. 

\smallskip
If a reply to the above public requirement / critisism states that the authors of a criticized text do not have a case to answer, then the replying authors / organization should publish a reference to the criticism. 
This would show that the statement `do not have a case to answer' stands the test of being read together with the critisism. 
Lack of publication of a reference to the criticism is \emph{suppression of criticism}. 
Such lack of publication shows that the authors of a criticized text do have a case to answer, but attempt to hide the fact that their reliability standards are lower than those expected from refereed journals. 

\end{example}

\begin{remark}\label{e:ap-blms} 
(a) In (b,c,i,j) we justify that the reply from Remarks \ref{r:aplms}.bci (and from the sentence after (i)) to the critical Remark \ref{r:histc-adpa} is not up to the reliability and clarity standards described in Example \ref{e:our-blms}. 
Thus the (important) critical Remarks \ref{r:histc-adpa}.bi remain unanswered.   

In (c,d) we show that the clarity standards (see \cite[footnote 1]{Sk21d}) of Bull. LMS leave it to the readers to cope with as misleading title and introduction as 

$\bullet$ hiding the claim for a famous conjecture;\footnote{The result of such a hiding is distraction of people's attention from reliability of the proof, and from unanswered public criticism of the proof, cf. Remark \ref{r:ggks}.} 

$\bullet$ giving name `a \emph{higher-dimensional} version of theorem A' to a result which is different from theorem A in \emph{low dimensions}. 


In the items below, the necessary material from Remark \ref{r:histc-adpa} is shortly recalled (so the justification can be read without reading Remark \ref{r:histc-adpa}, unless a reader wants to check that our references to Remark \ref{r:histc-adpa} are faithful). 

(b) Remark \ref{r:histc-adpa}.b asked for a proof of two statements (bullet points) in the proof of \cite[Lemma 6]{AP24}. 
The reply (Remark \ref{r:aplms}.b) consists only of 

$\bullet$ the information that the first statement follows from the previous two sentences (of the previous four sentences) in the proof of \cite[Lemma 6]{AP24}: `\emph{The first bullet point follows from the description of biased barycentric subdivision in terms of chains given in the proof of the lemma...}';  

$\bullet$ the information that `\emph{the second bullet point follows from the definition of induced subcomplex}' (that is, from the definition of the objects involved in the second bullet point);


$\bullet$ an information on different terminology: `\emph{(sometimes called partial barycentric subdivision)}',  `\emph{(usually called full subcomplex)}'; 

$\bullet$ a reference to a low-dimensional figure: `\emph{and illustrated in Figure 1.2}'.  

See Example \ref{e:our-blms}.abc.\footnote{Not only we have carefully read the proof, but we discussed possibility of recovery of the unproved statements with mathematicians of different nationalities and backgrounds, see footnote \ref{f:clever}.} 

(c) Remark \ref{r:histc-adpa}.c states and justifies that `\emph{The introduction hides the claim for ... the K\"uhnel conjecture on embeddings}'.
The reply (Remark \ref{r:aplms}.c) ignores the raised problem of `hiding'.  
In particular, the phrase `\emph{the question that is addressed in the main result of the paper does not appear to be a `famous conjecture'}' is the `straw man' logical fallacy \cite{SM} (indeed, the critical Remark~\ref{r:histc-adpa}.c does not assert the opposite, i.e that `the question that is addressed in the main result of the paper is a `famous conjecture''). 

(d) The reply (Remark \ref{r:aplms}.d) agrees with the critical Remark~\ref{r:histc-adpa}.d (although the reply uses different wording). 

(i) Remark~\ref{r:histc-adpa}.i asked for a reference or a proof of a statement used (without being explicitly stated) in the proof of \cite[Theorem~2]{AP24}. 
The reply (Remark \ref{r:aplms}.i) does not provide a reference or a proof, but only states that the statement `\emph{is an elementary fact in PL topology}'. 
(The parenthetical remark in the reply uses this unjustified statement as correct but does not contribute to its justification.). 
See Example \ref{e:our-blms}.de.\footnote{See   \url{https://old.mccme.ru//circles//oim/home/LIBRARY.pdf}. 
We certainly read \cite{RS72}, and before asking for the reference consulted it once again.}  

(j) The paragraph preceding Remark 6.8 is not a part of critical Remark \ref{r:histc-adpa}. 
So its discussion does not contribute to answering the criticism, and its inclusion into a reply to critical Remark~\ref{r:histc-adpa} is the `straw man' logical fallacy \cite{SM}.
See Example \ref{e:our-blms}.fg. 
\end{remark}

\small

\emph{Conventions for the following remarks.}
In the letters the references are updated.
In spite of that, the references are given to the very version of a paper cited in the letter.
Otherwise the letters are not changed (the grammar is not corrected).

\begin{remark}\label{r:ap-let} 
In order to avoid confusion, we present here our letter to the authors of \cite{AP24}, and our replies to the letters from J.~D.~Paik and K.~Adiprasito on our criticism of (arXiv version~2 of) \cite{AP24}.  
(They did not give us their permission to publish their letters.)

\vspace{2pt}
(February 25, 2025) \quad \emph{Dear Zuzana, Dear Karim,} 

Attached please find a planned update of arXiv:2208.04188 [added later: this update is arXiv:2208.04188v6]. 
We would be grateful for any comments or remarks, but please do not feel obliged. 

Remark 6.7 contains some criticism of your interesting paper arXiv:2404.12265v2. 
It would be nice if you could update your paper on arXiv according to this criticism. 

In order to make this discussion responsible, let us make it public. 
We perpetually release copyright for this letter to the public domain. 
If you kindly choose to reply to this letter, not just update your paper on arXiv, could you please perpetually release copyright for your reply to the public domain.

Sincerely Yours, Slava and Arkadiy. 

\vspace{2pt}
(March 8, 2025) \quad \emph{Dear Joshua David,}

Thank you for your letter disagreeing with our (arXiv:2208.04188v6) criticism 
of an unjustified claim to a `new triangulation version' of F\'ary's theorem (Conjecture 6.6 from arXiv:2208.04188v6), 
and mentioning Dima Burago, Yuri Burago, and Anton Petrunin. 
Our criticism is justified by specific critical remarks of Remark 6.7 from arXiv:2208.04188v6. 
Your letter does not disagree with any of those specific critical remarks. 
So in our opinion, your letter jeopardizes your reputation as a scientist. 

If your opinion is different, we welcome its responsible pronunciation. 
In order to make this discussion responsible, we have to make it public.
We encourage you to put your disagreement to arXiv, or to perpetually release copyright for your letter to the public domain.  
We would be glad to present critical remarks to (or praise) any text publicly available on arXiv (or in the public domain by author's perpetually releasing copyright), and relevant to our paper.

Sincerely Yours, Slava and Arkadiy.

\vspace{2pt}
(March 8, 2025) \quad
\emph{Karim,}
 
This is libel, so please make it public so that you would assume responsibility for the libel. 

A. 

\vspace{2pt}
(March 9, 2025) \quad
\emph{Dear Karim,}

Thank you for forwarding us your letter to Josh naming our criticism (arXiv:2208.04188v6, of an unjustified claim to a `new triangulation version' of Fary's theorem) `incompetence' . 
Our criticism is based on specific critical remarks of Remark 6.7 of arXiv:2208.04188v6. 
Your letter contains no criticism of those specific critical remarks. 
So, in our opinion, your letter jeopardizes your reputation as a scientist. 

If your opinion is different, we welcome its responsible pronunciation. 
In order to make this discussion responsible, we have to make it public.
We encourage you to put your disagreement to arXiv, or to perpetually release copyright for your letter to the public domain.  
We would be glad to present critical remarks to (or praise) any text publicly available on arXiv (or in public domain by author's perpetually releasing copyright for the text), and relevant to our paper.

Sincerely Yours, Slava and Arkadiy.
\end{remark}

\begin{remark}\label{r:ed-let} 
In order to avoid confusion, we present here our letters to the Editors of the journal where the paper \cite{AP24} was published. 

\vspace{2pt}
\textbf{(a) March 31, 2025}

Dear Managing Editors Minhyong Kim and Julia Wolf, Dear Editors,  

Hope you are fine and healthy. 

We write this letter on the understanding that LMS strives to ensure that its journals' contribution to the published record is reliable 
(\url{https://www.lms.ac.uk/publications/policies/ethicalpolicy}). 
We feel obliged to inform you that Theorem 2 claimed in 

[AP24] := \url{https://londmathsoc.onlinelibrary.wiley.com/doi/10.1112/blms.70036}  

is not reliable, and the introduction is misleading. 
See justification by specific critical Remarks 6.7 (b,c,d,e,g,h) of 
[DS] := arXiv:2208.04188v6.   
In these remarks we refer to the arXiv:2404.12265v3 version of [AP24]; all the remarks are relevant to the published version.
See also Remark 6.7.i from the attached update of [DS] (added later: this update is arXiv:2208.04188v7). 

Now that the responsibility for the claim is shared by the journal, it would be nice if the Editors could ask for the authors' and the referees' answers to these specific remarks. 
Moderation of such a discussion by the Editors (or just giving reference to arXiv published criticism of the paper published by LMS), would confirm high reliability and clarity standards of the journal.
The result of such a discussion could be either rejection of the paper, or reliable revision of the paper, or arXiv publication of answers to our (arXiv published) specific critical remarks. 
In order not to mislead the readership, either the publication could be suspended, or a reference to the ongoing discussion of reliability of the proof (and clarity of the introduction) could be added.  
We would only be glad if the revision taking into account [DS, Remarks 6.7 (b,c,d,e,g,h)] (and Remark 6.7.i from the attached update of [DS]) could be published.

Sincerely Yours, Arkadiy Skopenkov and Slava Dzhenzher.
 
PS We sent the criticism to the authors of [AP24] even before the v6 update of [DS]. 
After the appearance of the v6 update, we received letters expressing disagreement with our criticism. 
Although our criticism is justified by specific critical remarks, disagreement with our criticism did not consider any of those specific critical remarks. 
In order to avoid confusion, we present our replies to the disagreement in Remark 6.8 from the attached update of [DS].
Remark 6.9 of that update is a copy of this letter. 
We plan to publish the information on the Editors' reply (or on the lack of it).

PSS We do understand that the paper [AP24] was refereed. 
However, the referee did not give even those specific critical remarks which are obvious (even to a non-specialist in the area), see e.g. [DS, Remark 6.7.d]. 
So it is not surprising that a referee missed a gap in the (quite technical) proof. 

\vspace{2pt}
\textbf{(b) April 13, 2025}

Dear Minhyong Kim and Julia Wolf, 

Many thanks for your reply to our critical remarks on the paper

[AP25] := [Adiprasito-Patakova, Bull. London Math. Soc., February 2025]. 

The reply is very interesting to the math community, as a revelation of the bounds of clarity and correctness that the Editors of Bull. LMS journals normally expect of publications in the journal. 
So would you kindly allow us to publish the reply? 
Technically speaking, would you release the copyright for the reply to the public domain? 
Please feel free to make any changes in your reply before making it public. 

In an update of  arXiv:2208.04188 we shall either add your reply, or this letter and the information that you did not allow us to make the reply public. 
The first option would be a sound confirmation that our critical remarks on [AP25] are so far answered in a responsible way. We are afraid that the second option would be a sound confirmation that our (openly published in arXiv:2208.04188) critical remarks on [AP25] are justified, and cannot be answered in a responsible way. 

Sincerely Yours, Arkadiy Skopenkov and Slava Dzhenzher.

\vspace{2pt}
\textbf{(c) June 19, 2025}

Dear Managing Editors Minhyong Kim and Julia Wolf,

Hope you are fine and healthy.

Thank you for allowing us to publish your reply to a criticism of a paper published in Bull. LMS,  

[AP24] := \url{https://londmathsoc.onlinelibrary.wiley.com/doi/10.1112/blms.70036} 
\linebreak
(=arXiv:2404.12265v3). 

This cannot but deserve respect. 
However, we are sorry to justify (in Remark 6.10) that the reply (Remark 6.8) is not up to the reliability and clarity standards expected from a refereed journal (see Example 6.9).
The numbering is from the attached update of arXiv:2208.04188. 
Unless we receive your public reply, we plan to make the update publicly available in a week. 
Your public reply to Example 6.9 and Remark 6.10 is very interesting to the math community, as a revelation of the bounds of clarity and correctness that the Editors of Bull. LMS journals normally expect of publications in the journal. 

We value our collaboration with the London Mathematical Society very much (the more so because Arkadiy is grateful to LMS for the 2002 invitation to deliver research lectures in Great Britain). 
More specifically, we would be pleased 

$\bullet$ to postpone arXiv update of arXiv:2208.04188, and 

$\bullet$ to publish in a later revised update, instead of your reply (Remark 6.8) and its criticism (Remark 6.10), a revised version of your reply (which hopefully would not require such a harsh criticism). 

For this one needs LMS to publish on the journal's website the information on the ongoing discussion whether the reliability (and clarity) standards of [AP24] are up to the standards of LMS, 
with a link to the published criticism of [AP24] in \S6 of arXiv:2208.04188v7.

This would be a revelation of the high transparency standard of LMS (see arXiv:2101.03745v3, Remark 4.2.a), and its respect to responsible professional criticism.  
 
Sincerely, Arkadiy Skopenkov and Slava Dzhenzher.
\end{remark}

\normalsize

}

\aronly{
 
\section{Appendix: on the property of  non-triviality}\label{s:appnotr}
 
Recall that for a general position map $g\colon [3]^{*k+1}\to\R^{2k}$ the \emph{van Kampen number} $v(g) \in \Z_2$ is the parity of the number of all unordered pairs $\{\sigma,\tau\}$ of vertex-disjoint $k$-faces of $[3]^{*k+1}$ such that $\abs{g\sigma \cap g\tau}_2 = 1$.

 \begin{lemma}[van Kampen; {\cite[Satz 5]{vK32}}]\label{lem:vanKampen}
     For any general position map $g\colon [3]^{*k+1}\to\R^{2k}$ we have $v(g) = 1$.
\end{lemma}

The following proof, except the last paragraph, is alternative to known proofs.\footnote{The first part of the proof (namely, everything except the last paragraph) is construction of {\it some} map $g\colon [3]^{*k+1}\to\R^{2k}$ such that $v(g) = 1$.
    In the original proof van Kampen constructed another such $g$, splitting vertices on groups of three and placing the groups in different hyperplanes.
    Another construction of such a map $g$ is given in \cite[the second paragraph of Example~3.5]{Me06}: take $n=k$ and $n_i=0$ for $i\in[k+1]$; take $*\Delta^{n_i} = 1^{*k+1}$ and $*\partial\Delta^{n_i+1} = [2]^{*k+1}$.}
 
\begin{proof}[Proof of Lemma~\ref{lem:vanKampen}]
    Let $\gamma(t) = (t,\ldots,t^{2k})$ be the moment curve in $\R^{2k}$.
    Let $g\colon [3]^{*k+1}\to \R^{2k}$ be the linear map such that
    \[
        g(\varnothing^{*i}*a*\varnothing^{*k-i}) = \gamma(a+3i)\quad\text{for every}\quad i\in\{0,1,\ldots,k\}\quad\text{and}\quad a\in[3].
    \]
    It is well known that every at most $2k+1$ points on $\gamma$ are affine independent (for proof see e.g.~\cite[Lemma~5]{St24}).
    Then $g$ is a general position map.
    In the following paragraph we prove that $v(g)=1$.
        

    It is known that for vertex-disjoint $k$-faces $\sigma=\sigma_1*\ldots*\sigma_{k+1}$ and $\tau=\tau_1*\ldots*\tau_{k+1}$ their images $g\sigma$ and $g\tau$ intersect (at a single point) if and only if the vertices of the images alternate along the moment curve\footnote{For clear exposition see~\cite[Lemma~6]{St24}; for an earlier reference see \cite[Theorem]{Br73} (in \cite{Br73} the statement of the Theorem has undefined $A$ and $B$, and uses the term `the primitive Radon partition' defined elsewhere).}.
    The alternation means that either
     \begin{align*}
        &\sigma_1 < \tau_1 < 3+\sigma_2 < 3+\tau_2 < \ldots < 3k+\sigma_{k+1} < 3k+\tau_{k+1}\quad\text{or} \\
        &\tau_1 < \sigma_1 < 3+\tau_2 < 3+\sigma_2 < \ldots < 3k+\tau_{k+1} < 3k+\sigma_{k+1}.    
    \end{align*}
    The alternation is equivalent to `either $\sigma_i < \tau_i$ for every $i \in [k+1]$ or $\sigma_i > \tau_i$ for every $i \in [k+1]$'.
    Then 
    \[
        v(g) =
        \abs{\set{(\sigma_1, \ldots, \sigma_{k+1}, \tau_1, \ldots, \tau_{k+1})\in[3]^{2k+2}\ :\ \text{$\sigma_i < \tau_i$ for every $i\in [k+1]$}}}_2 = 1. 
    \]
    Here the last equality is proved as follows. 
    For such $(2k+2)$-tuples $(\sigma_1,\ldots,\sigma_{k+1}, \tau_1,\ldots,\tau_{k+1})$ every pair $(\sigma_i,\tau_i)$ is either $(1,2)$, or $(1,3)$, or $(2,3)$.
    Then by the Cartesian product rule the number of such $(2k+2)$-tuples is $3^{k+1}\equiv 1 \pmod{2}$.

    Proof that $v(g) = v(g')$ for any general position maps $g,g'\colon [3]^{*k+1}\to\R^{2k}$, repeats the fairly standard part in the last paragraph of the proof of the non-triviality in Lemma~\ref{lem:non-triv} (Embedding).
\end{proof}

 

Recall that

\begin{itemize}[nosep]
    \item $H$ is the set of all unordered pairs of vertex-disjoint $k$-faces of $[3]^{*k+1}$;

    \item $G_k$ is the set of unordered pairs of $k$-octahedra from $[3]^{*k+1}$ whose intersection is $1^{*k+1}$;

    \item for $\{P,Q\}\in G_k$ we denote by $T\{P,Q\}$ the set of pairs $\{\alpha,\beta\}$ formed by (not necessary distinct) $k$-faces $\alpha,\beta$ of $[3]^{*k+1}$ such that either $\alpha \in P$ and $\beta \in Q$, or vice versa.
\end{itemize}

\begin{remark}\label{r:deduc-emb-ks21e} 
\begin{remarkenumi}
    \item \cite[Lemma~2.3.1.a and footnote~8]{Sk24}
    \label{en:r-deduc-emb-ks21e:real}
    \emph{If a $(k-1)$-connected $k$-complex $K$ is $\Z_2$ embeddable to a $2k$-manifold $M$, then there are a collection of $k$-cycles $y_\sigma$ in $M$, parametrized by $k$-faces of $K$, and a general position map $g\colon K\to\R^{2k}$ such that}
    \[
        y_\sigma \cap_M y_\tau = |g\sigma\cap g\tau|_2\quad\textit{for any vertex-disjoint $k$-faces $\sigma,\tau$ of $K$}.
    \]

    \item\label{en:r-deduc-emb-ks21e:ded-t}
    \emph{Deduction from~\eqref{en:r-deduc-emb-ks21e:real} of the version (stated in Remark~\ref{r:emb-mod-2}.\ref{en:r-emb-mod-2-res}) of Theorem~\ref{t:basis-emb} for $\Z_2$-embeddings.}
    Take $y_\sigma$ and $g$ from~\eqref{en:r-deduc-emb-ks21e:real}.
    For $k$-octahedra $P,Q$ define
    $A_{P,Q}:=\sum\limits_{(\sigma,\tau)\in P \times Q}y_\sigma \cap_M y_\tau$.
    Take a basis in $H_k(M;\Z_2)$ in which the matrix of $\cap_M$ is $\Omega_M$.
    Let $Y$ be the $\beta_k(M) \times \binom{n}{2}^{k+1}$-matrix whose columns are coordinates of $y_\sigma$ in this basis.
    Then $A = Y^T\Omega_MY$.
    It remains to prove that $A$ is an $(n,k)$-matrix.
     
    The additivity and the symmetry are obvious.
    
    The independence holds since for vertex-disjoint $k$-octahedra $P,Q$ we have
    \[
        A_{P,Q} = \sum\limits_{(\sigma,\tau)\in P\times Q}y_\sigma\cap_M y_\tau = |gP\cap gQ|_2 = 0,
    \]
    where the last equality holds by the Parity Lemma (asserting that every two general position $k$-cycles in $\R^{2k}$ intersect at an even number of points, cf. \cite[\S1.3, \S4.8 `Algebraic intersection number', and Lemma~5.3.4]{Sk}). 
    
    The non-triviality holds since
    \begin{multline*}
        SA=
        \sum_{\{P,Q\}\in G_k}
            A_{P,Q} =
        \sum_{\{P,Q\}\in G_k}
            \sum_{\{\alpha,\beta\} \in T\{P,Q\}}
                y_\alpha \cap_M y_\beta = \\
        \sum_{\{\alpha, \beta\}\in H}
            y_\alpha \cap_M y_\beta =                  
        \sum_{\{\alpha, \beta\}\in H}
            |g\alpha \cap g\beta|_2 =                   1,
    \end{multline*}
    where the third equality holds by Lemma~\ref{l:comb} (Combinatorial), and the last equality is the result of van Kampen \cite[Satz~5]{vK32}, see Lemma~\ref{lem:vanKampen} (van Kampen).
    
    \item\label{en:r-deduc-emb-ks21e:ded-lem} Lemma~\ref{lem:non-triv} (for $\Z_2$-embeddings) is analogously deduced from~(\ref{en:r-deduc-emb-ks21e:real}) with the following addendum\footnote{The addendum is essentially obtained in \cite[\S2.3, Proof of Lemma 2.3.1.a]{Sk24}. We use the notation and the equalities (2), (3) from that proof. 
We obtain 
$$\sum\limits_{(\sigma,\tau)\in P \times Q} y_\sigma\cap_M y_\tau  = \sum\limits_{(\sigma,\tau)\in P \times Q} \left(|f'\sigma\cap f'\tau|_2+|h'\sigma\cap h'\tau|_2\right) = |f'P\cap f'Q|_2+A(h')_{P,Q} = A(h')_{P,Q}.$$ 
Here the last equality follows by the Parity Lemma.}: \emph{if $P$ and $Q$ are $k$-cycles in $K$ and $f\colon K\to M$ is an almost embedding, then $A(f)_{P,Q} = \sum\limits_{(\sigma,\tau)\in P \times Q}y_\sigma\cap_M y_\tau$.}
\end{remarkenumi}
\end{remark}

\begin{remark}[On alternative definition of non-triviality]
    Under the assumptions that $A$ is symmetric, independent and additive, the following is an equivalent definition of the non-triviality (the equvalence is clear from the proof of Lemma~\ref{l:pre-tour-like}).
    The matrix $A$ is said to be \emph{non-trivial} if for any complex $K \subset [n]^{*k+1}$ isomorphic to $[3]^{*k+1}$ and any $k$-face $\alpha\subset K$ the sum $S_{\alpha, K}A = 1$, where $S_{\alpha, K}A$ is the sum of $A_{P,Q}$ over all unordered pairs $\{P,Q\}$ of $k$-octahedra in $[3]^{*k+1}$ such that $P\cap Q=\alpha$.
    
    The analogue of Lemma~\ref{lem:non-triv} (Embedding) for the new definition is correct because every proof of non-triviality presented in the paper works for any subcomplex $K\subset [n]^{*k+1}$ isomorphic to $[3]^{*k+1}$ and any $k$-face $\alpha\subset K$.
\end{remark}

Below we give an alternative proof of non-triviality from Lemma~\ref{lem:non-triv}, which we did not succeed to generalize to $\Z_2$-embeddings.

\begin{proposition}[Intersection formula]
\label{p:inter}
    For any embedding $f\colon [3]^{*k+1}\to M$ into a $2k$-manifold $M$ there is a general position map $g\colon [3]^{*k+1}\to\R^{2k}$ such that for any $k$-octahedra $P,Q$ whose intersection is $1^{*k+1}$
    \[
        fP \cap_M fQ =
        \sum_{\{\alpha,\beta\} \in T\{P,Q\}}
                \babs{g(\alpha-\beta) \cap g(\beta-\alpha)}_2.
    \]
\end{proposition}

\setcounter{mcnt}{0}
\renewcommand{\themcnt}{\text{(\ref{p:inter}.\alph{mcnt})}}

\begin{proof}
    Denote $I := [\,0,1\,] \subset \R$.
    
    By ambient isotopy of $M$ we may assume that the image of $f$ is in the interior of $M$.
    Since $f$ is an embedding and since $1^{*k+1}\cong I^k$ is collapsible, by \cite[Corollary~3.27]{RS72} there is an embedding
    $i\colon I^{2k}\to M$ in general position to $f$, and such that
    \[   
        iI^{2k}\supset f(1^{*k+1})
        \quad\text{and}\quad iI^{2k}\cap f\bigl(\{2,3\}^{*k+1}\bigr)=\emptyset.
    \]
    
    Take any general position map 
    \[
        g\colon [3]^{*k+1}\to\R^{2k}\quad\text{such that}\quad f^{-1}(iI^{2k}) = g^{-1}(I^{2k})=:Z\quad\text{and}\quad f|_Z=ig|_Z.
    \]
    Here the property $f|_Z=ig|_Z$ means that  $f(x)=ig(x)$ when $f(x)\in iI^{2k}$ (or, equivalently, when $g(x)\in I^{2k}$).
    
    Take any general position map $g' \colon [3]^{*k+1} \to \R^{2k}$ such that
    \[
        g|_{g^{-1}(\R^{2k} - I^{2k})} = g'|_{g'^{-1}(\R^{2k} - I^{2k})}
    \]
    (which means that $g^{-1}(\R^{2k} - I^{2k}) = g'^{-1}(\R^{2k} - I^{2k})$, and the restrictions of $g$ and $g'$ to $g^{-1}(\R^{2k} - I^{2k})$ coincide),
    and
    {
    \renewcommand{\theenumi}{(CGP)}
    \renewcommand{\labelenumi}{(CGP)}
    \begin{enumerate}[leftmargin=3cm]
        \item
        \label{en:gp}
        \emph{$g$ and $g'$ are close, and $(g \sqcup g')|_{g^{-1}(I^{2k})}$ is a general position map.}
    \end{enumerate}
    }
    
    Then for any $k$-octahedra $P,Q$ whose intersection is $1^{*k+1}$
    
    \begin{multline*}
        fP \cap_M fQ                                                              \stackrel{\juststep}{=}
        \babs{\bigl(igP \cap iI^{2k}\bigr) \cap \bigl(ig'Q \cap iI^{2k}\bigr)}_2  \stackrel{\juststep}{=} \\
        \babs{\bigl(gP \cap I^{2k}\bigr) \cap \bigl(g'Q \cap I^{2k}\bigr)}_2      \stackrel{\juststep}{=}
        \babs{\bigl(gP - I^{2k}\bigr) \cap \bigl(g'Q - I^{2k}\bigr)}_2            \stackrel{\juststep}{=} \\
        \babs{\bigl(gP - I^{2k}\bigr) \cap \bigl(gQ - I^{2k}\bigr)}_2             \stackrel{\juststep}{=}
        \sum_{\{\alpha,\beta\} \in T\{P,Q\}}
            \babs{g(\alpha-\beta) \cap g(\beta-\alpha)}_2.
    \end{multline*}
\setcounter{mcnt}{0}    
    Let us prove the equalities.
    \begin{itemize}[leftmargin=6mm, nosep]
        \item Equality \cnt is proved as follows.
        Since $P\cap Q=1^{*k+1}$, we have $fP\cap fQ \subset iI^{2k}$.
        This, $f|_Z = ig|_Z$, and~\ref{en:gp}, imply \themcnt.
        
        \item Equality \cnt holds since $i$ is an embedding.
        
        \item
        Equality \cnt holds by the Parity Lemma (see the first usage in Remark~\ref{r:deduc-emb-ks21e}.\ref{en:r-deduc-emb-ks21e:ded-t}) and since $g|_P \sqcup g'|_Q$ is a general position map.
        On $g^{-1}(I^{2k})$ this map is in general position by~\ref{en:gp}.
        On $g^{-1}(\R^{2k} - I^{2k})$ this map is in general position since
        \begin{enumerate}
            \item[$\circ$] $g|_{g^{-1}(\R^{2k} - I^{2k})} = g'|_{g'^{-1}(\R^{2k} - I^{2k})}$,
            
            \item[$\circ$] $g$ is a general position map,
            
            \item[$\circ$] $P \cap Q = 1^{*k+1}$,

            \item[$\circ$] $iI^{2k}\supset f(1^{*k+1})$ and $iI^{2k}\cap f\bigl(\{2,3\}^{*k+1}\bigr)=\emptyset$, and

            \item[$\circ$] $f|_Z=ig|_Z$.
        \end{enumerate}
        
        \item Equality \cnt holds since $g|_{g^{-1}(\R^{2k} - I^{2k})} = g'|_{g'^{-1}(\R^{2k} - I^{2k})}$.
        
        \item Equality \cnt is proved as follows.
        For $\{\alpha,\beta\} \in T\{P,Q\}$ we have $\alpha \cap \beta \subset 1^{*k+1}$.
        This, $iI^{2k}\supset f(1^{*k+1})$, $iI^{2k}\cap f\bigl(\{2,3\}^{*k+1}\bigr)=\emptyset$ and $f|_Z=ig|_Z$ imply
        \[
            \bigl(gP-I^{2k}\bigr) \cap \bigl(gQ-I^{2k}\bigr) =
            \bigsqcup\limits_{\{\alpha,\beta\} \in T\{P,Q\}} g(\alpha-\beta) \cap g(\beta-\alpha).
        \]
    \end{itemize}
\end{proof}

\setcounter{mcnt}{0}
\renewcommand{\themcnt}{\text{(\ref{lem:non-triv}.\alph{mcnt})}}

\begin{proof}[Proof of the non-triviality in Lemma~\ref{lem:non-triv} (Embedding)]
    Take a map $g\colon[3]^{*k+1} \to \R^{2k}$ from Proposition~\ref{p:inter}.
    Then
    \begin{multline*}
        SA(f)                                               \stackrel{\juststep}{=}
        \sum_{\{P,Q\}\in G_k}
            fP \cap_M fQ                               \stackrel{\juststep}{=}
        \sum_{\{P,Q\}\in G_k}
            \sum_{\{\alpha,\beta\} \in T\{P,Q\}}
                \babs{g(\alpha-\beta) \cap g(\beta-\alpha)}_2          \stackrel{\juststep}{=} \\
        \sum_{\{\alpha, \beta\}\in H}
            \babs{g(\alpha-\beta) \cap g(\beta-\alpha)}_2              \stackrel{\juststep}{=}
        \sum_{\{\alpha, \beta\}\in H}
            \abs{g\alpha \cap g\beta}_2              \stackrel{\juststep}{=} 1.
    \end{multline*}

\setcounter{mcnt}{0}
    Here
    \begin{itemize}[leftmargin=6mm, nosep]

        
        \item equality \cnt is the definition of $SA(f)$;
        
        \item equality \cnt holds by Proposition~\ref{p:inter};

        
        \item equality \cnt holds by Lemma~\ref{l:comb} (Combinatorial);

        \item equality \cnt holds since $\alpha,\beta$ are vertex-disjoint;

        \item equality \cnt is the result of van Kampen \cite[Satz~5]{vK32} (see Lemma~\ref{lem:vanKampen}).
    \end{itemize}
\end{proof}

\begin{remark}[Relation to intrinsic linking results]\label{r:ramsey}
    (a) Non-embeddability of $\Delta_{2k+2}^k$ into $\R^{2k}$ is related to a congruence analogous to $SA(f)=1$ \cite[Proposition~16.ii]{PT19}, \cite[\S1, non-triviality]{KS21}, and to an intrinsic linking result for $(k-1)$-complexes in $\R^{2k-1}$ \cite[Theorem~1.6.odd]{KS21} (see also \cite[proof of Lemma~1.5]{KS21}).
    Non-embeddability of $K_5^k$ into $\R^{2k}$ is related to an intrinsic linking result for $(k-1)$-complexes in $\R^{2k-1}$ \cite{Sk03}. 
    Analogously, the proof of (well-known) non-embeddability of $[3]^{*k+1}$ in $\R^{2k}$ given by Lemma~\ref{lem:vanKampen} (van Kampen) is related to the congruence $SA(f)=1$, and to certain intrinsic linking result for a $(k-1)$-complex in $\R^{2k-1}$.
    For $k=2$ this is a result on a graph in $\R^3$; see (b,c,d) below.
                 
    (b) Let $G$ be the graph with the vertex set $\Z_4\times\Z_3$ and edges joining the following pairs of vertices: 
    $$(i,j)(i+1,j)\quad\text{and, \quad for }i=0,2, \quad(i+1,0)(i+1,1),\quad(i,1)(i,2),\quad (i+1,2)(i,0).$$
    Let $S_j$ be the induced subgraph on vertices $(i,j)$, $i\in\Z_4$. 
    Clearly, $S_j$ is a cycle of length~$4$.
    
    An \emph{octahedral cycle of length $6$} in $G$ is any of the following $8$ cycles for $i=0,2$ and  $\varepsilon_1,\varepsilon_2=\pm1$: 
    $$(i,0)\ (i+\varepsilon_1,0)\ (i+\varepsilon_1,1)\ (i+\varepsilon_1+\varepsilon_2,1)\ (i+\varepsilon_1+\varepsilon_2,2)\ (i+1,2).$$ 
    The two cycles for fixed $\varepsilon_1,\varepsilon_2$ and different $i$ are called {\it involutive}. 
    Clearly, involutive cycles are disjoint. 
    
    Take the (octahedral) cycles
    $$T_1 := (0,0)(1,2)(2,2)(3,2)(2,0)(3,0),$$  
    $$T_2 := (3,1)(2,1)(1,1)(1,0)(2,0)(3,0),$$ 
    $$T_3 := (2,2)(3,2)(0,2)(0,1)(3,1)(2,1).$$
    Clearly, $T_j\cap S_j=\emptyset$. 
     
    (c) \emph{Assertion.} \cite{Ni} Suppose that $f\colon G\to\R^3$ is an embedding. Then the sum of the three pairwise linking numbers of $f(S_j)$ and $f(T_j)$, $j=1,2,3$, and the four linking numbers of involutive octahedral cycles, is odd.  
    
    This follows from \cite{Sa81}. 
    Indeed, in the graph $G$, by contracting three edges $(0,0)(1,2)$, $(0,2)(0,1)$ and $(1,1)(1,0)$, we obtain a proper minor of $G$ isomorphic to the graph $G_9$ in the Petersen family \cite{Sa81}. 
    (Further, certain $\Delta Y$-move yields the Petersen graph $P_{10}$.)
    All disjoint cycle pairs of $G_9$ consist of six $(4,5)$-cycle pairs and exactly one $(3,6)$-cycle pair. 
    It is known that for every embedding  $G_9\to\R^3$ the sum of the linking numbers over all of the constituent $2$-component links is odd \cite{Sa81}.
    Three of the six $(4,5)$-cycle pairs and exactly one $(3,6)$-cycle pair of $G_9$ correspond to the four involutive octahedral cycle pairs of $G$. 
    The three remaining $(4,5)$-cycle pairs of $G_9$ correspond to pairs $S_j,T_j$ of $G$.  
    Thus the assertion follows. 
    
    (d) The following holds both by (c) (see (e)) and by our proof of non-triviality in Lemma~\ref{lem:non-triv} (see (f)). 
    Suppose that $f\colon G\to\R^3$ is an embedding such that the images $f(S_1),f(S_2),f(S_3)$ lie in pairwise disjoint $3$-balls, and the image of any edge outside $S_1\sqcup S_2\sqcup S_3$ is disjoint with one of the three balls.
    Then the sum of the four linking coefficients of involutive octahedral cycles is odd. 
    
    (e) \cite{Ni}  
    Denote by $B_j\supset f(S_j)$ the mutually disjoint 3-balls.  
    Then for $f(S_1\cup T_1)$, two edges $f((3,2)(2,2))$ and $f((0,0)(1,2))$ miss $B_1$, and the other edges of
    $f(T_1)$ also miss $B_1$ because they are contained in $B_2$ and $B_3$. 
    Thus the linking number of $f(S_1)$ and $f(T_1)$ is zero. 
    Analogously the linking number of $f(S_j)$ and $f(T_j)$ is zero for each $j=2,3$.
    This implies (d).  
    
    (f) Take the following representation of $G$. 
    Vertices of $G$ correspond to edges $a*b*c$ of $[3]^{*3}$, where among $a,b,c$ there is exactly one `3', and there is exactly one $\emptyset$. 
    We denote such a vertex by $abc$.  
    Edges of $G$ correspond to faces $3*b*c$ or $3*3*c$ of $[3]^{*3}$, and three times as many symmetric faces (i.e. faces obtained by changing the place of `3's).  
    So edges are $3b\emptyset,3\emptyset c$ ({\it short} edges),  $3\emptyset c,\emptyset3c$ ({\it long} edges), and symmetric edges.
    
    There are three cycles $S_1,S_2,S_3$ of length four obtained by changing the place of `3' from the cycle $32\emptyset,3\emptyset2,31\emptyset,3\emptyset1$ of short edges.  
    For each $c\in[2]$ there are three edges obtained by changing the place of $c$ in the long edge $3\emptyset c,\emptyset3c$.  
    
    An \emph{octahedral cycle of length $6$} is the cycle 
    $3b\emptyset,3\emptyset c,\emptyset3c,a3\emptyset,a\emptyset3,\emptyset b3,$ where $a,b,c\in[2]$.     
    This is $*(a*b*c)\cap g^{-1}(\partial [\,0,1\,]^{2k})$.  
    Take the involution on $G$ defined by interchanging $1$ and $2$. Then the eight octahedral cycles split into pairs of involutive cycles, and involutive octahedral cycles are disjoint. 
    
    Now (d) follows by transforming the left-hand side of (\ref{p:inter}.e) analogously to \cite[(2) in the proof of Lemma 1.5]{KS21}.   
\end{remark}

}


\newcommand{\aate}{\bibitem[AA38]{AA38} \emph{A. Adrian Albert}. Symmetric and alternate matrices in an arbitrary field, I. Trans. Amer. Math. Soc., (1938) 43(3):386--436.}

\newcommand{\abc}{\bibitem[ABC+]{ABC+} * \emph{M. Atiyah, A. Borel, G. J. Chaitin, D. Friedan, J. Glimm, J. J. Gray, M. W. Hirsch, S. MacLane, B. B. Mandelbrot, D. Ruelle, A. Schwarz, K. Uhlenbeck, R. Thom, E. Witten, C.  Zeeman.} Responses to ``Theoretical Mathematics: Toward a cultural synthesis of mathematics and theoretical physics'', by A. Jaffe and F. Quinn. Bull. Am. Math. Soc. 30 (1994) 178--207. arXiv:math/9404229.}

\newcommand{\abgmns}{\bibitem[ABM+]{ABM+} * \emph{E. Alkin, E. Bordacheva, A. Miroshnikov, O. Nikitenko, A. Skopenkov,} Invariants of almost embeddings of graphs in the plane: results and problems, arXiv:2408.06392.}

\newcommand{\abms}{\bibitem[ABM+]{ABM+} * \emph{Э. Алкин, Е. Бордачева, А. Мирошников, А. Скопенков,} Инварианты почти вложений графов в плоскость, arXiv:2410.09860.}

\newcommand{\adnt}{\bibitem[Ad93]{Ad93} * \emph{M. Adachi}. Embeddings and Immersions. Amer. Math.
Soc., 1993. (Transl. of Math. Monographs; V.~124).}

\newcommand{\adoe}{\bibitem[Ad18]{Ad18} {\it K. Adiprasito,} Combinatorial Lefschetz theorems beyond positivity, arXiv:1812.10454v4.}

\newcommand{\adnsv}{\bibitem[ADN+]{ADN+} * \emph{E. Alkin, S. Dzhenzher, O. Nikitenko, A. Skopenkov, A. Voropaev.} Cycles in graphs and in hypergraphs: results and problems, arXiv:2308.05175.}

\newcommand{\agles}{\bibitem[AGL]{AGL86} Mathematical Economics,  ed. by A. Ambrosetti, F. Gori, R. Lucchetti,
Lect. Notes Math. 1330, Springer, 1986.}


\newcommand{\akzz}{\bibitem[Ak00]{Ak00} * \emph{П. М. Ахметьев.} Вложения компактов, стабильные
гомотопические группы сфер и теория особенностей, Успехи Мат. Наук.  2000. 55:3. C.~3-62.}

\newcommand{\akoe}{\bibitem[AK19]{AK19} \emph{S. Avvakumov, R. Karasev.} Envy-free division using mapping degree,
Mathematika, 67:1 (2020), 36--53. arXiv:1907.11183.}

\newcommand{\akto}{\bibitem[AK21]{AK21} \emph{G. Arone and V. Krushkal.}
Embedding obstructions in $\R^d$ from the Goodwillie-Weiss calculus and Whitney disks. arXiv:2101.10995. }

\newcommand{\akm}{\bibitem[AKM]{AKM} \emph{M. Abrahamsen, L. Kleist and T. Miltzow.}
Geometric Embeddability of Complexes is $\exists\mathbb R$-complete, arXiv:2108.02585.}

\newcommand{\aksoe}{\bibitem[AKS]{AKS} \emph{S. Avvakumov, R. Karasev and A. Skopenkov.} Stronger counterexamples to the topological Tverberg conjecture, Combinatorica, 43 (2023), 717--727. arXiv:1908.08731.}


\newcommand{\akuoe}{\bibitem[AKu19]{AKu19} \emph{S. Avvakumov, S. Kudrya.}
Vanishing of all equivariant obstructions and the mapping degree.
Discr. Comp. Geom., 66:3 (2021) 1202--1216. arXiv:1910.12628.}

\newcommand{\alto}{\bibitem[Al22]{Al22} \emph{E. Alkin,}
Hardness of almost embedding simplicial complexes in $\R^d$, II. arXiv:2206.13486}

\newcommand{\amtf}{\bibitem[AM25]{AM25} \emph{E. Alkin, A. Miroshnikov,} On winding numbers of almost embeddings of $K_4$ in the plane, 	arXiv:2501.15642.}

\newcommand{\amsw}{\bibitem[AMS+]{AMSW} \emph{S. Avvakumov, I. Mabillard, A. Skopenkov and U. Wagner.}
Eliminating Higher-Multiplicity Intersections, III. Codimension 2, Israel J. Math. 245 (2021) 501--534.  arxiv:1511.03501.}


\newcommand{\anzt}{\bibitem[An03]{An03} * \emph{Д. В. Аносов.} Отображения окружности, векторные поля и их применения. М: МЦНМО, 2003.}

\newcommand{\arnf}{\bibitem[Ar95]{Ar95} * \emph{V. I. Arnold,}  Topological invariants of plane curves and caustics, University Lecture Series, Vol. 5, Amer. Math. Soc., Providence, RI, 1995.}

\newcommand{\arszo}{\bibitem[ARS01]{ARS01} \emph{P. Akhmetiev, D. Repov\v s and A. Skopenkov},
Embedding products of low-dimensional manifolds in $\R^m$, Topol. Appl. 113 (2001), 7--12.}

\newcommand{\arszt}{\bibitem[ARS02]{ARS02} \emph{P. Akhmetiev, D. Repovs and A. Skopenkov.} Obstructions to approximating maps of $n$-manifolds into $R^{2n}$ by embeddings, Topol. Appl., 123 (2002), 3--14.}

\newcommand{\asoed}{\bibitem[As]{As} \emph{A. Asanau,} \lowercase{A SIMPLE PROOF THAT CONNECTED SUM OF ORDERED
ORIENTED LINKS IS NOT WELL-DEFINED,} Math. Notes, to appear.}

\newcommand{\asoe}{\bibitem[As]{As} \emph{A. Asanau,} On the \lowercase{TRIPLE SELF-INTERSECTION NUMBER FOR GRAPHS IN THE PLANE,} unpublished, 2018.}

\newcommand{\avos}{\bibitem[Av14]{Av14} \emph{S. Avvakumov,} The classification of certain linked 3-manifolds in 6-space, Moscow Math. J., 16:1 (2016), 1--25. arXiv:1408.3918.}

\newcommand{\avose}{\bibitem[Av17]{Av17} \emph{S. Avvakumov,} The classification of linked 3-manifolds in 6-space, Algebraic \& Geometric Topology, 22:6 (2022) 2587--2630. arXiv:1704.06501.}

\newcommand{\bant}{\bibitem[Ba93]{Ba93} * \emph{T. Bartsch.} Topological methods for variational problems with
symmetries, Lecture Notes in Mathematics, 1560, Springer-Verlag, Berlin, 1993.}

\newcommand{\batt}{\bibitem[Ba23]{Ba23} * \emph{I. Barany.} Tverberg's theorem, a new proof. arXiv:2308.10105.}

\newcommand{\bbsn}{\bibitem[BB79]{BB} \emph{E.~G. Bajm{{\'o}}czy and I.~B{{\'a}}r{{\'a}}ny,}
\newblock On a common generalization of {B}orsuk's and {R}adon's theorem,
\newblock Acta Math.\ Acad.\ Sci.\ Hungar.\ 34:3 (1979), 347-350.}

\newcommand{\bbzos}{\bibitem[BBZ]{BBZ} * \emph{I.~B{{\'a}}r{{\'a}}ny, P.~V.~M. Blagojevi{{\'c}} and G.~M. Ziegler.} Tverberg's Theorem at 50: Extensions and Counterexamples, Notices of the Amer. Math. Soc., 63:7 (2016), 732--739.}


\newcommand{\bcm}{\bibitem[BCM]{BCM} * 13th Hilbert Problem on superpositions of functions, presented by A. Belov, A. Chilikov, I. Mitrofanov, S. Shaposhnikov and A. Skopenkov,
\url{http://www.turgor.ru/lktg/2016/5/index.htm}.}

\newcommand{\beet}{\bibitem[BE82]{BE82} * \emph{V.G. Boltyansky and V.A. Efremovich.} Intuitive Combinatorial Topology. Springer.}

\newcommand{\beetr}{\bibitem[BE82]{BE82} * \emph{В. Г. Болтянский и В. А. Ефремович.} Наглядная топология. М.:  Наука, 1982.}


\newcommand{\bfzn}{\bibitem[BF09]{BF09} \emph{K. Barnett, M. Farber}. Topology of Configuration Space of Two Particles on a Graph, I.  Algebr. Geom. Topol. 9 (2009) 593--624.	arXiv:0903.2180.}

\newcommand{\bfzof}{\bibitem[BFZ14]{BFZ14} \emph{P. V. M. Blagojevi{\'c}, F. Frick, and G. M. Ziegler,}
Tverberg plus constraints, Bull. Lond. Math. Soc. 46:5 (2014), 953-967, arXiv:1401.0690.}


\newcommand{\bfzos}{\bibitem[BFZ]{BFZ} \emph{P. V. M. Blagojevi{\'c}, F. Frick and G. M. Ziegler,}
Barycenters of Polytope Skeleta and Counterexamples to the Topological Tverberg Conjecture, via Constraints,
J. Eur. Math. Soc., 21:7 (2019) 2107-2116. arXiv:1510.07984.}


\newcommand{\bgso}{\bibitem[BG71]{BG71} J.C. Becker and H. H. Glover, {\it Note on the Embedding of Manifolds in Euclidean Space,} Proc. of the Amer. Math. Soc., 27:2 (1971) 405-410.}


\newcommand{\bgos}{\bibitem[BG16]{BG16} \emph{A. Bj\"orner and A. Goodarzi}, On Codimension one Embedding of Simplicial Complexes, in book: A Journey Through Discrete Mathematics, arXiv:1605.01240.}

\newcommand{\biet}{\bibitem[Bi83]{Bi83} * \emph{R. H. Bing.} The Geometric Topology of 3-Manifolds. Providence, R.~I. 1983. (Amer. Math. Soc. Colloq. Publ., 40).}

\newcommand{\bitz}{\bibitem[Bi20]{Bi20} * \emph{A. Bikeev.} Realizability of discs with ribbons on the M\"obius strip. Mat. Prosveschenie, 28 (2021), 150-158;
erratum to appear. arXiv:2010.15833.}

\newcommand{\bitzr}{\bibitem[Bi20]{Bi20} * \emph{А. Бикеев.} Реализуемость дисков с ленточками на ленте Мебиуса.
Мат. просвещение. Сер. 3. 28 (2021), 150--158.}

\newcommand{\bito}{\bibitem[Bi21]{Bi21} {\it A. I. Bikeev,}
Criteria for integer and modulo 2 embeddability of graphs to surfaces, arXiv:2012.12070v2.}


\newcommand{\bagos}{\bibitem[BG17]{BG17} \emph{S. Basu and S. Ghosh.} Equivariant maps related to the topological Tverberg conjecture, Homology, Homotopy and Applications 19:1 (2017) 155--170.}

\newcommand{\bkkmzof}{\bibitem[BKK]{BKK} \emph{M. Bestvina, M. Kapovich and B. Kleiner,}
Van Kampen's embedding obstruction for discrete groups, Invent. Math. 150 (2002) 219--235. arXiv:math/0010141.}

\newcommand{\bl}{\bibitem[BL]{BL} \url{https://en.wikipedia.org/wiki/Brunnian_link}}

\newcommand{\blf}{\bibitem[BL4]{BL4} Students form a 4-component Brunnian link,  \url{http://www.mccme.ru/circles/oim/foto2014/brunn4.png} (5Mb)}

\newcommand{\bmzf}{\bibitem[BM04]{BM04} \emph{Boyer, J. M. and Myrvold, W. J.} On the cutting edge: simplified $O(n)$ planarity by edge addition,  Journal of Graph Algorithms and Applications, 8:3 (2004) 241--273.}

\newcommand{\bm}{\bibitem[BM15]{BM15} \emph{I. Bogdanov and A. Matushkin.} Algebraic proofs of linear versions of the Conway--Gordon--Sachs theorem and the van Kampen--Flores theorem, arXiv:1508.03185.}


\newcommand{\bmzzn}{\bibitem[BMZ09]{BMZ09} \emph{P. V. M. Blagojevi{\'c}, B. Matschke, G. M. Ziegler,}
Optimal bounds for a colorful Tverberg-Vre\'cica type problem, Advances in Math., 226 (2011), 5198-5215, arXiv:0911.2692.}

\newcommand{\bmzof}{\bibitem[BMZ15]{BMZ15} \emph{P. V. M. Blagojevi{\'c}, B. Matschke, G. M. Ziegler,}
Optimal bounds for the colored Tverberg problem, J. Eur. Math. Soc.,  17:4 (2015) 739--754,
arXiv:0910.4987.}

\newcommand{\bpns}{\bibitem[BP97]{BP97} * \emph{R. Benedetti and C. Petronio.} Branched standard spines of 3-manifolds, Lecture Notes in Math. 1653, Springer-Verlag, Berlin-Heidelberg-New York, 1997.}

\newcommand{\brst}{\bibitem[Br72]{Br72} \emph{J. L. Bryant.} Approximating embeddings of polyhedra in codimension 3, Trans. Amer. Math. Soc., 170 (1972) 85--95.}

\newcommand{\brts}{\bibitem[Br68]{Br68} \emph{P. Bruegel,} 1568,
\url{https://en.wikipedia.org/wiki/The_Magpie_on_the_Gallows}.}


\newcommand{\bren}{\bibitem[Br82]{brown1982} * \emph{K.~S. Brown.} \newblock Cohomology of Groups. \newblock Springer-Verlag New York, 1982.}


\newcommand{\bssos}{\bibitem[BS17]{BS17} * \emph{I.~B\'{a}r\'{a}ny and P. Sober\'{o}n,} Tverberg's theorem is 50 years old: a survey, Bull. Amer. Math. Soc. (N.S.) 55:4 (2018), 459--492. arXiv:1712.06119.}

\newcommand{\bsto}{\bibitem[BS21]{BS21} * \emph{A. Buchaev and A. Skopenkov,} Simple proofs of estimations of Ramsey numbers and of discrepancy, Mat. Prosveschenie, to appear, arXiv:2107.13831.}

\newcommand{\brsnn}{\bibitem[BRS99]{BRS99} \emph{D. Repov\v s, N. Brodsky and A. B. Skopenkov.}
A classification of 3-thickenings of 2-polyhedra, Topol. Appl. 1999. 94. P.~307-314.}

\newcommand{\bsseo}{\bibitem[BSS]{BSS} \emph{I.~B\'{a}r\'{a}ny, S.~B. Shlosman, and A.~Sz{\H{u}}cs,}
\newblock On a topological generalization of a theorem of {T}verberg,
\newblock J.\ London Math.\ Soc.\ (II. Ser.) 23 (1981), 158--164.}

\newcommand{\btzs}{\bibitem[BT07]{BT07} \emph{A. Bj\"orner, M. Tancer}, Combinatorial Alexander Duality --- a Short and Elementary Proof, Discr. and Comp. Geom., 42 (2009) 586. arXiv:0710.1172.}

\newcommand{\buse}{\bibitem[Bu68]{Bu68} \emph{A. R. Butz,} Space filling curves and mathematical programming, Information and Control, 12:4 (1968) 314--330.}


\newcommand{\bz}{\bibitem[BZ16]{BZ16} * \emph{P. V. M. Blagojevi\'c and G. M. Ziegler,} Beyond the Borsuk-Ulam theorem: The topological Tverberg story, in: A Journey Through Discrete Mathematics, Eds. M. Loebl,
J. Ne\v set\v ril, R. Thomas, Springer, 2017, 273--341. arXiv:1605.07321v3.}


\newcommand{\cano}{\bibitem[Ca91]{Ca91} * \emph{D. de Caen}, The ranks of tournament matrices, Amer. Math. Monthly, 98:9 (1991) 829--831.}

\newcommand{\ca}{\bibitem[Ca]{Ca} \emph{J. Carmesin.} Embedding simply connected 2-complexes in 3-space, I-V, arXiv:1709.04642, arXiv:1709.04643, arXiv:1709.04645, arXiv:1709.04652, arXiv:1709.04659.}

\newcommand{\cfsz}{\bibitem[CF60]{CF60} \emph{P. E. Conner and E. E. Floyd}, Fixed points free involutions and equivariant maps, Bull. Amer. Math. Soc., 66 (1960) 416--441.}

\newcommand{\cfs}{\bibitem[CFS]{CFS} \emph{D. Crowley, S.C. Ferry, M. Skopenkov,} The rational classification of links of codimension $>2$, Forum Math. 26 (2014), 239--269. arXiv:1106.1455.}

\newcommand{\cget}{\bibitem[CG83]{CG83} \emph{J. H. Conway and C. M. A. Gordon},
Knots and links in spatial graphs, J. Graph Theory  7 (1983), 445--453.}

\newcommand{\cten}{\bibitem[Ch]{Ch} \emph{Chuang Tzu,} translated by H. A. Giles, Bernard Quaritch, London, 1889.}

\newcommand{\ctruku}{\bibitem[Ch]{Ch} \emph{Chuang Tzu,} translated to Russian by S. Kuchera, in: Ancient Chinese Philosophy, v. I, Mysl, Moscow, 1972.}


\newcommand{\chnn}{\bibitem[Ch99]{Ch99} * \emph{А. В. Чернавский,} Теорема Жордана.  Мат. Просвещение, 3 (1999), 142--157.}

\newcommand{\hcon}{\bibitem[HC19]{HC19} * \emph{C. Herbert Clemens.} Two-Dimensional Geometries. A Problem-Solving Approach, Amer. Math. Soc., 2019.}

\newcommand{\ckmoo}{\bibitem[CKMS]{CKMS} \emph{M. \v Cadek, M. Kr\v c\'al. J. Matou\v sek, F. Sergeraert,
L. Vok\v r\'inek, U. Wagner.} Computing all maps into a sphere, J. of the ACM, 61:3 (2014). arXiv:1105.6257.}


\newcommand{\ckmvwot}{\bibitem[CKM12+]{CKM12+} \emph{M. \v Cadek, M. Kr\v c\'al. J. Matou\v sek, L. Vok\v r\'inek, U. Wagner.} Polynomial-time computation of homotopy groups and Postnikov systems in fixed dimension, SIAM J. Comput., 43:5 (2014), 1728--1780. arXiv:1211.3093.}

\newcommand{\ckmvw}{\bibitem[CKM+]{CKM+} \emph{M. \v Cadek, M. Kr\v c\'al. J. Matou\v sek, L. Vok\v r\'inek, U. Wagner.} Extendability of continuous maps is undecidable, Discr. and Comp. Geom. 51 (2014) 24--66.
arXiv:1302.2370.}

\newcommand{\ckppt}{\bibitem[CKP+]{CKP+} \emph{E. Colin de Verdi\'ere, V. Kalu\v za, P. Pat\'ak, Z. Pat\'akov\'a and M. Tancer.} A direct proof of the strong Hanani-Tutte theorem on the projective plane. Journal of Graph Algorithms and Applications, 21:5 (2017) 939--981.}

\newcommand{\cksof}{\bibitem[CKS+]{CKS+} * New ways of weaving baskets, presented by G. Chelnokov, Yu. Kudryashov, A.Skopenkov and A. Sossinsky, \url{http://www.turgor.ru/lktg/2004/lines.en/index.htm}.}

\newcommand{\ckv}{\bibitem[CKV]{CKV} \emph{M.~{\v{C}}adek, M.~Kr\v{c}\'{a}l, and L.~Vok\v{r}\'{\i}nek.}
Algorithmic solvability of the lifting-extension problem, Discr. Comp. Geom. 57 (2017), 915--965. arXiv:1307.6444.}


\newcommand{\clr}{\bibitem[CLR]{CLR} * \emph{Т. Кормен, Ч. Лейзерсон, Р. Ривест.} Алгоритмы:
построение и анализ, МЦНМО, Москва, 1999.}

\newcommand{\clreng}{\bibitem[CLR]{CLR} * \emph{T. H. Cormen, C. E.Leiserson, R. L.Rivest, C. Stein.} Introduction to Algorithms, MIT Press, 2009.}

\newcommand{\crzfru}{\bibitem[CR]{CR} * \emph{Р. Курант, Дж. Роббинс,} Что такое математика. М.: МЦНМО, 2004.}

\newcommand{\crzfen}{\bibitem[CR]{CR} * \emph{R. Courant and H. Robbins,} What is Mathematics, Oxford Univ. Press.}

\newcommand{\crsne}{\bibitem[CRS98]{CRS98} * \emph{A. Cavicchioli, D. Repov\v s and A. B. Skopenkov.}
Open problems on graphs, arising from geometric topology, Topol. Appl. 1998. 84. P.~207-226.}

\newcommand{\crsot}{\bibitem[CRS]{CRS} \emph{M. Cencelj, D. Repov\v s and M. Skopenkov,}
Classification of knotted tori in the 2-metastable dimension, Mat. Sbornik, 203:11 (2012), 1654--1681.
arxiv:math/0811.2745.}

\newcommand{\csoo}{\bibitem[CS08]{CS08} \emph{D. Crowley and A. Skopenkov.} A classification of smooth embeddings of 4-manifolds in 7-space, II, Intern. J. Math., 22:6 (2011) 731-757, arxiv:math/0808.1795.}

\newcommand{\csos}{\bibitem[CS16]{CS16} \emph{D. Crowley and A. Skopenkov,} Embeddings of non-simply-connected 4-manifolds in 7-space. I. Classification modulo knots, Moscow Math. J., 21 (2021), 43--98. arXiv:1611.04738.}


\newcommand{\csoso}{\bibitem[CS16o]{CS16o} \emph{D. Crowley and A. Skopenkov,} Embeddings of non-simply-connected 4-manifolds in 7-space. II. On the smooth classification, Proc. A of the Royal Soc. of Edinburgh 152:1 (2022), 163--181. arXiv:1612.04776.}


\newcommand{\crsk}{\bibitem[CS]{CS} \emph{D. Crowley and A. Skopenkov,} Embeddings of non-simply-connected 4-manifolds in 7-space. III. Piecewise-linear classification. draft.}

\newcommand{\cutz}{\bibitem[Cu20]{Cu20} \emph{C. Culter,} Cantor sets are not tangent homogeneous,
Topol. Appl. 271 (2020) 1--9.}

\newcommand{\dies}{\bibitem[Di87]{Di} * \emph{T. tom Dieck,} Transformation groups, Studies in Mathematics, vol. 8, Walter de Gruyter, Berlin, 1987.}

\newcommand{\dize}{\bibitem[Di08]{Di08} * \emph{T. tom Dieck,} Algebraic topology, EMS Textbooks in Mathematics, 
EMS, Z\"urich, 2008.}

\newcommand{\dent}{\bibitem[De93]{De93}  \emph{T.K. Dey.} On counting triangulations in $d$-dimensions. Comput. Geom.  3:6 (1993) 315--325.}

\newcommand{\denf}{\bibitem[DE94]{DE94}  \emph{T.K. Dey and H. Edelsbrunner.} Counting triangle crossings and halving planes, Discrete Comput. Geom. 12 (1994), 281--289.}

\newcommand{\dgn}{\bibitem[DGN+]{DGN+} * S. Dzhenzher, T. Garaev, O. Nikitenko, A. Petukhov, A. Skopenkov, A. Voropaev, Low rank matrix completion and realization of graphs: results and problems, arXiv:2501.13935.}

\newcommand{\dgnr}{\bibitem[DGN+]{DGN+} * Минимизация ранга восполнением матриц, представляли А. Воропаев, Т. Гараев, С. Дженжер, О. Никитенко, А. Петухов и А. Скопенков, \url{https://www.mccme.ru/circles/oim/netflix_rus.pdf}.}
 
\newcommand{\dstt}{\bibitem[DS22]{DS22}  \emph{S. Dzhenzher and A. Skopenkov,} A quadratic estimation for the K\"uhnel conjecture on embeddings, arXiv:2208.04188.}

\newcommand{\botf}{\bibitem[Dz25]{Dz25} \emph{E. Dzhenzher,} Symmetric 1-cycles in the deleted product of a graph, Topol. Appl. (2025) 109277.}

\newcommand{\embo}{\bibitem[Eb]{Eb} * \url{http://www.map.mpim-bonn.mpg.de/Embeddings_of_manifolds_with_boundary:_classification}}

\newcommand{\embe}{\bibitem[Em]{Em} * \url{http://www.map.mpim-bonn.mpg.de/Embedding_(simple_definition)}}

\newcommand{\ers}{\bibitem[ERS]{ERS} * Invariants of graph drawings in the plane, presented by A. Enne, A. Ryabichev, A. Skopenkov and T. Zaitsev, \url{http://www.turgor.ru/lktg/2017/6/index.htm}}

\newcommand{\feto}{\bibitem[Fe21]{Fe21} \emph{M. Fedorov.} A description of values of Seifert form for punctured $n$-manifolds in $(2n-1)$-space, arXiv:2107.02541.}

\newcommand{\ffen}{\bibitem[FF89]{FF89} * \emph{А. Т. Фоменко и Д. Б. Фукс.} Курс гомотопической топологии. М.: Наука, 1989.}

\newcommand{\ffene}{\bibitem[FF89]{FF89} * \emph{A.T. Fomenko and D.B. Fuchs.} Homotopical Topology, Springer, 2016.}


\newcommand{\fhzo}{\bibitem[FH10]{FH10}  \emph{M. Farber, E. Hanbury}. Topology of Configuration Space of Two Particles on a Graph, II. Algebr. Geom. Topol. 10 (2010) 2203--2227. arXiv:1005.2300.}


\newcommand{\fkosc}{\bibitem[FK17]{FK17} \emph{R. Fulek, J. Kyn{\v{c}}l,} Counterexample to an Extension of the Hanani-Tutte Theorem on the Surface of Genus 4, Combinatorica, 39 (2019) 1267--1279, arXiv:1709.00508.}

\newcommand{\fkos}{\bibitem[FK17]{FK17} \emph{R. Fulek, J. Kyn{\v{c}}l,} Hanani-Tutte for approximating maps of graphs, arXiv:1705.05243.}

\newcommand{\fkon}{\bibitem[FK19]{FK19} \emph{R. Fulek, J. Kyn{\v{c}}l,}
$\Z_2$-genus of graphs and minimum rank of partial symmetric matrices,
35th Intern. Symp. on Comp. Geom. (SoCG 2019), Article No. 39; pp. 39:1--39:16, \linebreak
\url{https://drops.dagstuhl.de/opus/volltexte/2019/10443/pdf/LIPIcs-SoCG-2019-39.pdf}.
We refer to numbering in arXiv version: arXiv:1903.08637.}

\newcommand{\fktnf}{\bibitem[FKT]{FKT} \emph{M. H. Freedman, V. S. Krushkal and P. Teichner.} Van Kampen's
embedding obstruction is incomplete for 2-complexes in~$\R^4$, Math. Res. Letters. 1994. 1. P.~167-176.}

\newcommand{\fltf}{\bibitem[Fl34]{Fl34} \emph{A. Flores}, \"Uber $n$-dimensionale Komplexe die im $E^{2n+1}$ absolut selbstverschlungen sind, Ergeb. Math. Koll. 6 (1934) 4--7.}

\newcommand{\fo}{\bibitem[Fo]{Fo} * \emph{L. Fortnow.} Time for Computer Science to Grow Up,  \url{https://people.cs.uchicago.edu/~fortnow/papers/growup.pdf}.}

\newcommand{\fozf}{\bibitem[Fo04]{Fo04} * \emph{R. Fokkink.} A forgotten mathematician, Eur. Math. Soc. Newsletter 52 (2004) 9--14.}


\newcommand{\fpstz}{\bibitem[FPS]{FPS} \emph{R. Fulek, M.J. Pelsmajer and M. Schaefer.}
Strong Hanani-Tutte for the Torus, arXiv:2009.01683.}

\newcommand{\frse}{\bibitem[Fr78]{Fr78} \emph{M. Freedman,} Quadruple points of 3-manifolds in $S^4$, Comment. Math. Helv. 53 (1978), 385-394.}

\newcommand{\fres}{\bibitem[FR86]{FR86} \emph{R. Fenn, D. Rolfsen.}
Spheres may link homotopically in 4-space, J. London Math. Soc. 34 (1986) 177-184.}

\newcommand{\frofea}{\bibitem[Fr15']{Fr15'} \emph{F. Frick}, Counterexamples to the topological Tverberg conjecture, arXiv:1502.00947v1.}


\newcommand{\frof}{\bibitem[Fr15]{Fr15} \emph{F. Frick}, Counterexamples to the topological Tverberg conjecture,
Oberwolfach reports, 12:1 (2015), 318--321. arXiv:1502.00947.}

\newcommand{\fros}{\bibitem[Fr17]{Fr17} \emph{F. Frick}, O\lowercase{N AFFINE TVERBERG-TYPE RESULTS WITHOUT CONTINUOUS GENERALIZATION}, arXiv:1702.05466}


\newcommand{\fstz}{\bibitem[FS20]{FS20} \emph{F. Frick and P. Sober\'on}, The topological Tverberg problem beyond prime powers, arXiv:2005.05251.}

\newcommand{\ftss}{\bibitem[FT77]{FT77} \emph{R. Fenn, P. Taylor,} Introducing doodles, pp. 37-43
in: Topology of Low-Dimensional Manifolds, Proceedings of the Second Sussex Conference, 1977,
Ed. R. Fenn, V. 722 of Lecture Notes in Math.}

\newcommand{\fvto}{\bibitem[FV21]{FV21} \emph{M. Filakovsk\'y, L. Vok\v r\'inek.} Computing homotopy classes for diagrams, Discr. Comp. Geom. 70 (2023), 866--920. arXiv:2104.10152.}

\newcommand{\fwz}{\bibitem[FWZ]{FWZ} \emph{M. Filakovsk\'y, U. Wagner, S. Zhechev.} Embeddability of simplicial complexes is undecidable. Oberwolfach reports, to appear.}

\newcommand{\fwztz}{\bibitem[FWZ]{FWZ} \emph{M. Filakovsk\'y, U. Wagner, S. Zhechev.} Embeddability of simplicial complexes is undecidable. Proceedings of the 2020 ACM-SIAM Symposium on Discrete Algorithms.}


\newcommand{\ga}{\bibitem[GA]{GA} * \url{https://en.wikipedia.org/wiki/Galactic_algorithm}}

\newcommand{\gatt}{\bibitem[Ga23]{Ga23} \emph{T. Garaev}, On drawing $K_5$ minus an edge in the plane, arXiv:2303.14503.}

\newcommand{\gdikrse}{\bibitem[GDI]{GDI} * {\it A. Chernov, A. Daynyak, A. Glibichuk, M. Ilyinskiy, A. Kupavskiy, A. Raigorodskiy and A. Skopenkov,} Elements of Discrete Mathematics As a Sequence of Problems (in Russian),
MCCME, Moscow, 2016. Update of a part: \url{http://www.mccme.ru/circles/oim/discrbook.pdf}}

\newcommand{\gdikrs}{\bibitem[GDI]{GDI} * {\it А.А. Глибичук, А.Б. Дайняк, Д.Г. Ильинский, А.Б. Купавский, А.М. Райгородский, А.Б. Скопенков, А.А. Чернов,} Элементы дискретной математики в задачах, М, МЦНМО, 2016. Обновляемая версия части книги: 
\url{http://www.mccme.ru/circles/oim/discrbook.pdf}}

\newcommand{\giso}{\bibitem[Gi71]{Gi71} * {\it S. Gitler,} Immersion and Embedding of Manifolds, Proc. Symp. Pure Math. 22, 87-96 (1971).}

\newcommand{\gkp}{\bibitem[GKP]{GKP} * {\it R. Graham, D. Knuth, and O. Patashnik,} Concrete Mathematics: A Foundation for Computer Science, Addison–Wesley, first published in 1989, \url{https://www.csie.ntu.edu.tw/~r97002/temp/Concrete\%20Mathematics\%202e.pdf}.}

\newcommand{\gmpptw}{\bibitem[GMP+]{GMP+} \emph{X. Goaoc, I. Mabillard, P. Pat\'ak, Z. Pat\'akov\'a, M. Tancer, U. Wagner}, On Generalized Heawood Inequalities for Manifolds: a van Kampen--Flores-type Nonembeddability Result,
Israel J. Math., 222(2) (2017) 841-866. arXiv:1610.09063.}


\newcommand{\gppot}{\bibitem[GPP+]{GPP+} \emph{X. Goaoc, P. Pat\'ak, Z. Pat\'akov\'a, M. Tancer, and U. Wagner.} Bounding Helly numbers via Betti numbers. In 31st International Symposium on Computational Geometry, volume 34
of LIPIcs. Leibniz Int. Proc. Inform., pp. 507-521. Schloss Dagstuhl. Leibniz-Zent. Inform., Wadern, 2015. Full version: arXiv:1310.4613.}

\newcommand{\group}{\bibitem[Gr]{Gr} * \url{https://en.wikipedia.org/wiki/Groupthink}}

\newcommand{\grsz}{\bibitem[Gr69]{Gr69} \emph{B. Gr\"unbaum.} Imbeddings of simplicial complexes. Comment. Math. Helv., 44:1, 502--513, 1969.}


\newcommand{\gres}{\bibitem[Gr86]{Gr86} * \emph{M. Gromov}, Partial Differential Relations,
Ergebnisse der Mathematik und ihrer Grenzgebiete (3), Springer Verlag, Berlin-New York, 1986.}

\newcommand{\groz}{\bibitem[Gr10]{Gr10} \emph{M. Gromov,}
\newblock Singularities, expanders and topology of maps. Part 2: From combinatorics to topology via algebraic isoperimetry, \newblock Geometric and Functional Analysis 20 (2010), no.~2, 416--526.}

\newcommand{\grsn}{\bibitem[GR79]{GR79} \emph{J. L. Gross	and R. H. Rosen}, A linear time planarity algorithm for 2-complexes, Journal of the ACM, 26:4 (1979), 611--617.}

\newcommand{\gs}{\bibitem[GS]{GS} \emph{М. Гортинский и О. Скрябин.} Критерий вложимости графов в плоскость вдоль прямой, препринт.}

\newcommand{\gssn}{\bibitem[GS79]{GS} \emph{P.~M. Gruber and R.~Schneider,} Problems in geometric convexity. In {\em Contributions to geometry (Proc. Geom. Sympos., Siegen, 1978)}, 255--278. Birkh{\"a}user, Basel-Boston, Mass., 1979.}

\newcommand{\gsnn}{\bibitem[GS99]{GS99} \emph{R. Gompf and A. Stipsicz,}
4-manifolds and Kirby calculus, GSM20, AMS, Providence, RI, 1999.}


\newcommand{\gszs}{\bibitem[GS06]{GS06} \emph{D. Goncalves and A. Skopenkov,} Embeddings of homology equivalent manifolds with boundary, Topol. Appl., 153:12 (2006) 2026-2034. arxiv:1207.1326.}

\newcommand{\gssoe}{\bibitem[GSS+]{GSS+} * Projections of skew lines, presented by A. Gaifullin, A. Shapovalov, A. Skopenkov and M. Skopenkov, \url{http://www.turgor.ru/lktg/2001/index.php}.}

\newcommand{\gtes}{\bibitem[GT87]{GT87} * \emph{J. L. Gross and T. W. Tucker.}
Topological graph theory. New York: Wiley-Interscience, 1987.}

\newcommand{\guzn}{\bibitem[Gu09]{Gu09} \emph{A. Gundert.} On the complexity of embeddable simplicial complexes. Diplomarbeit, Freie Universit\"at Berlin, 2009. 	arXiv:1812.08447.}

\newcommand{\ha}{\bibitem[Ha]{Ha} * \emph{F. Harary.} Graph theory.
Рус. пер.: Ф. Харари. Теория графов. М., Мир, 1973.}

\newcommand{\hats}{\bibitem[Ha37]{Ha37} \emph{W. Hantzsche,} Einlagerung von Mannigfaltigkeiten in euklidische R\" aume, Math. Zeitschrift, 43:1 (1937) 38--58.}

\newcommand{\hastk}{\bibitem[Ha62k]{Ha62k} {\em A.~Haefliger,}  Knotted $(4k-1)$-spheres in $6k$-space, Ann. of Math. 75 (1962) 452--466.}

\newcommand{\hastl}{\bibitem[Ha62l]{Ha62l} \emph{A. Haefliger,} Differentiable links, Topology, 1 (1962) 241--244.}

\newcommand{\hast}{\bibitem[Ha63]{Ha63} \emph{A.~Haefliger,} Plongements differentiables dans le domain stable, Comment. Math. Helv. 36 (1962-63) 155--176.}

\newcommand{\hassa}{\bibitem[Ha66A]{Ha66A} \textit{A. Haefliger}. Differential embeddings of~$S^n$ in $S^{n+q}$ for $q>2$. Ann. Math. (2), 83 (1966), 402--~436.}

\newcommand{\hass}{\bibitem[Ha66C]{Ha66C} \emph{A.~Haefliger,}  Enlacements de spheres en codimension superiure \`a 2, Comment. Math. Helv. 41 (1966-67) 51--72.}

\newcommand{\hase}{\bibitem[Ha68]{Ha68} \emph{A. Haefliger,} Knotted Spheres and Related Geometric Topic,
in Proc. Int. Congr. Math., Moscow, 1966 (Mir, Moscow, 1968), 437--445.}

\newcommand{\hasn}{\bibitem[Ha69]{Ha69} \emph{L.~S.~Harris,} Intersections and embeddings of polyhedra, Topology 8 (1969) 1--26.}

\newcommand{\hasf}{\bibitem[Ha74]{Ha74} * \emph{P. Halmos,} How to talk mathematics. Notices of the Amer. Math. Soc., 21 (1974) 155--158.}

\newcommand{\haef}{\bibitem[Ha84]{Ha84} \emph{N. Habegger,} Obstruction to embedding disks II: a proof of a conjecture by Hudson, Topol. Appl. 17 (1984).}

\newcommand{\haes}{\bibitem[Ha86]{Ha86} \emph{N. Habegger,} Knots and links in codimension greater than 2, Topology, 25:3 (1986) 253--260.}

\newcommand{\hogr}{\bibitem[HG]{HG} * \url{http://www.map.mpim-bonn.mpg.de/Homology_groups_(simplicial;_simple_definition)}}

\newcommand{\hifn}{\bibitem[Hi59]{Hi59} \emph{M. W. Hirsch.} Immersions of manifolds, Trans. Amer. Math. Soc. 93 (1959) 242--276.}

\newcommand{\hjsf}{\bibitem[HJ64]{HJ64} \emph{R. Halin and H. A. Jung.}
Karakterisierung der Komplexe der Ebene und der 2-Sph\"are, Arch. Math. 1964. 15. P.~466-469.}

\newcommand{\hkne}{\bibitem[HK98]{HK98} \emph{N. Habegger and U. Kaiser,} Link homotopy in 2--metastable range, Topology 37:1 (1998) 75--94.}

\newcommand{\hmsnt}{\bibitem[HMS]{HMS93} * \emph{C. Hog-Angeloni, W. Metzler and A. J. Sieradski.}
Two-dimensional homotopy and combinatorial group theory. Cambridge: Cambridge Univ. Press, 1993. (London Math. Soc. Lecture Notes, 197).}

\newcommand{\ho}{\bibitem[Ho]{Ho} * The Hopf fibration, \url{https://www.youtube.com/watch?v=AKotMPGFJYk}}

\newcommand{\hozs}{\bibitem[Ho06]{Ho06} \emph{H. van der Holst,} Graphs and obstructions in four dimensions, J. Combin. Theory Ser. B 96:3 (2006), 388--404.}


\newcommand{\hpzn}{\bibitem[HP09]{HP09} \emph{H. van der Holst and R. Pendavingh,} On a graph property generalizing planarity and flatness, Combinatorica, 29 (2009) 337--361.}

\newcommand{\hssf}{\bibitem[HS64]{HS64} \emph{A. Haefliger and B. Steer,} Symmetry of linking coefficients, Comment. Math. Helv. 39 (1964) 259-270.}

\newcommand{\htsf}{\bibitem[HT74]{HT74} \emph{J. Hopcroft and R. E. Tarjan,} Efficient planarity testing, J. of the Association for Computing Machinery, 21:4 (1974) 549--568.}

\newcommand{\hufn}{\bibitem[Hu59]{hu59} * \emph{S. T. Hu,} Homotopy Theory, Academic Press, New York, 1959.}

\newcommand{\husn}{\bibitem[Hu69]{Hu69} * \emph{J. F. P. Hudson.} Piecewise linear topology, W. A. Benjamin, Inc., New York-Amsterdam, 1969.}

\newcommand{\io}{\bibitem[Io]{Io} * \url{https://en.wikipedia.org/wiki/Category:Impossible_objects}}

\newcommand{\info}{\bibitem[IF]{IF} * \url{http://www.map.mpim-bonn.mpg.de/Intersection_form}}

\newcommand{\irsf}{\bibitem[Ir65]{Ir65} \emph{M.~C.~Irwin,} Embeddings of polyhedral manifolds, Ann. of Math. (2)
82 (1965) 1--14.}

\newcommand{\isot}{\bibitem[Is]{Is} * \url{http://www.map.mpim-bonn.mpg.de/Isotopy}}

\newcommand{\jqnt}{\bibitem[JQ93]{JQ93} * \emph{A. Jaffe, F. Quinn,} ``Theoretical mathematics'': Toward a cultural synthesis of mathematics and theoretical physics. Bull.Am.Math.Soc. 29 (1993) 1-13. arXiv:math/9307227.}

\newcommand{\jozt}{\bibitem[Jo02]{Jo02} \emph{C. M. Johnson.} An obstruction to embedding a simplicial $n$-complex into a $2n$-manifold, Topology Appl. 122:3 (2002) 581--591.}

\newcommand{\jvz}{\bibitem[JVZ]{JVZ} D. Joji\'c, S. T. Vre\'cica, R. T. \v Zivaljevi\' c,
Topology and combinatorics of 'unavoidable complexes', arXiv:1603.08472v1.}


\newcommand{\kalai}{\bibitem[Ka]{Ka} G. Kalai, From Oberwolfach: The Topological Tverberg Conjecture is False, `Combinatorics and more' blog post, February 6, 2015, \url{gilkalai.wordpress.com}}

\newcommand{\kh}{\bibitem[Kh]{Kh} \emph{А.И. Храбров.} Руководство по чтению лекций
\url{http://vm.tstu.tver.ru/topics/pdf_tests/lection.pdf}}

\newcommand{\kho}{\bibitem[Kho]{Kho} \emph{N. Khoroshavkina.} A simple characterization of graphs of cutwidth 2, arXiv:1811.06716.}

\newcommand{\kkrot}{\bibitem[KKR]{KKR} \emph{K. Kawarabayashi, Y. Kobayashi and B. Reed.} The disjoint paths problem in quadratic time, J. of Comb. Theory, Ser. B, 102:2 (2012), 424--435.}

\newcommand{\kmsth}{\bibitem[KM63]{KM63} \emph{M. A. Kervaire and J. W. Milnor,} Groups of homotopy spheres. I,  Ann. of Math. (2) 77 (1963), 504-537.}

\newcommand{\kozeru}{\bibitem[Ko18]{Ko18} * \emph{Е. Колпаков.}
Доказательство теоремы Радона при помощи понижения размерности, Мат. Просвещение, 23 (2018), arXiv:1903.11055.}

\newcommand{\koze}{\bibitem[Ko18]{Ko18} * \emph{E. Kolpakov.}
A proof of Radon Theorem via lowering of dimension, Mat. Prosveschenie, 23 (2018), arXiv:1903.11055.}

\newcommand{\ko}{\bibitem[Ko]{Ko} \emph{E. Kolpakov.} A `converse' to the Constraint Lemma, arXiv:1903.08910.}

\newcommand{\koon}{\bibitem[Ko19]{Ko19} \emph{E. Kogan.} Linking of three triangles in 3-space, arXiv:1908.03865.}

\newcommand{\koto}{\bibitem[Ko21]{Ko21} \emph{E. Kogan.} On the rank of $\Z_2$-matrices with free entries on the diagonal, arXiv:2104.10668.}

\newcommand{\koee}{\bibitem[Ko88]{Ko88} \emph{U. Koschorke.} Link maps and the geometry of their invariants,
Manuscripta Math. 61:4 (1988) 383--415.}

\newcommand{\kono}{\bibitem[Ko91]{Ko91} \emph{U. Koschorke.} Link homotopy with many components,
Topology 30:2 (1991) 267--281.}

\newcommand{\kons}{\bibitem[Ko97]{Ko97} \emph{U. Koschorke.} A generalization of Milnor's $\mu$-invariants to higher-dimensional link maps, Topology 36:2 (1997) 301--324.}

\newcommand{\kps}{\bibitem[KPS]{KPS} * \emph{A. Kaibkhanov, D. Permyakov and A. Skopenkov.}
Realization of graphs with rotation, \url{http://www.turgor.ru/lktg/2005/3/index.htm}.}

\newcommand{\krzz}{\bibitem[Kr00]{Kr00} \emph{V. S. Krushkal.} Embedding obstructions and 4-dimensional thickenings of 2-complexes, Proc. Amer. Math. Soc. 128:12 (2000) 3683--3691. arXiv:math/0004058. }

\newcommand{\ksnn}{\bibitem[KS99]{KS99} * \emph{П. Кожевников и А. Скопенков.} Узкие деревья на плоскости, Мат. Образование. 1999. 2-3. С.~126-131.}

\newcommand{\kstz}{\bibitem[KS20]{KS20} \emph{R. Karasev and A. Skopenkov.}
Some `converses' to intrinsic linking theorems, Discr. Comp. Geom., 70:3 (2023), 921--930, arXiv:2008.02523.}


\newcommand{\ksto}{\bibitem[KS21]{KS21} * \emph{E. Kogan and A. Skopenkov.} A short exposition of the Patak-Tancer theorem on non-embeddability of $k$-complexes in $2k$-manifolds,  arXiv:2106.14010.}

\newcommand{\kstoe}{\bibitem[KS21e]{KS21e} \emph{E. Kogan and A. Skopenkov.}
Embeddings of $k$-complexes in $2k$-manifolds and minimum rank of partial symmetric matrices, arXiv:2112.06636v2.}

\newcommand{\kutt}{\bibitem[Ku23]{Ku23} \emph{W. K\"uhnel.} Generalized Heawood Numbers, The Electronic Journal of Combinatorics, 30:4 (2023) \#P4.17.}


\newcommand{\kuse}{\bibitem[Ku68]{Ku68} * \emph{К. Куратовский.} Топология. Т.~1,~2. М.: Мир, 1969.}

\newcommand{\kunfo}{\bibitem[Ku94]{Ku94} \emph{W. K\"uhnel.} Manifolds in the skeletons of convex polytopes, tightness, and generalized Heawood inequalities. In Polytopes: abstract, convex and computational (Scarborough, ON, 1993), volume 440 of NATO Adv. Sci. Inst. Ser. C Math. Phys. Sci., pp. 241--247. Kluwer
Acad. Publ., Dordrecht, 1994.}


\newcommand{\kunf}{\bibitem[Ku95]{Ku95} * \emph{W. K\"uhnel}, Tight Polyhedral Submanifolds and Tight Triangulations, Lecture Notes in Math. 1612, Springer, 1995.}

\newcommand{\lazz}{\bibitem[La00]{La00} \emph{F. Lasheras.} An obstruction to 3-dimensional thickening,
Proc. Amer. Math. Soc. 2000. 128. P.~893-902.}

\newcommand{\lfma}{\bibitem[LF]{LF} \url{http://www.map.mpim-bonn.mpg.de/Linking_form}}

\newcommand{\lloe}{\bibitem[LL18]{LL18} \emph{A.S. Levine and T. Lidman.} Simply connected, spineless 4-manifolds, Forum of Math., Sigma, 7 (2019) e14, 1--11, arxiv:1803.01765.}

\newcommand{\lo}{\bibitem[Lo]{Lo} M.~de~Longueville. Notes on the topological Tverberg theorem.
Discrete Math.  247 (2002), no.~1--3, 271--297.
(The paper first appeared in
Discrete Math. 241 (2001) 207--233, but the original version suffered from serious publisher's typesetting errors.)}

\newcommand{\loot}{\bibitem[Lo13]{Lo13} \emph{M. de Longueville.} A course in topological combinatorics. Universitext. Springer, New York (2013).}

\newcommand{\lssn}{\bibitem[LS69]{LS69} \emph{W. B. R. Lickorish and L. C. Siebenmann.}
Regular neighborhoods and the stable range,  Trans. Amer. Math. Soc.. 1969. 139. P.~207-230.}

\newcommand{\lsne}{\bibitem[LS98]{LS98} \emph{L. Lovasz and A. Schrijver,}
A Borsuk theorem for antipodal links and a spectral characterization of linklessly embeddable graphs, Proc. Amer. Math. Soc. 126:5 (1998), 1275-1285.}

\newcommand{\ltof}{\bibitem[LT14]{LT14} \emph{E. Lindenstrauss and M. Tsukamoto,} Mean dimension and an embedding problem: an example, Israel J. Math. 199 (2014).}


\newcommand{\lyzf}{\bibitem[LY04]{LY04} * \emph{Y. Lin and A. Yang,} On 3-cutwidth critical graphs, Discrete Mathematics, 275 (2004), 339--346.}

\newcommand{\lz}{\bibitem[LZ]{LZ} * \emph{S. Lando and A. Zvonkin.} Embedded Graphs. Springer.}

\newcommand{\maez}{\bibitem[Ma80]{Ma80} * R. Mandelbaum, {\em Four-Dimensional Topology: An introduction},
Bull. Amer. Math. Soc. (N.S.) 2 (1980) 1-159.}

\newcommand{\mast}{\bibitem[Ma73]{Ma73} \emph{С. В. Матвеев.} Специальные остовы кусочно-линейных многообразий, Мат. Сборник. 1973. 92. С.~282-293.}

\newcommand{\maste}{\bibitem[Ma73]{Ma73} \emph{S. V. Matveev.} Special skeletons of PL manifolds (in Russian), Mat. Sbornik. 1973. 92. P.~282-293.}

\newcommand{\manz}{\bibitem[Ma90]{Ma90} \emph{W. S.  Massey.} Homotopy classification of 3-component links of codimension greater than 2, Topol.  Appl. 34 (1990) 269--300.}

\newcommand{\mans}{\bibitem[Ma97]{Ma97} \emph{Yu. Makarychev.} A short proof of Kuratowski's graph planarity criterion, J. of Graph Theory, 25 (1997), 129--131.}

\newcommand{\matns}{\bibitem[Mat97]{Mat97} \emph{J. Matou\v sek.} A Helly-type theorem for unions of convex sets. Discr. Comp. Geom., 18:1 (1997) 1-12.}

\newcommand{\mazt}{\bibitem[Ma03]{Ma03} * \emph{J.~Matou{\v{s}}ek.} Using the {B}orsuk-{U}lam theorem:
Lectures on topological methods in combinatorics and geometry. Springer Verlag, 2008.}


\newcommand{\mazf}{\bibitem[Ma05]{Ma05} \emph{V. Manturov.} A proof of the Vasiliev conjecture on the planarity of singular links, Izv. RAN 2005.}

\newcommand{\metn}{\bibitem[Me29]{Me29} \emph{K. Menger.} \"Uber pl\"attbare Dreiergraphen und Potenzen nicht pl\"attbarer Graphen, Ergebnisse Math. Kolloq., 2 (1929) 30--31.}

\newcommand{\mezf}{\bibitem[Me04]{Me04} \emph{S. Melikhov.} Sphere eversions and realization of mappings, Trudy MIAN 247 (2004) 159-181 (in Russian) arXiv:math.GT/0305158.}

\newcommand{\mezs}{\bibitem[Me06]{Me06} \emph{S. A. Melikhov}, The van Kampen obstruction and its relatives, 	
Proc. Steklov Inst. Math 266 (2009), 142-176 (= Trudy MIAN 266 (2009), 149-183), arXiv:math/0612082.}

\newcommand{\meoo}{\bibitem[Me11]{Me11} \emph{S. A. Melikhov}, Combinatorics of embeddings, arXiv:1103.5457.}

\newcommand{\meos}{\bibitem[Me17]{Me17} \emph{S. Melikhov,} Gauss type formulas for link map invariants, arXiv:1711.03530.}

\newcommand{\meoe}{\bibitem[Me18]{Me18} \emph{S. A. Melikhov,} A triple-point Whitney trick, J. Topol. Anal., 2018, 1--6. arXiv:2210.04016.}


\newcommand{\metz}{\bibitem[Me20]{Me20} \emph{S. A. Melikhov,} Topological isotopy and Cochran's derived invariants, in `Topology, Geometry, and Dynamics: Rokhlin Memorial', Contemporary Mathematics, 772, AMS, Providence, RI, 2021. arXiv:2011.01409.}

\newcommand{\mett}{\bibitem[Me22]{Me22} \emph{S. A. Melikhov,} Embeddability of joins and products of polyhedra, Topol. Methods in Nonlinear Analysis, 60:1 (2022), 185-201. arXiv:2210.04015.}

\newcommand{\miff}{\bibitem[Mi54]{Mi54} \emph{J. Milnor,} Link groups, Ann. of Math. 59 (1954), 177--195.}

\newcommand{\miso}{\bibitem[Mi61]{Mi61} \emph{J. Milnor,} A procedure for killing homotopy groups of differentiable manifolds, Proc. Sympos. Pure Math, Vol. III (1961), 39--55.}

\newcommand{\mins}{\bibitem[Mi97]{Mi97} \emph{P. Minc.} Embedding simplicial arcs into the plane, Topol. Proc. 1997. 22. 305--340.}


\newcommand{\adnsvr}{\bibitem[MNS]{MNS} * \emph{А. Мирошников, О. Никитенко и А. Скопенков.} Циклы в графах и в гиперграфах: в направлении теории гомологий, arXiv:2406.16705.}
 
\newcommand{\dmnse}{\bibitem[MNS]{MNS} * \emph{A. Miroshnikov, O. Nikitenko, A. Skopenkov.}
Cycles in graphs and in hypergraphs: towards homology theory (in Russian), arXiv:2406.16705.}

\newcommand{\moss}{\bibitem[Mo77]{Mo77} * \emph{E. E. Moise.} Geometric Topology in Dimensions 2 and 3 (GTM), Springer-Verlag, 1977.}

\newcommand{\moen}{\bibitem[Mo89]{Mo89} \textit{B. Mohar}. An obstruction to embedding graphs in
surfaces. Discrete Math. 78 (1989) 135--142.}

\newcommand{\moze}{\bibitem[Mo08]{Mo08} \textit{T. Moriyama}. An invariant of embeddings of 3–manifolds in 6–manifolds and Milnor's triple linking number, J. Math. Sci. Univ. Tokyo, 18 (2011), 193--237. arXiv:0806.3733.}


\newcommand{\mrst}{\bibitem[MRS+]{MRS+} \emph{A. de Mesmay, Y. Rieck, E. Sedgwick, M. Tancer,}
Embeddability in $\R^3$ is NP-hard. arXiv:1708.07734.}

\newcommand{\mesczs}{\bibitem[MS06]{MS06} \emph{S.A. Melikhov, E.V. Shchepin,} The telescope approach to embeddability of compacta. arXiv:math.GT/0612085.}

\newcommand{\msos}{\bibitem[MS17]{MS17}  \emph{T. Maciazek, A. Sawicki.} Homology groups for particles on one-connected graphs
J. Math. Phys. 58, 062103 (2017). arXiv:1606.03414.}

\newcommand{\mstwof}{\bibitem[MST+]{MST+} \emph{J. Matou\v sek, E. Sedgwick, M. Tancer, U. Wagner}, Embeddability in the 3-sphere is decidable, Journal of the ACM 65:1 (2018) 1--49, arXiv:1402.0815.}


\newcommand{\mtzo}{\bibitem[MT01]{MT01} * \emph{B. Mohar and C. Thomassen.} Graphs on Surfaces.
The John Hopkins University Press, 2001.}

\newcommand{\mtwoz}{\bibitem[MTW10]{MTW10} \emph{J. Matou\v sek, M. Tancer, U. Wagner.} A geometric proof of
the colored Tverberg theorem, Discr. and Comp. Geometry, 47:2 (2012), 245--265. arXiv:1008.5275.}


\newcommand{\mtwoo}{\bibitem[MTW]{MTW} \emph{J. Matou\v sek, M. Tancer, U. Wagner.}
Hardness of embedding simplicial complexes in $\R^d$, J. Eur. Math. Soc. 13:2 (2011), 259--295. arXiv:0807.0336.}



\newcommand{\mwoe}{\bibitem[MW18]{MW18} * \emph{F. Manin, S. Weinberger.} Algorithmic aspects of immersibility and embeddability, Intern. Math. Res. Notices, rnae170. arXiv:1812.09413.}


\newcommand{\mwsn}{\bibitem[MW69]{MW69} * \emph{J. MacWilliams}. Orthogonal matrices over finite fields. Amer. Math. Monthly, 76 (1969) 152--164.}

\newcommand{\mwofo}{\bibitem[MW14]{MW14} \emph{I. Mabillard and U. Wagner.} Eliminating Tverberg Points, I. An Analogue of the Whitney Trick, Proc. of the 30th Annual Symp. on Comp. Geom. (SoCG'14), ACM, New York, 2014, pp. 171--180.}

\newcommand{\mwof}{\bibitem[MW15]{MW15} \emph{I. Mabillard and U. Wagner.}
Eliminating Higher-Multiplicity Intersections, I. A Whitney Trick for Tverberg-Type Problems. arXiv:1508.02349.}


\newcommand{\mwos}{\bibitem[MW16]{MW16} \emph{I. Mabillard and U. Wagner.} Eliminating Higher-Multiplicity Intersections, II. The Deleted Product Criterion in the $r$-Metastable Range. arXiv:1601.00876v2.}

\newcommand{\mwosd}{\bibitem[MW16']{MW16'} \emph{I. Mabillard and U. Wagner.} Eliminating Higher-Multiplicity Intersections, II. The Deleted Product Criterion in the r-Metastable Range,
Proceedings of the 32nd Annual Symposium on Computational Geometry (SoCG'16).}

\newcommand{\neno}{\bibitem[Ne91]{Ne91} \emph{S. Negami.} Ramsey theorems for knots, links and spatial graphs,
Trans. Amer. Math. Soc., 324 (1991), 527--541.}



\newcommand{\nizz}{\bibitem[Ni00]{Ni00} \emph{R. Nikkuni.} The second skew-symmetric cohomology group and spatial embeddings of graphs, J. Knot Theory Ram. 9 (2000), 387–411.}

\newcommand{\nkon}{\bibitem[NKS]{NKS} * \emph{L. T. Nguyen, J. Kim, B. Shim.}
Low-Rank Matrix Completion: A Contemporary Survey. arXiv:1907.11705.}

\newcommand{\noss}{\bibitem[No76]{No76} * \emph{С. П. Новиков.} Топология-1. М.: Наука, 1976. (Итоги науки и техники. ВИНИТИ. Современные проблемы математики. Основные направления, 12).}

\newcommand{\nszn}{\bibitem[NS09]{NS09} \emph{I. Novik and E. Swartz,} Socles of Buchsbaum modules, complexes and posets, Adv. Math. 222 (2009), 2059-2084. arXiv:0711.0783.}

\newcommand{\nwns}{\bibitem[NW97]{NW97} \emph{A. Nabutovsky, S. Weinberger}. Algorithmic aspects of homeomorphism problems. arXiv:math/9707232.}

\newcommand{\omoe}{\bibitem[Om18]{Om18} * \emph{А. Омельченко,} Теория графов. М.: МЦНМО, 2018.}

\newcommand{\orszo}{\bibitem[ORS]{ORS} \emph{A. Onischenko, D. Repov\v s and A. Skopenkov.}
Resolutions of 2-polyhedra by fake surfaces and embeddings into $\R^4$, Contemp. Math.  288 (2001) 396--400.}

\newcommand{\ossf}{\bibitem[OS74]{OS74} \emph{R. P. Osborne and R. S. Stevens.} Group presentations
corresponding to spines of 3-manifolds, I, Amer. J.~Math. 1974. 96. P.~454-471; II, Amer. J.~Math. 1977. 234.
P.~213-243; III, Amer. J.~Math. 1977. 234 P.~245-251.}


\newcommand{\oz}{\bibitem[Oz]{Oz} \emph{M. \"Ozaydin,} Equivariant maps for the symmetric group, unpublished,
\url{http://minds.wisconsin.edu/handle/1793/63829}.}

\newcommand{\panof}{\bibitem[Pan15]{Pan15} \emph{K. Panagiotis.} A note on the topology of irreducible $SO(3)$-manifolds, 	arXiv:1508.06150.}

\newcommand{\paof}{\bibitem[Pa15]{Pa15} \emph{S. Parsa,} On links of vertices in simplicial $d$-complexes embeddable in the Euclidean $2d$-space, Discrete Comput. Geom. 59:3 (2018), 663--679.
This is arXiv:1512.05164v4 up to numbering of sections, theorems etc.; we refer to numbering in arxiv version.
Correction: Discrete Comput. Geom. 64:3 (2020) 227--228.}

\newcommand{\paoe}{\bibitem[Pa18]{Pa18} \emph{S. Parsa,} On links of vertices in simplicial $d$-complexes
embeddable in the euclidean $2d$-space, arXiv:1512.05164v6.}

\newcommand{\patz}{\bibitem[Pa20]{Pa20} \emph{S. Parsa,} On links of vertices in simplicial $d$-complexes
embeddable in the euclidean $2d$-space, arXiv:1512.05164v8.}


\newcommand{\patzl}{\bibitem[Pa20]{Pa20} \emph{S. Parsa,}
Correction to: On the Links of Vertices in Simplicial $d$-Complexes Embeddable in the Euclidean $2d$-Space,
Discrete Comput. Geom. 64:3 (2020) 227--228.}

\newcommand{\patza}{\bibitem[Pa20]{Pa20} \emph{S. Parsa,} On the Smith classes, the van Kampen obstruction and embeddability of $[3]*K$, arXiv:2001.06478.}

\newcommand{\patzb}{\bibitem[Pa20b]{Pa20b} \emph{S. Parsa,} On the embeddability of $[3]*K$, arXiv:2001.06506.}

\newcommand{\pato}{\bibitem[Pa21]{Pa21} \emph{S. Parsa,} Instability of the Smith index under joins and applications to embeddability, Trans. Amer. Math. Soc. 375 (2022), 7149--7185, arXiv:2103.02563.}

\newcommand{\pak}{\bibitem[Pa]{Pa} * \emph{I. Pak}, Lectures on Discrete and Polyhedral Geometry, \url{http://www.math.ucla.edu/~pak/geompol8.pdf}.}

\newcommand{\peze}{\bibitem[Pe08]{Pe08} \emph{Д. Пермяков.} Классификация погружений графов в плоскость,
Вестник МГУ, сер.1, 2008, N5, 55-56.}

\newcommand{\peos}{\bibitem[Pe16]{Pe16} \emph{Д. Пермяков.} Матем. сб., 207:6 (2016),  93--112.}

\newcommand{\pest}{\bibitem[Pe72]{Pe72} * \emph{B. B. Peterson.} The Geometry of Radon's Theorem, Amer. Math. Monthly 79 (1972), 949-963.}


\newcommand{\prnf}{\bibitem[Pr95]{Pr95} * \emph{V. V. Prasolov.} Intuitive topology. Amer. Math. Soc., Providence, R.I., 1995.}

\newcommand{\prnfr}{\bibitem[Pr95]{Pr95} * \emph{В. В. Прасолов.} Наглядная топология. М.: МЦНМО, 1995.}


\newcommand{\przs}{\bibitem[Pr06]{Pr06} * \emph{V. V. Prasolov.}
Elements of Combinatorial and Differential Topology, 2006, GSM 74, Amer. Math. Soc., Providence, RI.}

\newcommand{\przsru}{\bibitem[Pr04]{Pr04} * \emph{В. В. Прасолов.}
Элементы комбинаторной и дифференциальной топологии. М.: МЦНМО, 2004. \url{http://www.mccme.ru/prasolov}.}

\newcommand{\przse}{\bibitem[Pr07]{Pr07} * \emph{V. V. Prasolov.} Elements of homology theory. 2007, GSM 74, Amer. Math. Soc., Providence, RI.}


\newcommand{\przseru}{\bibitem[Pr06]{Pr06} * \emph{В. В. Прасолов.} Элементы теории гомологий. М.: МЦНМО, 2006.}


\newcommand{\psns}{\bibitem[PS96]{PS96} * \emph{V. V. Prasolov, A. B. Sossinsky } Knots, Links, Braids, and 3-manifolds. Amer. Math. Soc. Publ., Providence, R.I., 1996.}


\newcommand{\pszf}{\bibitem[PS05]{PS05} * \emph{В. В. Прасолов и М. Б. Скопенков.}
Рамсеевская теория зацеплений, Мат. Просвещение. 2005. 9. С.~108--115.}

\newcommand{\pszfen}{\bibitem[PS05]{PS05} * \emph{V. V. Prasolov and M.B. Skopenkov.}
Ramsey link theory, Mat, Prosvescheniye, 9 (2005), 108--115.}

\newcommand{\psoo}{\bibitem[PS11]{PS11} \emph{Y. Ponty and C. Saule.} A combinatorial framework for designing (pseudoknotted) RNA algorithms, Proc. of the 11th Intern. Workshop on Algorithms in Bioinformatics, WABI'11, 250--269.}


\newcommand{\pstz}{\bibitem[PS20]{PS20} \emph{S. Parsa and A. Skopenkov.} On embeddability of joins and their `factors', Topol. Appl., 326 (2023) 108409, arXiv:2003.12285.}



\newcommand{\psszn}{\bibitem[PSS]{PSS} \emph{M. J. Pelsmajer, M. Schaefer and D. Stasi.} Strong Hanani-Tutte on the projective plane. SIAM J. Discrete Math., 23:3 (2009) 1317--1323.}

\newcommand{\psszs}{\bibitem[PSS]{PSS} \emph{M. J. Pelsmajer, M. Schaefer, and D. \v Stefankovi\v c.}
Removing even crossings. J. Combin. Theory Ser. B, 97(4):489–500, 2007.}

\newcommand{\pton}{\bibitem[PT19]{PT19} \emph{P. Pat\'ak and M. Tancer.} Embeddings of $k$-complexes into $2k$-manifolds. Discrete Comput. Geom. 71 (2024), 960--991. arXiv:1904.02404.}

\newcommand{\pw}{\bibitem[PW]{PW} \emph{I. Pak, S. Wilson}, G\lowercase{EOMETRIC REALIZATIONS OF POLYHEDRAL COMPLEXES}, \linebreak \url{http://www.math.ucla.edu/~pak/papers/Fary-full31.pdf}.}

\newcommand{\razf}{\bibitem[RA05]{RA05} * \emph{J. L. Ram\'irez Alfons\'in.} Knots and links in spatial graphs: a survey. Discrete Math., 302 (2005), 225--242.}

\newcommand{\rep}{\bibitem[Rep]{Rep} Referee's report on the paper ``Some `converses' to intrinsic linking theorems', \url{https://www.mccme.ru/circles/oim/materials/ksreport.pdf}}

\newcommand{\rnoo}{\bibitem[RN11]{RN11} * \emph{R. L. Ricca, B. Nipoti.} Gauss' linking number revisited.
J. of Knot Theory and Its Ramif. 20:10 (2011) 1325--1343. \url{https://www.maths.ed.ac.uk/~v1ranick/papers/ricca.pdf} .}

\newcommand{\rrstz}{\bibitem[RRS]{RRS} * \emph{V. Retinskiy, A. Ryabichev and A. Skopenkov.}
Motivated exposition of the proof of the Tverberg Theorem (in Russian).
Mat. Prosveschenie, 27 (2021), 166--169. arXiv:2008.08361.}


\newcommand{\rssec}{\bibitem[RS68]{RS68} \emph{C. P. Rourke and B. J. Sanderson,} Block bundles II, Ann. of Math. (2), 87 (1968) 431--483.}

\newcommand{\rsst}{\bibitem[RS72]{RS72} * \emph{C. P. Rourke and B. J. Sanderson,}
\newblock Introduction to Piecewise-Linear Topology,
\newblock \emph{Ergebn.\ der Math.} 69, Springer-Verlag, Berlin, 1972.}

\newcommand{\rsstr}{\bibitem[RS72]{RS72} * \emph{К. П. Рурк и Б. Дж. Сандерсон.} Введение в кусочно-линейную топологию, Москва. Мир. 1974.}

\newcommand{\rsns}{\bibitem[RS96]{RS96} * \emph{D. Repov\v s and A. B. Skopenkov.}
Embeddability and isotopy of polyhedra in Euclidean spaces,
Proc. of the Steklov Inst. Math. 1996. 212. P.~173-188.}

\newcommand{\rsne}{\bibitem[RS98]{RS98} \emph{D. Repov\v s and A. B. Skopenkov.}
A deleted product criterion for approximability of a map by embeddings, Topol. Appl. 1998. 87 P.~1-19.}

\newcommand{\rsnn}{\bibitem[RS99]{RS99} * \emph{D. Repov\v s and A. B. Skopenkov.} New results on embeddings of polyhedra and manifolds into Euclidean spaces,
Russ. Math. Surv. 54:6 (1999), 1149--1196.}


\newcommand{\rsnnd}{\bibitem[RS99']{RS99'} * \emph{Д. Реповш и А. Скопенков.}
Кольца Борромео и препятствия к вложимости, Труды МИРАН. 1999. 225. С.~331-338.}

\newcommand{\rszz}{\bibitem[RS00]{RS00} \emph{D. Repov\v s and A. Skopenkov.} Cell-like resolutions of polyhedra by special ones,  Colloq. Math. 2000. 86:2. P. 231--237.}

\newcommand{\rszzd}{\bibitem[RS00']{RS00'} * \emph{Д. Реповш и А. Скопенков.} Характеристические классы для начинающих, Мат. Просвещение. 2000. 4. С.~151-176.}

\newcommand{\rszo}{\bibitem[RS01]{RS01} \emph{D. Repovs and A. Skopenkov.} On contractible $n$-dimensional compacta, non-embeddable into $\R^{2n}$, Proc. Amer. Math. Soc. 129 (2001) 627--628.}

\newcommand{\rszt}{\bibitem[RS02]{RS02} * \emph{Д. Реповш и А. Скопенков.} Теория препятствий для начинающих,
Мат. Просвещение. 2002. 6. C.~60-77.}

\newcommand{\rszf}{\bibitem[RS04]{RS04} \emph{N. Robertson and P. Seymour.} Graph Minors. XX. Wagner's conjecture, J. of Comb. Theory, B, 92:2 (2004) 325--357.}

\newcommand{\rssnf}{\bibitem[RSS]{RSS95} \emph{D. Repov\v s, A. B. Skopenkov  and E. V. \v S\v cepin.}
On uncountable collections of continua and their span, Colloq. Math. 1995. 69:2. P.~289-296.}

\newcommand{\rssnfd}{\bibitem[RSS']{RSS95'} \emph{D. Repov\v s, A. B. Skopenkov and E. V \v S\v cepin.}
On embeddability of $X\times I$ into Euclidean space, Houston J.~Math. 1995. 21. P.~199-204.}

\newcommand{\rssz}{\bibitem[RSS+]{RSSZ} * \emph{A. Rukhovich, A. Skopenkov, M. Skopenkov, A. Zimin},
Realizability of hypergraphs, \url{https://www.turgor.ru/lktg/2013/1/1-1en.pdf} .}


\newcommand{\rstnt}{\bibitem[RST']{RST93} \emph{N. Robertson, P. Seymour and R. Thomas}, Linkless embeddings of graphs in 3-space, Bull. of the Amer. Math. Soc., 21 (1993) 84--89.}

\newcommand{\rstno}{\bibitem[RST]{RST91} * \emph{N. Robertson, P. Seymour and R. Thomas}, A survey of
linkless embeddings, Graph Structure Theory (Seattle, WA, 1991), Contemp. Math. 147, (1993) 125--136.}


\newcommand{\rwzl}{\bibitem[RWZ+]{RWZ+} \emph{Y. Ren, C. Wen, S. Zhen, N. Lei, F. Luo, D.X. Gu},
Characteristic class of isotopy for surfaces, J. Syst. Sci. Complex. 33 (2020) 2139--2156.}

\newcommand{\saeo}{\bibitem[Sa81]{Sa81} \emph{H. Sachs.} On spatial representation of finite graphs,
in: Finite and infinite sets (Eger, 1981), 649--662, Colloq. Math. Soc. Janos Bolyai, 37, North-Holland, Amsterdam, 1984.}

\newcommand{\sano}{\bibitem[Sa91]{Sa91} \emph{K. S. Sarkaria.}
A one-dimensional Whitney trick and Kuratowski's graph planarity criterion, Israel J.~Math. 73 (1991), 79--89.} 


\newcommand{\sanov}{\bibitem[Sa91g]{Sa91g} \emph{K. S. Sarkaria.} A generalized Van Kampen-Flores theorem, Proc. Amer. Math. Soc. 111 (1991), 559--565.}

\newcommand{\sant}{\bibitem[Sa92]{Sa92} \emph{K. S. Sarkaria.} Tverberg’s theorem via number fields. Israel J. Math., 79:317–320, 1992.}

\newcommand{\sann}{\bibitem[Sa99]{Sa99} O. Saeki {\em On punctured 3-manifolds in 5-sphere}, Hiroshima Math. J. 29 (1999) 255--272.}


\newcommand{\sazz}{\bibitem[Sa00]{Sa00} \emph{K. S. Sarkaria.} Tverberg partitions and Borsuk-Ulam theorems. Pacific J. Math., 196:1 (2000) 231--241.}

\newcommand{\sczf}{\bibitem[Sc04]{Sc04} \emph{T. Sch\"oneborn.} On the Topological Tverberg Theorem, arXiv:math/0405393.}


\newcommand{\scot}{\bibitem[Sc13]{Sc13} * \emph{M. Schaefer.} Hanani-Tutte and related results. In Geometry --- intuitive, discrete, and convex, Bolyai Soc. Math. Stud., 24 (2013), 259--299.
\url{http://ovid.cs.depaul.edu/documents/htsurvey.pdf} }


\newcommand{\sctz}{\bibitem[Sc20]{Sc20} \emph{M. Schaefer.} The Graph Crossing Number and
its Variants: A Survey. The Electr. J. of Comb. (2020), DS21, \url{https://www.combinatorics.org/files/Surveys/ds21/ds21v5-2020.pdf}}

\newcommand{\scef}{\bibitem[Sc84]{Sc84} \emph{E.~V.~\v S\v cepin.} Soft mappings of manifolds, Russian Math. Surveys, 39:5 (1984).}

\newcommand{\shfs}{\bibitem[Sh57]{Sh57} \emph{A. Shapiro,} Obstructions to the embedding of a complex in a Euclidean space, I, The first obstruction, Ann. Math. 66 (1957), 256--269.}


\newcommand{\shen}{\bibitem[Sh89]{Sh89} * \emph{Ю. А. Шашкин,} Неподвижные точки, М., Наука, 1989.}

\newcommand{\shoe}{\bibitem[Sh18]{Sh18} * \emph{S. Shlosman},  Topological Tverberg Theorem: the proofs and the counterexamples, Russian Math. Surveys, 73:2 (2018), 175–182. arXiv:1804.03120.}

\newcommand{\sisn}{\bibitem[Si69]{Si69} \emph{K. Sieklucki.} Realization of mappings, Fund. Math. 1969. 65. P.~325-343.}

\newcommand{\sios}{\bibitem[Si16]{Si16} \emph{S. Simon,} Average-Value Tverberg Partitions via Finite Fourier Analysis, Israel J. Math., 216 (2016), 891-904, arXiv:1501.04612.}


\newcommand{\sknf}{\bibitem[Sk94]{Sk94} \emph{А. Скопенков.} Геометрическое доказательство теоремы
Нойвирта об утолщаемости 2-мерных полиэдров, Math. Notes. 1995. 58:5. P.~1244-1247.}


\newcommand{\skns}{\bibitem[Sk97]{Sk97} \emph{A. Skopenkov,} On the deleted product criterion for embeddability of manifolds in $\R^m$, Comment. Math. Helv. 72 (1997), 543--555.}

\newcommand{\skne}{\bibitem[Sk98]{Sk98} \emph{A. B. Skopenkov.} On the deleted product criterion for embeddability in $\R^m$, Proc. Amer. Math. Soc., 126:8 (1998), 2467-2476.}

\newcommand{\skzz}{\bibitem[Sk00]{Sk00} \emph{A. Skopenkov,} On the generalized Massey--Rolfsen invariant for link maps, Fund. Math. 165 (2000), 1--15.}

\newcommand{\skzt}{\bibitem[Sk02]{Sk02} \emph{A. Skopenkov,} On the Haefliger-Hirsch-Wu invariants for embeddings and immersions, Comment. Math. Helv. 77 (2002), 78--124.}

\newcommand{\skzth}{\bibitem[Sk03]{Sk03} \emph{M. Skopenkov,} Embedding products of graphs into Euclidean spaces,
Fund. Math. 179 (2003),~191--198, arXiv:0808.1199.}

\newcommand{\skzthd}{\bibitem[Sk03']{Sk03'} \emph{M. Skopenkov,} On approximability by embeddings of cycles in the plane, Topol. Appl. 134 (2003),~1--22, arXiv:0808.1187.}

\newcommand{\skzf}{\bibitem[Sk05]{Sk05} * \emph{A. Skopenkov,}
On the Kuratowski graph planarity criterion, Mat. Prosveschenie, 9 (2005), 116-128. arXiv:0802.3820.}


\newcommand{\skzs}{\bibitem[Sk05i]{Sk05i} \emph{A. Skopenkov,} A new invariant and parametric connected sum of embeddings, Fund. Math. 197 (2007) 253--269. arxiv:math/0509621.}

\newcommand{\skzei}{\bibitem[Sk05]{Sk05} \emph{A.  Skopenkov,} A classification of smooth embeddings of
4-manifolds in 7-space, I, Topol. Appl., 157 (2010) 2094--2110. arXiv:math/0512594.}

\newcommand{\skze}{\bibitem[Sk06]{Sk06} * \emph{A. Skopenkov,} Embedding and knotting of manifolds in Euclidean spaces, London Math. Soc. Lect. Notes, 347 (2008) 248--342. arXiv:math/0604045.}

\newcommand{\skzsi}{\bibitem[Sk06']{Sk06'} \emph{A. Skopenkov,} A classification of smooth embeddings of 3-manifolds in 6-space, Math. Zeitschrift, 260:3 (2008) 647--672. arxiv:math/0603429.}

\newcommand{\skozp}{\bibitem[Sk08]{Sk08} \emph{A.  Skopenkov,} Embeddings of $k$-connected $n$-manifolds into
$\R^{2n-k-1}$. arxiv:math/0812.0263; earlier version published in Proc. Amer. Math. Soc., 138 (2010) 3377--3389.}

\newcommand{\skoz}{\bibitem[Sk10]{Sk10} * \emph{А. Скопенков,} Вложения в плоскость графов с вершинами степени 4,
Мат. просвещение, 21 (2017), arXiv:1008.4940.}

\newcommand{\skoo}{\bibitem[Sk11]{Sk11} \emph{M. Skopenkov,} When is the set of embeddings finite up to isotopy? Intern. J. Math. 26:7 (2015), 28 pp. arXiv:1106.1878.}

\newcommand{\skofo}{\bibitem[Sk14]{Sk14} \emph{A. Skopenkov,} How do autodiffeomorphisms act on embeddings, Proc. A of the Royal Society of Edinburgh, 148:4 (2018), 835--848. arXiv:1402.1853.}

\newcommand{\sks}{\bibitem[Sk14]{Sk14} * \emph{A. Skopenkov,} Realizability of hypergraphs and intrinsic linking  theory, Mat. Prosveschenie, 32 (2024), 125--159, arXiv:1402.0658.}

\newcommand{\sksr}{\bibitem[Sk14]{Sk14} * \emph{А. Скопенков,} Реализуемость гиперграфов и неотъемлемая зацепленность, Мат. просвещение, 32 (2024), 125--159. arXiv:1402.0658.}


\newcommand{\skof}{\bibitem[Sk15]{Sk15} * \emph{А. Скопенков,} Алгебраическая топология с геометрической точки зрения, Москва, МЦНМО, 2015 (1е издание).}

\newcommand{\skofe}{\bibitem[Sk15]{Sk15} * \emph{A. Skopenkov,} Algebraic Topology From Geometric Viewpoint (in Russian), MCCME, Moscow, 2015 (1st edition). }

\newcommand{\skofel}{\bibitem[Sk15e]{Sk15e} * \emph{А. Скопенков,} Алгебраическая топология
с геометрической точки зрения, эл. версия, \url{http://www.mccme.ru/circles/oim/home/combtop13.htm\#photo}}


\newcommand{\skotzr}{\bibitem[Sk20]{Sk20} * \emph{А. Скопенков,} Алгебраическая топология с геометрической точки зрения, Москва, МЦНМО, 2020 (2е издание).
Обновляемая версия части книги: \url{http://www.mccme.ru/circles/oim/obstruct.pdf}}

\newcommand{\skotz}{\bibitem[Sk20]{Sk20} * \emph{A. Skopenkov,} Algebraic Topology From Geometric Standpoint (in Russian), MCCME, Moscow, 2020 (2nd edition).
Update of a part: \url{http://www.mccme.ru/circles/oim/obstruct.pdf} .
Part of the English translation: \url{https://www.mccme.ru/circles/oim/obstructeng.pdf}.}


\newcommand{\skofp}{\bibitem[Sk15]{Sk15} \emph{A. Skopenkov,} Classification of knotted tori,
Proc. A of the Royal Soc. of Edinburgh, 150:2 (2020), 549-567. Full version: arXiv:1502.04470.}


\newcommand{\skos}{\bibitem[Sk16]{Sk16} * \emph{A. Skopenkov,} A user's guide to the topological Tverberg Conjecture, arXiv:1605.05141v5. Abridged earlier published version: Russian Math. Surveys, 73:2 (2018), 323--353.}



\newcommand{\skosd}{\bibitem[Sk16']{Sk16'} * \emph{A. Skopenkov,} Stability of intersections of graphs in the plane and the van Kampen obstruction, Topol. Appl. 240(2018) 259--269, arXiv:1609.03727.}


\newcommand{\skosc}{\bibitem[Sk16c]{Sk16c} * \emph{A. Skopenkov,}  Embeddings in Euclidean space: an introduction to their classification, to appear in Boll. Man. Atl. 
\url{http://www.map.mpim-bonn.mpg.de/Embeddings_in_Euclidean_space:_an_introduction_to_their_classification}}

\newcommand{\skosie}{\bibitem[Sk16e]{Sk16e} * \emph{A. Skopenkov,} Embeddings just below the stable range: classification, to appear in Boll. Man. Atl.
\url{http://www.map.mpim-bonn.mpg.de/Embeddings_just_below_the_stable_range:_classification}}

\newcommand{\skost}{\bibitem[Sk16t]{Sk16t} * \emph{A. Skopenkov,} 3-manifolds in 6-space, to appear in Boll. Man. Atl. \url{http://www.map.mpim-bonn.mpg.de/3-manifolds_in_6-space}.}

\newcommand{\skosf}{\bibitem[Sk16f]{Sk16f} * \emph{A. Skopenkov,} 4-manifolds in 7-space, to appear in Boll. Man. Atl. \url{http://www.map.mpim-bonn.mpg.de/4-manifolds_in_7-space}.}

\newcommand{\skosh}{\bibitem[Sk16h]{Sk16h} * \emph{A. Skopenkov,} High codimension links, to appear in Boll. Man. Atl.
\linebreak
\url{http://www.map.mpim-bonn.mpg.de/High_codimension_links}.}

\newcommand{\skosi}{\bibitem[Sk16i]{Sk16i} * \emph{A. Skopenkov,} Isotopy, submitted to Boll. Man. Atl.
\url{http://www.map.mpim-bonn.mpg.de/Isotopy}.}

\newcommand{\skosk}{\bibitem[Sk16k]{Sk16k} * \emph{A. Skopenkov,} Knotted tori,
\url{http://www.map.mpim-bonn.mpg.de/Knotted_tori}.}

\newcommand{\skoss}{\bibitem[Sk16s]{Sk16s} * \emph{A. Skopenkov,} Knots, i.e. embeddings of spheres,
\linebreak
\url{http://www.map.mpim-bonn.mpg.de/Knots,_i.e._embeddings_of_spheres}.}

\newcommand{\skose}{\bibitem[Sk17]{Sk17} \emph{A. Skopenkov,}
Eliminating higher-multiplicity intersections in the metastable dimension range. arXiv:1704.00143.}

\newcommand{\skosed}{\bibitem[Sk17v]{Sk17v} * \emph{A. Skopenkov,}
On van Kampen-Flores, Conway-Gordon-Sachs and Radon theorems,  arXiv:1704.00300.}

\newcommand{\sk}{\bibitem[Sk17o]{Sk17o} \emph{A. Skopenkov,} On the metastable Mabillard-Wagner conjecture.  arXiv:1702.04259.}

\newcommand{\skmos}{\bibitem[Sk17d]{Sk17d} \emph{M. Skopenkov}. Discrete field theory: symmetries and conservation laws, arXiv:1709.04788.}

\newcommand{\skoe}{\bibitem[Sk18]{Sk18} * \emph{A. Skopenkov.} Invariants of graph drawings in the plane.
Arnold Math. J., 6 (2020) 21--55; full version: arXiv:1805.10237.}


\newcommand{\skoer}{\bibitem[Sk18]{Sk18} * \emph{А. Скопенков,} Инварианты изображений графов на плоскости,
Мат. просвещение, 31 (2023), 74-127. arXiv:1805.10237.}


\newcommand{\sktthd}{\bibitem[Sk23']{Sk23'} * \emph{A. Skopenkov.} Invariants of graph drawings in the plane (in Russian). Mat. Prosveschenie, 31 (2023), 74-127. arXiv:1805.10237.}

\newcommand{\skoeo}{\bibitem[Sk18o]{Sk18o} * \emph{A. Skopenkov.} A short exposition of S. Parsa's theorems on intrinsic linking and non-realizability. Discr. Comp. Geom. 65:2 (2021), 584--585; full version:  arXiv:1808.08363.}


\newcommand{\skona}{\bibitem[Sk19]{Sk19} * \emph{A. Skopenkov,} A short exposition of the Levine-Lidman example of spineless 4-manifolds, arXiv:1911.07330.}

\newcommand{\sktze}{\bibitem[Sk21m]{Sk21m} * \emph{A. Skopenkov.} Mathematics via Problems. Part 1: Algebra. Amer. Math. Soc., Providence, 2021. Preliminary version: \url{https://www.mccme.ru/circles/oim/algebra_eng.pdf}}

\newcommand{\sktz}{\bibitem[Sk20u]{Sk20u} * \emph{A. Skopenkov.} A user's guide to basic knot and link theory,
in: Topology, Geometry, and Dynamics, Contemporary Mathematics, vol. 772, Amer. Math. Soc., Providence, RI, 2021, pp. 281--309.
Russian version: Mat. Prosveschenie 27 (2021), 128--165. arXiv:2001.01472.}

\newcommand{\sktzru}{\bibitem[Sk20u]{Sk20u} * \emph{А. Скопенков.} Основы теории узлов и зацеплений для пользователя, Мат. просвещение, 27 (2021), 128--165. arXiv:2001.01472.}

\newcommand{\sktzo}{\bibitem[Sk20o]{Sk20o} \emph{A. Skopenkov.} On some results of S. Abramyan and T. Panov, arXiv:2005.11152.}

\newcommand{\sktzr}{\bibitem[Sk20e]{Sk20e} * \emph{A. Skopenkov.}
Extendability of simplicial maps is undecidable, Discr. Comp. Geom., 69:1 (2023), 250--259, arXiv:2008.00492.}


\newcommand{\sktzd}{\bibitem[Sk21d]{Sk21d} * \emph{A. Skopenkov.}
On different reliability standards in current mathematical research, arXiv:2101.03745.
More often updated version: \url{https://www.mccme.ru/circles/oim/rese_inte.pdf}.}

\newcommand{\sktt}{\bibitem[Sk22]{Sk22} * \emph{A. Skopenkov.} Invariants of embeddings of 2-surfaces in 3-space,
arXiv:2201.10944.}

\newcommand{\skttn}{\bibitem[Sk22n]{Sk22n} * \emph{A. Skopenkov}, Netflix problem and realization of (hyper)graphs, \url{https://www.mccme.ru/circles/oim/home/netflix20sep.pdf}}

\newcommand{\sktth}{\bibitem[Sk23]{Sk23} \emph{A. Skopenkov.}  To S. Parsa's theorem on embeddability of joins, arXiv:2302.11537.}

\newcommand{\sktf}{\bibitem[Sk24]{Sk24} * \emph{A. Skopenkov.} Double and triple linking numbers in space (in Russian). Mat. Prosveschenie, 33 (2024), 87--132.}

\newcommand{\sktfr}{\bibitem[Sk24]{Sk24} * \emph{А. Скопенков.} Двойные и тройные коэффициенты зацепления в пространстве. Мат. просвещение, 33 (2024), 87--132.}

\newcommand{\sktfb}{\bibitem[Sk24]{Sk24} \emph{A. Skopenkov.} The band connected sum and the second Kirby move for higher-dimensional links (full version), arXiv:2406.15367.}


\newcommand{\sktfe}{\bibitem[Sk24]{Sk24} \emph{A. Skopenkov.}
Embeddings of $k$-complexes in $2k$-manifolds and minimum rank of partial symmetric matrices, arXiv:2112.06636v4.}

\newcommand{\skd}{\bibitem[Sk]{Sk} * \emph{А. Скопенков.} Алгебраическая топология с алгоритмической точки зрения, 
\url{http://www.mccme.ru/circles/oim/algor.pdf}.}

\newcommand{\skde}{\bibitem[Sk]{Sk} * \emph{A. Skopenkov.} Algebraic Topology From Algorithmic Standpoint, draft of a book, mostly in Russian,
\url{http://www.mccme.ru/circles/oim/algor.pdf}.}


\newcommand{\skon}{\bibitem[Skw]{Skw} * \emph{A. Skopenkov.} Whitney trick for eliminating multiple intersections, slides for talks at St Petersburg, Brno, Kiev, Moscow,  \url{https://www.mccme.ru/circles/oim/eliminat_talk.pdf}.}

\newcommand{\skl}{\bibitem[EEF]{EEF} * {\it Proposed by D. Eliseev, A. Enne, M. Fedorov, A. Glebov, N. Khoroshavkina, E. Morozov, A. Skopenkov, R. \v Zivaljevi\'c.}
A user's guide to knot and link theory, \url{https://www.turgor.ru/lktg/2019/3} .}

\newcommand{\skr}{\bibitem[Skr]{Skr} * \emph{A. Skopenkov.} Realizability of hypergraphs, slides for talks,  \url{https://www.mccme.ru/circles/oim/algor1_beamer.pdf}.}


\newcommand{\skt}{\bibitem[Skt]{Skt} * \emph{A. Skopenkov.} Transparent anonymous peer review,
\linebreak
\url{https://www.mccme.ru/circles/oim/home/transp_peer_review.htm} .}

\newcommand{\rslktg}{\bibitem[KRR+]{RRSl} * Towards higher-dimensional combinatorial geometry, presented by
E. Kogan, V. Retinskiy, E. Riabov and A. Skopenkov, \url{https://www.mccme.ru/circles/oim/multicomb.pdf} .}


\newcommand{\sm}{\bibitem[Sm]{Sm} S. Smirnov.}

\newcommand{\sper}{\bibitem[Sp]{Sp} * Sperner's lemma defeats the rental harmony problem, \url{https://www.youtube.com/watch?v=7s-YM-kcKME}.}

\newcommand{\sset}{\bibitem[SS83]{SS83} \emph{Е. В. Щепин, М. А. Штанько.} Спектральный критерий вложимости компактов в евклидовы пространства, Труды Ленинградской Международной Топологической конференции. Л.: Наука, 1983. С.~135-142.}

\newcommand{\ssnt}{\bibitem[SS92]{SS92} \emph{J.~Segal and S.~Spie\.z.} Quasi embeddings and embeddings of polyhedra in $\R^m$,  Topol. Appl., 45 (1992) 275--282.}

\newcommand{\sszt}{\bibitem[SS03]{SS03} \emph{F. W. Simmons and F. E. Su.}
Consensus-halving via theorems of Borsuk-Ulam and Tucker, Math. Social Sciences 45 (2003) 15–25. \url{https://www.math.hmc.edu/~su/papers.dir/tucker.pdf}.}

\newcommand{\ssot}{\bibitem[SS13]{SS13} \emph{M. Schaefer and D. \v Stefankovi\v c.} Block additivity of $\Z_2$-embeddings. In Graph drawing, volume 8242 of Lecture Notes in Comput. Sci., 185--195.
Springer, Cham, 2013. \url{http://ovid.cs.depaul.edu/documents/genus.pdf}}

\newcommand{\sstt}{\bibitem[SS23]{SS23} \emph{A. Skopenkov and O. Styrt,} Embeddability of joinpowers, and minimal rank of partial matrices, arXiv:2305.06339.}

\newcommand{\sssne}{\bibitem[SSS]{SSS} \emph{J. Segal, A. Skopenkov and S. Spie\. z.}
Embeddings of polyhedra in $\R^m$ and the deleted product obstruction, Topol. Appl., 85 (1998), 225-234.}

\newcommand{\sstnf}{\bibitem[SST95]{SST95} \emph{R. S. Simon, S. Spie\. z and H. Toru\'nczyk.}
T\lowercase{HE EXISTENCE OF EQUILIBRIA IN CERTAIN GAMES, SEPARATION FOR FAMILIES OF CONVEX FUNCTIONS
AND A THEOREM OF BORSUK-ULAM TYPE}, Israel J. Math 92 (1995) 1--21.}

\newcommand{\sstzt}{\bibitem[SST02]{SST02} \emph{R. S. Simon, S. Spie\. z and H. Toru\'nczyk.}
E\lowercase{QUILIBRIUM EXISTENCE AND TOPOLOGY IN SOME REPEATED GAMES WITH INCOMPLETE INFORMATION},
Trans. Amer. Math. Soc. 354:12 (2002) 5005-5026.}

\newcommand{\stez}{\bibitem[ST80]{ST80} * {\it H.~Seifert and W.~Threlfall.}
A textbook of topology, v~89 of {\em Pure and Applied Mathematics}.
Academic Press, New York-London, 1980.}


\newcommand{\stzs}{\bibitem[ST07]{ST07} * \emph{А. Скопенков и А. Телишев.}
И вновь о критерии Куратовского планарности графов, Мат. Просвещение, 11 (2007), 159--160.}

\newcommand{\stzse}{\bibitem[ST07]{ST07} * \emph{A. Skopenkov and A. Telishev}, Once again on the Kuratowski graph planarity criterion, Mat. Prosveschenie, 11 (2007), 159-160. arXiv:0802.3820.}

\newcommand{\stos}{\bibitem[ST17]{ST17} \emph{A. Skopenkov  and M. Tancer,}
Hardness of almost embedding simplicial complexes in $\R^d$, Discr. Comp. Geom., 61:2 (2019), 452--463. arXiv:1703.06305.}

\newcommand{\stno}{\bibitem[ST91]{ST91} \emph{S.~Spie\. z and H.~Toru\'nczyk}, Moving compacta in $\R^m$ apart,
Topol. Appl. 41 (1991), 193--204.}

\newcommand{\sttt}{\bibitem[St24]{St24} \emph{M. Starkov,} An example of an `unlinked' set of $2k+3$ points in $2k$-space, arXiv:2402.09002.}

\newcommand{\sunt}{\bibitem[Su]{Su} * \emph{Д. Судзуки.} Основы дзэн-буддизма. Наука дзэн --- ум дзэн. Киев: Преса Украiни. 1992.}

\newcommand{\stwh}{\bibitem[SW]{SW} * \url{http://www.map.mpim-bonn.mpg.de/Stiefel-Whitney_characteristic_classes}}

\newcommand{\sz}{\bibitem[SZ05]{SZ} \emph{T. Sch\"oneborn and G. Ziegler}, The Topological Tverberg Theorem and Winding Numbers, J. Comb. Theory, Ser. A, 112:1 (2005) 82--104, arXiv:math/0409081.}

\newcommand{\szno}{\bibitem[Sz91]{Sz91} \emph{A.~Sz\"ucs,} On the cobordism groups of immersions and embeddings,
Math. Proc. Camb. Phil. Soc., 109 (1991) 343--349.}

\newcommand{\ta}{\bibitem[Ta]{Ta} * Handbook of Graph Drawing and Visualization. ed. by R. Tamassia, CRC Press, 2016.}


\newcommand{\tanf}{\bibitem[Ta95]{Ta95} \emph{K. Taniyama,} Homology classification of spatial embeddings of a graph, Topol. Appl. 65 (1995) 205--228.}

\newcommand{\tazz}{\bibitem[Ta00]{Ta00} \emph{K. Taniyama,} Higher dimensional links in a simplicial complex embedded in a sphere, Pacific Jour. of Math. 194:2 (2000), 465-467.}

\newcommand{\theo}{\bibitem[Th81]{Th81} * \emph{C.~Thomassen,} Kuratowski's theorem, J.~Graph. Theory 5 (1981), 225--242.}

\newcommand{\tooo}{\bibitem[To11]{To11} \emph{Tonkonog D.} Embedding 3-manifolds with boundary into closed 3-manifolds, Topol. Appl. 158 (2011), 1157-1162. arXiv:1003.3029.}


\newcommand{\tsbzf}{\bibitem[TSB]{TSB} \emph{D. M. Thilikos, M. Serna and H. L. Bodlaender},
Cutwidth I: A linear time fixed parameter algorithm, J. of Algorithms, 56:1 (2005), 1--24.}


\newcommand{\tsbzfd}{\bibitem[TSB05']{TSB05'} \emph{D. M. Thilikos, M. Serna and H. L. Bodlaender},
Cutwidth II: , J. of Algorithms, 56:1 (2005), 25--49.}



\newcommand{\umse}{\bibitem[Um78]{Um78} \emph{B. Ummel.} The product of nonplanar complexes does not imbed in 4-space, Trans. Amer. Math. Soc., 242 (1978) 319--328.}




\newcommand{\vant}{\bibitem[Va92]{Va92} * \emph{V.~A.~Vassiliev.} Complements of discriminants of smooth maps: Topology and applications, Amer. Math. Soc., Providence, RI, 1992 (рус. перевод: В. А. Васильев, Топология дополнений к дискриминантам, Фазис, Москва, 1997).}

\newcommand{\val}{\bibitem[Val]{Val} * \url{https://en.wikipedia.org/wiki/Valknut}}


\newcommand{\vi}{\bibitem[Vi]{Vi} * \emph{O. Viro.}
Some integral calculus based on Euler characteristic, Lect. Notes in Math. 1346.}

\newcommand{\vizt}{\bibitem[Vi02]{Vi02} * \emph{Э. Б. Винберг.} Курс алгебры. Москва. Факториал Пресс. 2002.}

\newcommand{\vizteng}{\bibitem[Vi02]{Vi02} * \emph{E. B. Vinberg.} A Course in Algebra. Graduate Studies in Mathematics, vol. 56. 2003.}

\newcommand{\vinhzs}{\bibitem[VINH07]{VINH07} * \emph{О. Я. Виро, О. А. Иванов, Н. Ю. Нецветаев и В. М. Харламов.}
Элементарная топология, МЦНМО. 2007.}

\newcommand{\vktt}{\bibitem[vK32]{vK32} \emph{E.~R.~van~Kampen}, Komplexe in euklidischen R\"aumen, Abh. Math. Sem. Hamburg, 9 (1933) 72--78; Berichtigung dazu, 152--153.}

\newcommand{\kafo}{\bibitem[vK41]{vK41} \emph{E. R. van Kampen,} Remark on the address of S. S. Cairns,
in Lectures in Topology, 311--313, University of Michigan Press, Ann Arbor, MI, 1941.}

\newcommand{\vo}{\bibitem[Vo96]{vo96} \emph{A. Yu. Volovikov,} On a topological generalization of the Tverberg theorem. Math. Notes 59:3 (1996), 324--326.}

\newcommand{\vopns}{\bibitem[Vo96v]{Vo96v} \emph{A. Yu. Volovikov,} On the van Kampen-Flores Theorem.
Math. Notes 59:5 (1996), 477--481.}


\newcommand{\vznt}{\bibitem[VZ93]{VZ93} \emph{A. Vu\v ci\'c and R. T. \v Zivaljevi\'c}, Note on a conjecture of Sierksma, Discr. Comput. Geom. 9 (1993), 339-349.}

\newcommand{\vzzn}{\bibitem[VZ09]{VZ09} \emph{S. T. Vre\'cica and R. T. \v Zivaljevi\'c},  Chessboard complexes
indomitable, J. of Comb. Theory, Ser. A 118:7 (2011), 2157--2166. arXiv:0911.3512.}

\newcommand{\walst}{\bibitem[Wa62]{Wa62} \emph{C.~T.~C.~Wall}, Classification of $(n-1)$-connected $2n$-manifolds, Ann. of Math., 75 (1962) 163--189.}


\newcommand{\wallss}{\bibitem[Wa67]{Wa67} \emph{C.~T.~C.~Wall.} Classification problems in differential topology, IV, Thickenings, Topology 1966. 5. P. 73--94.}

\newcommand{\waldss}{\bibitem[Wa67m]{Wa67m} \emph{F. Waldhausen.} Eine Klasse von 3-dimensional Mannigfaltigkeiten, I. Invent. Math. 1967. 3. P.~308-333.}

\newcommand{\walsz}{\bibitem[Wa70]{Wa70} \emph{C. T. C. Wall,} Surgery on compact manifolds,
1970, Academic Press, London.}

\newcommand{\wess}{\bibitem[We67]{We67} \emph{C.~Weber.} Plongements de poly\`edres dans le domain metastable, Comment. Math. Helv. 42 (1967), 1--27.}

\newcommand{\whit}{\bibitem[Wl]{Wl} * \url{https://en.wikipedia.org/wiki/Whitehead_link}}

\newcommand{\winum}{\bibitem[Wn]{Wn} * \url{https://en.wikipedia.org/wiki/Winding_number}}

\newcommand{\wrss}{\bibitem[Wr77]{Wr77} \emph{P. Wright.} Covering 2-dimensional polyhedra by 3-manifolds spines.
Topology. 16 (1977), 435--439.}

\newcommand{\wufe}{\bibitem[Wu58]{Wu58} \emph{W. T. Wu.} On the realization of complexes in a euclidean space (in Chinese): I, Sci Sinica, 7 (1958) 251--297; II, Sci Sinica, 7 (1958) 365--387; III, Sci Sinica, 8 (1959) 133--150.}

\newcommand{\wufn}{\bibitem[Wu59]{Wu59} \emph{W.~T.~Wu.} On the isotopy of a finite complex in Euclidean space, I, II, Science Record, N.S. 3:8 (1959) 342--347, 348--351.}

\newcommand{\wusf}{\bibitem[Wu65]{Wu65} * \emph{W. T. Wu.} A Theory of Embedding, Immersion and Isotopy of Polytopes in an Euclidean Space. Peking: Science Press, 1965.}


\newcommand{\yann}{\bibitem[Ya99]{Ya99} \emph{Z. Yang.} Computing Equilibria and Fixed Points: The Solution of Nonlinear Inequalities, Kluwer, Springer Science + Business Media, 1990.}

\newcommand{\zesz}{\bibitem[Ze60]{Ze60} \emph{E. C. Zeeman}, Unknotting spheres in five dimensions, Bull. Amer. Math. Soc. 66 (1960) 198.
\linebreak
\url{https://www.ams.org/journals/bull/1960-66-03/S0002-9904-1960-10431-4/S0002-9904-1960-10431-4.pdf}}

\newcommand{\z}{\bibitem[Ze]{Z} * \emph{E. C. Zeeman}, A Brief History of Topology, UC Berkeley, October 27, 1993, On the occasion of Moe Hirsch's 60th birthday, \url{http://zakuski.utsa.edu/~gokhman/ecz/hirsch60.pdf}.}

\newcommand{\zioz}{\bibitem[Zi10]{Zi10} * \emph{D. \v Zivaljevi\'c}, Borromean and Brunnian Rings, \url{http://www.rade-zivaljevic.appspot.com/borromean.html}.}

\newcommand{\zioo}{\bibitem[Zi11]{Zi11} * \emph{G. M. Ziegler}, 3N Colored Points in a Plane, Notices of the Amer. Math. Soc., 58:4 (2011), 550-557.}


\newcommand{\zot}{\bibitem[Zi13]{Z13} \emph{A. Zimin.} Alternative proofs of the Conway-Gordon-Sachs Theorems, arXiv:1311.2882.}

\newcommand{\zss}{\bibitem[ZSS]{ZSS} * Элементы математики в задачах: через олимпиады и кружки к профессии. Сборник под редакцией А. Заславского, А. Скопенкова и М. Скопенкова. М.: МЦНМО, 2018.
Обновляемая версия части книги: \url{http://www.mccme.ru/circles/oim/materials/sturm.pdf}.}

\newcommand{\zu}{\bibitem[Zu]{Zu} \emph{J. Zung.} A non-general-position Parity Lemma,
\url{http://www.turgor.ru/lktg/2013/1/parity.pdf}.}


{\it Books, surveys and expository papers in this list are marked by the stars.}

\end{document}